\documentclass{amsart}
\usepackage{amssymb}
\usepackage[utf8]{inputenc}
\usepackage[dvipsnames]{xcolor}
\usepackage{tikz-cd}
\usepackage{fullpage}
\usepackage{mathtools}
\usepackage{stmaryrd}
\usepackage[shortlabels]{enumitem}
\usepackage{bbm}
\numberwithin{equation}{section}
\usepackage{graphicx}
\usepackage{array}
\usepackage[export]{adjustbox}
\usepackage{pgfplots}
\pgfplotsset{compat=1.13}
\usepackage{esvect}
\usepackage{hyperref}

\usetikzlibrary{calc,intersections,through,backgrounds}

\makeatletter
\newcommand{\xMapsto}[2][]{\ext@arrow 0599{\Mapstofill@}{#1}{#2}}
\def\Mapstofill@{\arrowfill@{\Mapstochar\Relbar}\Relbar\Rightarrow}
\makeatother

\newtheorem{thm}{Theorem}[section]
\newtheorem{thmx}{Theorem}

\newtheorem{prop}[thm]{Proposition}

\newtheorem{lem}[thm]{Lemma}

\newtheorem{lemma}[thm]{Lemma}
\newtheorem{cor}[thm]{Corollary}

\newtheorem{defprop}[thm]{Definition/Proposition}
\theoremstyle{definition}
\newtheorem{definition}[thm]{Definition}

\newtheorem{notation}[thm]{Notation}
\newtheorem{assumption}[thm]{Assumption}
\newtheorem{condition}[thm]{Condition}
\newtheorem{claim}[thm]{Claim}

\theoremstyle{remark}

\newtheorem{remark}[thm]{Remark}

\newtheorem{ex}[thm]{Example}

\newcommand{\trop}{\mathrm{trop}}
\newcommand{\tree}{\tau}
\newcommand{\TV}{Y}
\newcommand{\kap}{\Psi}
\newcommand{\facet}{\kappa}

\renewcommand{\bar}[1]{\overline{#1}}
\newcommand{\bfT}{\mathbf{T}}
\newcommand{\bbT}{\mathbb{T}}

\DeclareMathOperator{\Hom}{Hom}
\DeclareMathOperator{\len}{len}

\DeclareMathOperator{\Spec}{Spec}

\DeclareMathOperator{\edges}{Edges}

\DeclareMathOperator{\val}{val}
\DeclareMathOperator{\Trop}{Trop}
\DeclareMathOperator{\FW}{FW}
\DeclareMathOperator{\relint}{relint}

\newcommand{\Newt}{\mathrm{Newt}}

\DeclareMathOperator{\Gr}{Gr}
\DeclareMathOperator{\Plucker}{Pl}

\DeclareMathOperator{\length}{len}
\DeclareMathOperator{\EI}{EI}
\DeclareMathOperator{\mult}{mult}
\DeclareMathOperator{\init}{in}
\DeclareMathOperator{\Star}{Star}

\DeclareMathOperator{\id}{id}
\newcommand{\Z}{\mathbb{Z}}
\newcommand{\C}{\mathbb{C}}

\newcommand{\A}{\mathbb{A}}
\newcommand{\R}{\mathbb{R}}
\newcommand{\X}{\mathcal{X}}
\newcommand{\m}{\mathfrak{m}}
\renewcommand{\k}{\mathbf{k}}

\renewcommand{\P}{\mathbb{P}}
\newcommand{\M}{\mathcal{M}}
\newcommand{\TropInt}{\Sigma_r\cap\bigcap_{q=1}^r \Trop(H_q)}
\newcommand{\TropIntPDFString}{Sigma\_r \textbackslash cap (\textbackslash bigcap q=1 to r of Trop(H\_q))} % uses Unicode symbols U+2229 (?) and U+22C2 (?) for intersections
\newcommand{\rel}{\thickspace\mathrm{rel}\thickspace}

\newcommand{\abs}[1]{\left\lvert#1\right\rvert}

\newcommand{\Mbar}{\overline{\mathcal{M}}}
\newcommand{\into}{\hookrightarrow}
\newcommand{\K}{K}


\newcommand{\MbarZeroN}{\overline{\mathcal{M}}_{0,n}}

\title{Degenerations in tropical compactifications and tropical intersection theory of \texorpdfstring{$\protect\MbarZeroN$}{M\_{0,n}-bar}}

\author[Griffin]{Sean T. Griffin}
\address{Department of Mathematics, University of North Texas, Denton, TX, USA, 76203}
\email{sean.griffin@unt.edu}
\thanks{Griffin was supported by ERC grant ``Refined invariants in combinatorics, low-dimensional topology and geometry of moduli spaces'' No.~101001159.}

\author[Levinson]{Jake Levinson}
\address{
   D\'epartement de math\'ematiques et de statistique, Universit\'e de Montr\'eal, Montr\'eal, QC, Canada, H3T 1J4
}
\email{jake.levinson@umontreal.ca}
\author[Ramadas]{Rohini Ramadas}
\address{
   Rohini Ramadas, Department of Mathematics, Emory University, 400 Dowman Drive, Atlanta, GA, 30322
}
\thanks{Levinson was partially supported by NSERC grant DG RGPIN-2021-04169 ``Combinatorics and geometry of moduli spaces''.}
\email{rohini.ramadas@emory.edu}
\thanks{Ramadas was partially supported by EPSRC New Investigator Award (EP/X026612/1). 
}
\author[Silversmith]{Rob Silversmith}
\address{
   Rob Silversmith, Department of Mathematics, Emory University, 400 Dowman Drive, Atlanta, GA, 30322
}
\email{rob.silversmith@emory.edu}

\begin{document}

\begin{abstract}
The main result of this paper is a formula for the limit cycle of a 1-parameter family of subvarieties of a tropical compactification, expressed in terms of tropical intersections. Our theorem generalizes results of Dickenstein-Feichtner-Sturmfels and Katz to the case of tropical compactifications.

In the second part of the paper, we apply our formula to the moduli space $\overline{\mathcal{M}}_{0, n}$ of stable marked rational curves. We describe the tropicalization of the Kapranov maps $\overline{\mathcal{M}}_{0,n}\to \mathbb{P}^{n-3}$, whose hyperplane pullbacks are the $\psi$-classes, with respect to a suitable choice of torus. We introduce tropical $\psi$-hypersurfaces (in genus zero). These are different from the standard definition of Mikhalkin and Kerber-Markwig, and may be of independent interest. We demonstrate our main result by giving a ``firework algorithm'' that computes limits of intersections of $\psi$-hypersurfaces.
\end{abstract}

\maketitle

\section{Introduction}

This paper has two parts. In the first part, described in Section \ref{sec:TropicalResultSummary} of the introduction, we generalize theorems of Dickenstein-Feichtner-Sturmfels and Katz \cite{DickensteinFeichtnerSturmfels2007,Katz2009} for computing limit cycles in toric varieties via tropical intersection theory. We prove Theorem \ref{thm:intro1}, which computes limit cycles in \emph{tropical compactifications} as introduced by Tevelev \cite{Tevelev2007}. Our generalization is analogous to Osserman-Payne's generalization of tropical intersection theory \cite{OssermanPayne2013}, from intersections in tori to intersections in subvarieties of tori. 

In the second part, described in Section \ref{sec:ApplicationSummary} of the introduction, we introduce a notion of tropical $\psi$-hypersurfaces on the moduli space $\M_{0,n}^{\trop}$ of genus-zero tropical curves, which is different from, and in a sense `dual to', the standard definition of Mikhalkin and Kerber-Markwig \cite{Mikhalkin2006,KerberMarkwig2009}. We apply Theorem \ref{thm:intro1} to compute limit cycles of intersections of these hypersurfaces. We will see that in this class of computations, the underlying tropical geometry is intricate, while still being well-behaved enough to allow us to conclude all the geometric information we need from the tropical picture. In this case, the formula for the limit cycle is combinatorially described by a ``firework algorithm'' in $\M_{0,n}^{\trop},$ described in Section \ref{sec:IntersectionInM0nIntro}. (See also Figure \ref{fig:FireworkPicture}.)

\subsection{Background: Tropical intersection theory} \label{sec:IntroTropicalIntersectionTheory} Tropical intersection theory was developed in the early 2000s, with the goal of computing intersections in algebraic geometry (e.g. in the Chow ring of a variety) using the geometry of polyhedral complexes in $\R^n$ \cite{FultonSturmfels1997,Mikhalkin2006ICM,BogartJensenSpeyerSturmfelsThomas2007,GathmannKerberMarkwig2009,Katz2009,AllermanRau2010,OssermanPayne2013}. The theory has been subsequently applied in many examples, see e.g. \cite{Mikhalkin2005,KerberMarkwig2009,HuhKatz2012,BrugalleShaw2015,BakerLenMorrisonPfluegerRen2016,Gross2018,Goldner2021,Rau2023,CuetoMarkwig2023,PuentesMarkwigPauliRohrle2024}. To an $m$-dimensional subvariety $X\subseteq\bfT$ of an $n$-dimensional algebraic torus, one associates an $m$-dimensional weighted polyhedral complex $\Trop(X)\subseteq\R^n$ called its \emph{tropicalization}, and one can recover information about the intersections of subvarieties in $\bfT$ by intersecting their tropicalizations.
\begin{ex}\label{ex:BezoutExample}
    Consider the curves $X_1=V(x^2y^2+x^3+y^3+1)$ and $X_2=V(x^2y+xy^2+x+y)$ in $\bfT=(\C^*)^2.$ The tropicalizations $\Trop(X_1),\Trop(X_2)$, where each 1-dimensional polyhedron has weight 1 unless otherwise specified, are shown on the left in Figure \ref{fig:BezoutExample}. Taking a small translate of $\Trop(X_2)$ and intersecting yields the picture on the right. Here the intersection points are labeled with their \emph{tropical intersection multiplicities}, which are defined as indices of certain sublattices of the integer lattice in $\R^2$, times the weights of the intersecting cones, see \cite[Sec. 3.6]{MaclaganSturmfels2015}. By \cite[Thm. 3.6.1]{MaclaganSturmfels2015}, possibly after acting on $X_2$ by a general element of $\bfT$, $X_1\cap X_2$ consists of ten points counted with multiplicity.
    
    The limit of the above ``intersection-after-generic-perturbation'', as the perturbation goes to zero, is called the \emph{stable tropical intersection} of $\Trop(X_1)$ and $\Trop(X_2)$ --- in this case, 10 points concentrated at the origin. The fact that the stable intersection is independent of perturbation follows from the fact that $\Trop(X_1)$ and $\Trop(X_2)$ are \emph{balanced polyhedral complexes} (see Section \ref{sec:BackgroundPolyhedralGeometry}) with respect to their weightings --- that is, the weighted sum of the primitive integer vectors along the rays is zero.

    \begin{figure}
        \centering
        \begin{tikzpicture}[scale=1.5]
            \draw[red,very thick] (0,0)--(1,0);
            \draw[red] (1,0) node[right] {$3$};
            \draw[red,very thick] (0,0)--(0,1);
            \draw[red] (0,1) node[above] {$3$};
            \draw[red,very thick] (0,0)--(-1,-.5);
            \draw[red,very thick] (0,0)--(-.5,-1);
            \draw[blue,very thick,densely dashed] (-1,-1)--(1,1);
            \draw[blue,very thick,densely dashed] (-1,1)--(1,-1);
            \draw[red] (-1.5,-0.6) node {$\Trop(X_1)$};
            \draw[blue] (1.5,-0.9) node {$\Trop(X_2)$};
        \end{tikzpicture}
        \hspace{1in}
        \begin{tikzpicture}[scale=1.5]
            \draw[red,very thick] (0,0)--(1,0);
            \draw[red] (1,0) node[right] {$3$};
            \draw[red,very thick] (0,0)--(0,1);
            \draw[red] (0,1) node[above] {$3$};
            \draw[red,very thick] (0,0)--(-1,-.5);
            \draw[red,very thick] (0,0)--(-.5,-1);
            \draw[blue,very thick,densely dashed] (-1,-.6)--(.6,1);
            \draw[blue,very thick,densely dashed] (-1,.4)--(.4,-1);
            \draw (0,.4) node {$\bullet$} node[left] {$3$};
            \draw (-.2,-.4) node {$\bullet$} node[right] {$3$};
            \draw (-.4,-.2) node {$\bullet$} node[below] {$3$};
            \draw (-.8,-.4) node {$\bullet$} node[below] {$1$};
        \end{tikzpicture}
        \caption{Translating $\Trop(X_2)$ yields transverse intersection points, from Example \ref{ex:BezoutExample}.}
        \label{fig:BezoutExample}
    \end{figure}
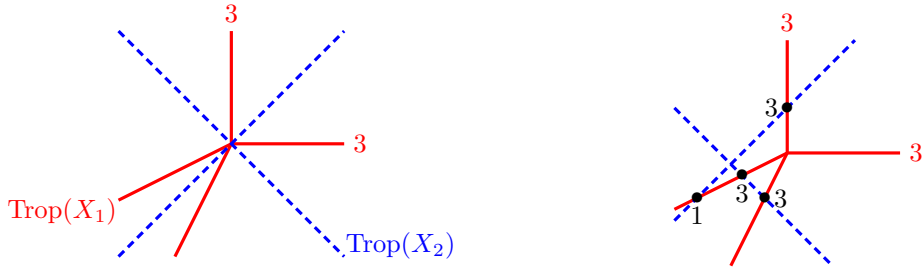
\end{ex}
As in usual intersection theory, notions of transversality play a large role in arguments like the above --- we would say that the intersection on the right in Figure \ref{fig:BezoutExample} is ``tropically transverse'', as every intersection point is locally the transverse intersection of affine-linear spaces.

More generally, Osserman-Payne showed that one can use tropical intersection theory to compute not only intersection products of subvarieties $X_1,X_2\subseteq\bfT$ in an ambient torus, but also to compute intersection products of subvarieties $X_1,X_2\subseteq Y\subseteq\bfT$ in an ambient closed subvariety $Y$ of a torus. To do so, one calculates the stable tropical intersection $\Trop(X_1)\cap\Trop(X_2)$ inside the ambient polyhedral complex $\Trop(Y)$. This result requires $\Trop(X_1) \cap \Trop(X_2)$ to be transverse ``in $\Trop(Y)$'', and in particular to lie in the locus where $\Trop(Y)$ is locally affine-linear and of weight $1$ (a tropical analog of the requirement in ordinary intersection theory that the ambient space be smooth).

\subsection{Background: Tropical intersection theory in the context of degenerations onto boundary strata} In \cite[Thm. 2.2]{DickensteinFeichtnerSturmfels2007}, Dickenstein-Feichtner-Sturmfels introduced a rather different way to use tropical intersection theory. In their setup, one wishes to compute the \emph{limit of a family} of $m$-dimensional subvarieties $X_t\subseteq\P^n,$ parametrized by some variable $t$, such that as $t\to0$, $X_t$ degenerates to a subscheme $X_0\subseteq\P^n$ supported on the union of the coordinate $m$-planes. As a cycle, $[X_0]$ is a sum of coordinate $m$-planes, each with a nonnegative integer coefficient, and \cite[Thm. 2.2]{DickensteinFeichtnerSturmfels2007} expresses these coefficients via tropical intersection theory; specifically, a coordinate $m$-plane $L_\sigma$ corresponds to a codimension-$m$ cone $\sigma$ in the fan $\Sigma$ of $\P^n$, and the coefficient of $[L_\sigma]$ in $[X_0]$ is the tropical intersection number of $\Trop(X_t)$ with the (complementary-dimensional) cone $\sigma$, computed using the same tropical intersection multiplicities as in Example \ref{ex:BezoutExample}. That is, $[X_0]$ is entirely determined by the way $\Trop(X_t)$ intersects the ``$m$-dimensional skeleton'' $\Sigma_m$ of the fan $\Sigma$.
\begin{ex}\label{ex:P2Case}
    Consider the family of curves $X_t=V(x_0^3+t^{-2}x_0x_1x_2+tx_0x_2^2+tx_1x_2^2+t^{2}x_2^3)\subseteq\P^2,$ and let $X_0$ denote the limit of the cycles $[X_t]$ as $t\to0.$ The tropicalization $\Trop(X_t)$ is the subset shown in red in Figure \ref{fig:P2Case} (where we have dehomogenized with respect to $x_2$, see Example \ref{ex:P2Newton}). The 1-dimensional cones of the fan $\Sigma$ of $\P^2$ are shown in black. Since $\Trop(X_t)$ intersects all three rays of $\Sigma$ with multiplicity 1, we conclude that $[X_0]=L_0+L_1+L_2$. Indeed, multiplying the equation of $X_t$ by $t^2$ and setting $t\to0$ yields $x_0x_1x_2.$

    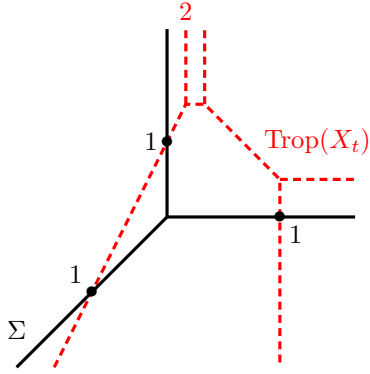
\begin{figure}
        \centering
        \begin{tikzpicture}[scale=.5]
            \draw[red,very thick,densely dashed] (-3,-4)--(.5,3)--(1,3)--(3,1)--(3,-4);
            \draw[red] (.5,5) node[above] {$2$};
            \draw[red,very thick,densely dashed] (.5,3)--(.5,5);
            \draw[red,very thick,densely dashed] (1,3)--(1,5);
            \draw[red,very thick,densely dashed] (3,1)--(5,1);
            \draw[very thick] (0,0)--(0,5);
            \draw[very thick] (0,0)--(5,0);
            \draw[very thick] (0,0)--(-4,-4);
            \draw (0,2) node {$\bullet$} node[left] {$1$};
            \draw (3,0) node {$\bullet$} node[below right] {$1$};
            \draw (-2,-2) node {$\bullet$} node[above left] {$1$};
            \draw[red] (4,2) node {$\Trop(X_t)$};
            \draw (-4,-3) node {$\Sigma$};
        \end{tikzpicture}
        \caption{Computing the cycle $[X_0]$ by intersecting $\Trop(X_t)$ with cones of the fan $\Sigma$ of $\P^2$, from Example \ref{ex:P2Case}.}
        \label{fig:P2Case}
    \end{figure}
\end{ex}
Katz \cite{Katz2009} generalized this idea to general toric varieties in the natural way. Let $X_t$ be a family of $m$-dimensional subschemes of a torus $\bfT$, and let $\bar X_t$ denote the closure in a toric variety $Y_\Sigma$ with dense torus $\bfT$. Assume that, as $t\to0$, $\bar X_t$ degenerates to a subscheme $(\bar X)_0$ supported on the union of $m$-dimensional torus orbit closures in $Y_\Sigma.$ Then Katz showed \cite[Thm. 10.1]{Katz2009} that in the decomposition of the limiting cycle $[(\bar X)_0]$ as a sum $[(\bar X)_0]=\sum_{\sigma}a_\sigma[Y_\sigma]$ of $m$-dimensional torus orbit closures, the coefficient $a_\sigma$ is the tropical intersection number of $\Trop(X_t)$ with the codimension-$m$ cone $\sigma.$

\subsection{Part \ref{part:MainTheorem}: Degenerations in tropical compactifications}\label{sec:TropicalResultSummary}
Our main theorem generalizes Katz's results to describe ``degenerating subvarieties of stratified varieties'', in which the subvariety degenerates into a union of boundary strata. We make this idea precise as follows.

Let $\bfT$ be an algebraic torus over an algebraically closed field $\k$ of characteristic zero, and let $\M\subseteq\bfT$ be a closed irreducible subvariety. Tevelev \cite{Tevelev2007} defined the notion of a \emph{tropical compactification} $\Mbar\subseteq Y_\Sigma$: the closure of $\M$ in a toric variety $Y_\Sigma\supseteq\bfT$ with certain special properties (see Section \ref{sec:BackgroundTropicalCompactifications}), including:
\begin{itemize}
    \item $\Mbar$ intersects torus orbits in $Y_\Sigma$ in the expected dimension, and 
    \item The fan $\Sigma$ of $Y_\Sigma$ has support equal to $\Trop(\M)$.
\end{itemize}
The tropical compactification $\Mbar$ inherits a stratification $\Mbar=\bigcup_{\sigma}\Mbar_\sigma$ from $Y_\Sigma$. We fix such a tropical compactification and further assume that for every cone $\sigma\subseteq\Sigma$, corresponding to a torus orbit $Y_\sigma,$ the closed stratum $\Mbar_\sigma:=\Mbar\cap \overline{Y_\sigma}$ is irreducible and generically reduced. 

The setup for our main theorem involves a ``degenerating family $X$ of subvarieties of $\M$". More precisely, let $K$ be a complete valued field extension\footnote{The reader who is not used to the formalism of valued fields may take $(K, R, \m)=(\C((t)), \C[[t]], (t))$. Our theorem is vacuous if $K$ is trivially valued.} of $\k$, with valuation ring $R$ and maximal ideal $\m$, such that $R/\m \cong \k$.
Let $X\subseteq\M\times_{\Spec(\k)}\Spec(K)$ be a closed subscheme of codimension $r$. The ``$t=0$ limit'' $(\bar X)_0$ described above is, in this language, obtained by taking the closure $\X$ of $X$ in $\Mbar\times_{\Spec(\k)}\Spec(R)$, and letting $$(\bar X)_0=\X\times_{\Spec(R)}\Spec(\k).$$ 

Our main theorem describes the decomposition of the cycle $[(\bar X)_0]$ as a sum of boundary strata, in terms of the intersection of $\Trop(X)$ with the $r$-skeleton $\Sigma_r$ of $\Sigma$, treating $\Sigma = \Trop(\M)$ as the ambient space. Crucially, the intersection of $\Trop(X)$ with $\Sigma_r$ is \textbf{not} in general transverse in $\Sigma$ in the sense of \cite{OssermanPayne2013}, because $\Sigma$ is in general not locally affine at a point of this intersection, which lies in $\Sigma_r$. Nonetheless, we introduce a new definition of tropical intersection multiplicity $\mult_{\Gamma}(\Trop(X),\Sigma_r;\Sigma)$ of $\Trop(X)$ with $\Sigma_r$, at an isolated point $\Gamma\in\Trop(X)\cap\Sigma_r$. Roughly, this definition involves perturbing $\Sigma_r$ only near the intersection point, and only into a nearby facet of $\Trop(\M)$; the step of perturbing only into $\Trop(\M)$ resembles \cite{OssermanPayne2013}. (See Definition \ref{def:MultIndependentOfChoice-Intro}/\ref{def:MultIndependentOfChoice}, and also Section \ref{sec:TropicalMultiplicityIntro} for an overview and discussion of the definition.) We prove: 
\begin{thmx}[Thm. \ref{thm:mainthm}]\label{thm:intro1}
    Suppose $\Trop(X)$ intersects $\Sigma_r$  in finitely many points, each of them in the relative interior of an $r$-dimensional cone. Then the cycle $[(\bar X)_0]$ is supported on the union of codimension-$r$ boundary strata of $\Mbar$, and the coefficient of a stratum $[\Mbar_{\sigma}]$ (corresponding to an $r$-dimensional cone $\sigma\subseteq\Sigma$) in $[(\bar X)_0]$ is equal to \begin{align}\label{eq:MultSumIntro}
    \sum_{\Gamma\in\Trop(X)\cap\sigma}\mult_{\Gamma}(\Trop(X),\Sigma_r;\Sigma).
    \end{align}
\end{thmx}
\begin{remark}
    A striking --- and useful --- feature of Theorem \ref{thm:intro1} is that it does not require much information about the geometry of $\Trop(X)$. Rather, the multiplicities $\mult_{\Gamma}(\Trop(X),\Sigma_r;\Sigma)$, and hence the limit cycle $[(\bar{X})_0]$, are determined completely by the local geometry of $\Trop(X)$ near the finite set $\Trop(X)\cap\Sigma_r$. Indeed, in Part \ref{part:Application}, we use Theorem \ref{thm:intro1} to compute $[(\bar{X})_0]$ in a situation where we know very little about the global geometry of $\Trop(X).$
\end{remark}

If $\M=\bfT$, then $\mult_{\Gamma}(\Trop(X),\Sigma_r;\Sigma)$ coincides with the usual tropical intersection multiplicity of $\Trop(X)$ and $\Sigma_r$ at $\Gamma$, and Theorem \ref{thm:intro1} specializes to Katz's result \cite[Thm. 10.1]{Katz2009} mentioned above. Our proof is similar in spirit to Katz's --- the main extra ingredient is a result of Sturmfels-Tevelev \cite[Thm. 1.2]{SturmfelsTevelev2008} on elimination theory for tropical varieties.

In the special case that $X\subseteq\M$ arises as a sufficiently transverse intersection of the form $X=\widetilde{X}\cap\M$, for a subvariety $\widetilde X\subseteq Y$ of a larger ambient space, we prove (see Theorem \ref{thm:XTilde} in Section \ref{sec:TropicalMultiplicityIntro}) that our ``new" tropical intersection multiplicity of $\Trop(X)$ and $\Sigma_r$ agrees with the usual tropical intersection multiplicity of $\Trop(\widetilde{X})$ and $\Sigma_r$ in $\Trop(Y)$.

\subsection{Part \ref{part:Application}: Application to \texorpdfstring{$\Mbar_{0,n}$}{M\_{0,n}-bar}}\label{sec:ApplicationSummary}

Perhaps the most well-studied example of a tropical compactification is $\Mbar_{0,n},$ the moduli space of stable rational $n$-marked curves, whose tropicalization is the moduli space $\M_{0,n}^{\trop}$ of stable $n$-marked metric trees. Recall that $\Mbar_{0,n}$ has tautological \emph{cotangent classes} $\psi_1,\ldots,\psi_n\in A^1(\Mbar_{0,n})$, which are pullbacks of hyperplane classes along certain birational morphisms $\Psi_1,\ldots,\Psi_n:\Mbar_{0,n}\to\P^{n-3}$ known as the Kapranov maps. (See Section \ref{sec:Kapranov}.) 

\subsubsection{Tropical Kapranov maps} The target $\P^{n-3}$ of $\Psi_i$ does not have a natural choice of coordinates, and therefore does not have a natural choice of dense torus. However, the choice of an auxiliary marking $j\in[n]\setminus \{i\}$ determines natural coordinates on $\P^{n-3}$ and therefore specifies a dense torus. Having made this choice, the morphism $\Psi_i$ is a map from a subvariety $\Mbar_{0,n}$ of a toric variety to a toric variety $\P^{n-3}$. We show that this morphism is toric, and describe its tropicalization explicitly as follows. Let $\mathbb{R}^{n-2}$ have coordinates indexed by $[n] \setminus \{i, j\}$ and let $\R^{n-2}/\R$ be the quotient by the all-ones vector $\mathbbm{1} \in \R^{n-2}$.

\begin{definition}[cf. Definition \ref{def:TropicalPsi}]\label{def:TropicalPsiIntro}
    The \emph{$i$-th tropical Kapranov map relative to $j$} is
    \[\Psi_{i \rel j}^\trop : \M_{0,n}^\trop\to \R^{n-2}/\R,\] defined as follows. Let $\Gamma \in \M_{0, n}^\trop$ be a stable $n$-marked metric tree. Choose any isometry between $\R$ and the convex hull of legs $i$ and $j$ in $\Gamma$, oriented from $i$ to $j$. For each $\ell \in [n] \setminus \{i, j\}$, let $d_\ell$ be the coordinate of the nearest point to the $\ell$-th leg in this convex hull. Then
    \begin{equation}\label{eqn:PsiCoordsIntro}
    \Psi_{i \rel j}^\trop(\Gamma) := (d_\ell : \ell \in [n] \setminus \{i, j\}),
    \end{equation}
    considered up to translation by $\mathbbm{1}$, which accounts for the choice of isometry.
    \end{definition}

\begin{ex}\label{ex:TropicalKapranov} 
    We can think of $\kap_{i \rel j}^{\trop}(\Gamma)$ as the result of pulling the legs $i$ and $j$ into a straight-line path, then measuring how far along this path each other marked point falls:
    \begin{center}
    \begin{tikzpicture}
    \node at (0, -1) {$\Gamma =$};
    \begin{scope}[shift = {(2, 0)}]
        \draw (0, 0)
            -- node[above left, midway] {$a$} (-0.7, -0.7) node {$\bullet$};
        \draw  (-1.1, -0.95) node[below left] {$4$} 
            -- (-0.7, -0.7) 
            -- (-1.1, -0.45) node[above left] {$8$};
        \draw (-0.35, -2.35) node [below left] {$5$}
            -- (0, -2)
            -- (0, -2.5) node [below] {$6$};
        \draw (1, -0.7)
            -- node[above right, midway] {$d$} (1.7, -1.5) node {$\bullet$}
            -- node[above right, midway] {$e$} (2, -2) node {$\bullet$}
            -- (1.75, -2.4) node[below left] {$7$};
        \draw (2, -2) node {$\bullet$}
            -- (2.25, -2.4) node[below right] {$9$};
        \draw (1.7, -1.5)
            -- (1.45, -1.9) node[below left] {$3$};
        \draw [red] (0, 0.5) node[above] {$1$}
            -- (0, 0) node {$\bullet$}
            -- node[above right, midway] {$b$} (1, -0.7) node {$\bullet$}
            -- node[above left, midway] {$c$} (0, -2) node {$\bullet$}
            -- (-0.5, -2) node [left] {$2$};
    \end{scope}
    
    \draw[->,decorate, decoration={snake,amplitude=.4mm,segment length=2mm,post length=1mm},thick] (5,-1)--(6.25,-1) node [pos=0.4, above=1mm] {$\kap_{1 \rel 2}^\trop$};
        
    \begin{scope}[shift={(8, -1.65)}]
        \draw[dashed] (0.5, 0)
            -- node [left, midway] {$a$} (0.5, 1) node {$\bullet$};
        \draw (0.25, 1.4) node [above left] {$4$}
            -- (0.5, 1)
            -- (0.75, 1.4) node [above right] {$8$};
        \draw[dashed] (1.7, 0)
            -- node [left, midway] {$d$} (1.7, 1.05) node {$\bullet$}
            -- node [left, midway] {$e$} (1.7, 1.65) node {$\bullet$};
        \draw (1.7, 1.05)
            -- (2.15, 1.05) node[right] {$3$};
        \draw (1.45, 2.05) node [above left] {$7$}
            -- (1.7, 1.65)
            -- (1.95, 2.05) node [above right] {$9$};
        \draw (3.1, 0.4) node [above left] {$5$}
            -- (3.35, 0)
            -- (3.6, 0.4) node [above right] {$6$};
        \node at (1.925, -1) {$\kap_{1 \rel 2}^\trop(\Gamma) = (b, 0, b+c, b+c, b, 0, b)$,};
        \draw [red] (-1, 0) node [left] {$1$}
            -- (0.5, 0) node {$\bullet$}
            -- node [below, midway] {$b$} (1.7, 0) node {$\bullet$}
            -- node [below, midway] {$c\phantom{b}$} (3.35, 0) node {$\bullet$}
            -- (4.85, 0) node [right] {$2$};
        \node [below, red] at (-0.7, 0) {$-\infty$};
        \node [below, red] at (4.5, 0) {$+\infty$};
    \end{scope}
    \end{tikzpicture}
    \end{center}
    where we have chosen to normalize so that the attachment point of leg $1$ is at $0$. Note that the edges outside the convex hull of the legs $1$ and $2$, shown as dashed lines, can be ignored or even contracted.
\end{ex}

    Recall also that for $S\subseteq[n]$ with $\abs{S}\ge3$, there is a forgetful morphism $\pi_S:\Mbar_{0,n}\to\Mbar_{0,\abs{S}}$ that forgets all marked points not in $S$. (See Section \ref{sec:M0nBar}.) There is a tropical forgetful map $\pi_S^\trop : \M_{0, n}^\trop \to \M_{0, |S|}^\trop$, where $\pi_S(\Gamma)$ is the convex hull in $\Gamma$ of the legs marked by elements of $S$. It is straightforward to describe the tropicalization of the composition $\Psi_i\circ\pi_S:\Mbar_{0,n}\to \P^{\abs{S}-3}$. 

\subsubsection{Tropical \texorpdfstring{$\psi$}{psi}-hypersurfaces} \label{subsec:tropicalpsiintro} We next use the tropical Kapranov map to define tropical $\psi$-classes, in a different way from the existing definition of tropical $\psi$-classes due to Mikhalkin and Kerber--Markwig \cite{Mikhalkin2006, KerberMarkwig2009}. Specifically, we describe certain tropical hypersurfaces $H^{\trop}=\Trop(H)\subset\M_{0,n}^{\trop}$, where $H\subset\M_{0,n}$ is a hypersurface whose closure $\overline{H}$ in $\Mbar_{0,n}$ represents $\psi_i$. The hypersurfaces $H$ are defined over a non-trivially valued field and represent degenerating hypersurfaces in $\Mbar_{0, n}$; as such, $H^{\trop}$ is not a sub-fan of $\M_{0,n}^{\trop}$, in contrast to the Mikhalkin/Kerber-Markwig tropical $\psi$-class. Our hypersurfaces $H$ are pulled back from sufficiently general hyperplanes $\underline{H}$ in the target $\P^{n-3}$ of $\Psi_i$. The specific condition on $\underline{H}$ (following \cite{GillespieGriffinLevinson2022}) is that the coefficients of its defining equation should all have different valuations, which ensures that it is  transversal to the exceptional locus of $\Psi_i$. Again, we generalize all of this to describe tropical hypersurfaces that represent pullbacks of $\psi$-classes along forgetful maps.  

\subsubsection{Intersecting tropical \texorpdfstring{$\psi$}{psi}-hypersurfaces}\label{sec:IntersectionInM0nIntro} Given $m\le n-3$ subsets $S_q\subset [n]$ and $i_q\in S_q$, we choose hypersurfaces $H_q$ representing $\pi_{S_q}^*(\psi_{i_q})$ as in Section \ref{subsec:tropicalpsiintro}. Roughly speaking, the valuations of the coefficients of the defining equations of $\underline{H}_q$ decrease successively in order of magnitude as $q$ ranges from $1$ to $m$. These choices follow \cite{GillespieGriffinLevinson2022}, which does not use tropical techniques, but this approach is common in tropical intersection theory, see for example \cite{Goldner2020Generalizing}. As we describe, these choices ensure that the hypersurfaces $H_q^{\trop}$ (and $H_q$) intersect sufficiently nicely. We prove

\begin{thmx}[Prop. \ref{prop:RelativeInteriors2}, Prop. \ref{prop:FWTropX}, Cor. \ref{cor:TropMultIs1}]\label{thm:introthmabouttropicalintersection}
    For $r\in\{1,\ldots, m\}$, the intersection $\Sigma_r\cap\Trop(\bigcap_{q=1}^r H_q)$ is finite and contained in the relative interior of $\Sigma_r$. At each point of the intersection, the tropical intersection multiplicity (in the sense of Theorem \ref{thm:intro1}) equals $1$. 
\end{thmx}

 For the proof of Theorem \ref{thm:introthmabouttropicalintersection}, we first describe $\TropInt$, as well as each $H_q^{\trop}$ locally near a point of $\TropInt$. This part (Prop. \ref{prop:RelativeInteriors2}) is purely tropical. We then use Osserman-Payne's results on lifting tropical intersections in order to prove (Prop. \ref{prop:intersectionandtropicalization}) that tropicalization and intersection commute in this setting; specifically that $\bigcap_{q=1}^r \Trop(H_q)$ agrees with $\Trop(\bigcap_{q=1}^r H_q)$, locally near $\Sigma_r$. 
 We also prove a version of Theorem \ref{thm:introthmabouttropicalintersection} that holds at the boundary of the extended tropical moduli space $\Mbar_{0,n}^{\trop}$. Using these two versions of Theorem \ref{thm:introthmabouttropicalintersection} and applying Theorem \ref{thm:intro1}, we obtain:

\begin{thmx}[Cor. \ref{cor:XbarHasCodimensionR}, Prop. \ref{prop:NoIrreducibleComponentsInBoundary}, Thm. \ref{thm:LimitCycleFW}, Cor. \ref{cor:InjectivityCorollary}]\label{thm:introapplypartonetoparttwo}\hfill
\begin{enumerate}
    \item The intersection $\bigcap_{q=1}^r \overline{H}_q$ is of pure codimension $r$ and has no irreducible components on the boundary of $\Mbar_{0,n}$. 
    \item In particular, $\overline{\bigcap_{q=1}^r H_q}$ represents the product $\prod_{q=1}^{r}\pi_{S_q}^*(\psi_{i_q})$. 
    \item The flat limit of $\overline{\bigcap_{q=1}^r H_q}$ is supported on a union of codimension-$r$ boundary strata, and its cycle is the corresponding sum of codimension-$r$ boundary strata, each with coefficient $1$. 
\end{enumerate}
   \end{thmx}

We prove that as $r$ varies, the collection of subsets $\Sigma_r\cap\Trop(\bigcap_{q=1}^r H_q)\subset\M_{0,n}^{\trop}$ satisfies a nice recursion: each point of $\Sigma_{r+1}\cap\Trop(\bigcap_{q=1}^{r+1} H_q)$ can be located via a ``nearby" point of $\Sigma_r\cap\Trop(\bigcap_{q=1}^r H_q)$.

\begin{thmx}[See Lem. \ref{lem:GammaInTropIntCombinatorics}, Rem. \ref{rmk:rounding-lengths}, Cor. \ref{cor:ContractingSendsTropIntToTropInt}, Prop. \ref{prop:tropintafteredgeinsertion}]\label{thm:introfirework}
    Let $\Gamma'\in \Sigma_{r+1}\cap\Trop(\bigcap_{q=1}^{r+1} H_q)$. Then $\Gamma'$ has exactly $r+1$ edges, whose lengths are ``all different orders of magnitude''.
    
    Moreover, there is a unique $\Gamma \in \Sigma_r\cap\Trop(\bigcap_{q=1}^r H_q)$, whose underlying tree is obtained by contracting the shortest edge $e$ of $\Gamma'$, and whose other edges differ in length from those of $\Gamma'$ by quantities of order of magnitude $\len(e)$. Conversely, given $\Gamma$, all nearby $\Gamma'$ can be generated by a combinatorial criterion.
\end{thmx}

We use this to give an algorithm that finds all the points of $\Sigma_r\cap\Trop(\bigcap_{q=1}^r H_q)$; see Section \ref{subsec:firework}. We name it the \emph{firework algorithm} because it produces the finite set $\FW_r= \Sigma_r\cap\Trop(\bigcap_{q=1}^r H_q)$ as the $r$-th step of a recursive procedure. At the $0$-th step, $\FW_0$ is the cone point of $\M_{0,n}^{\trop}$, a singleton set. At step $r=1$, we ``shoot out a large distance in many directions'' from the cone point, along various rays of $\M_{0,n}^{\trop}$, giving a finite subset $\FW_1$. At the $(r+1)$-st step, each point $\Gamma\in\FW_r$ is contained in an $r$-dimensional cone, and we ``shoot out from each $\Gamma$ in several directions'' into adjacent $(r+1)$-dimensional cones, traveling distances an order of magnitude smaller than those traveled at step $r$ (Figure \ref{fig:FireworkPicture}).

In fact, for all $r$, $\FW_r$ consists of at most 1 point in each cone of $\M_{0,n}^{\trop}$ (Proposition \ref{prop:Injective}).

\begin{figure}
\includegraphics[height=2.5in]{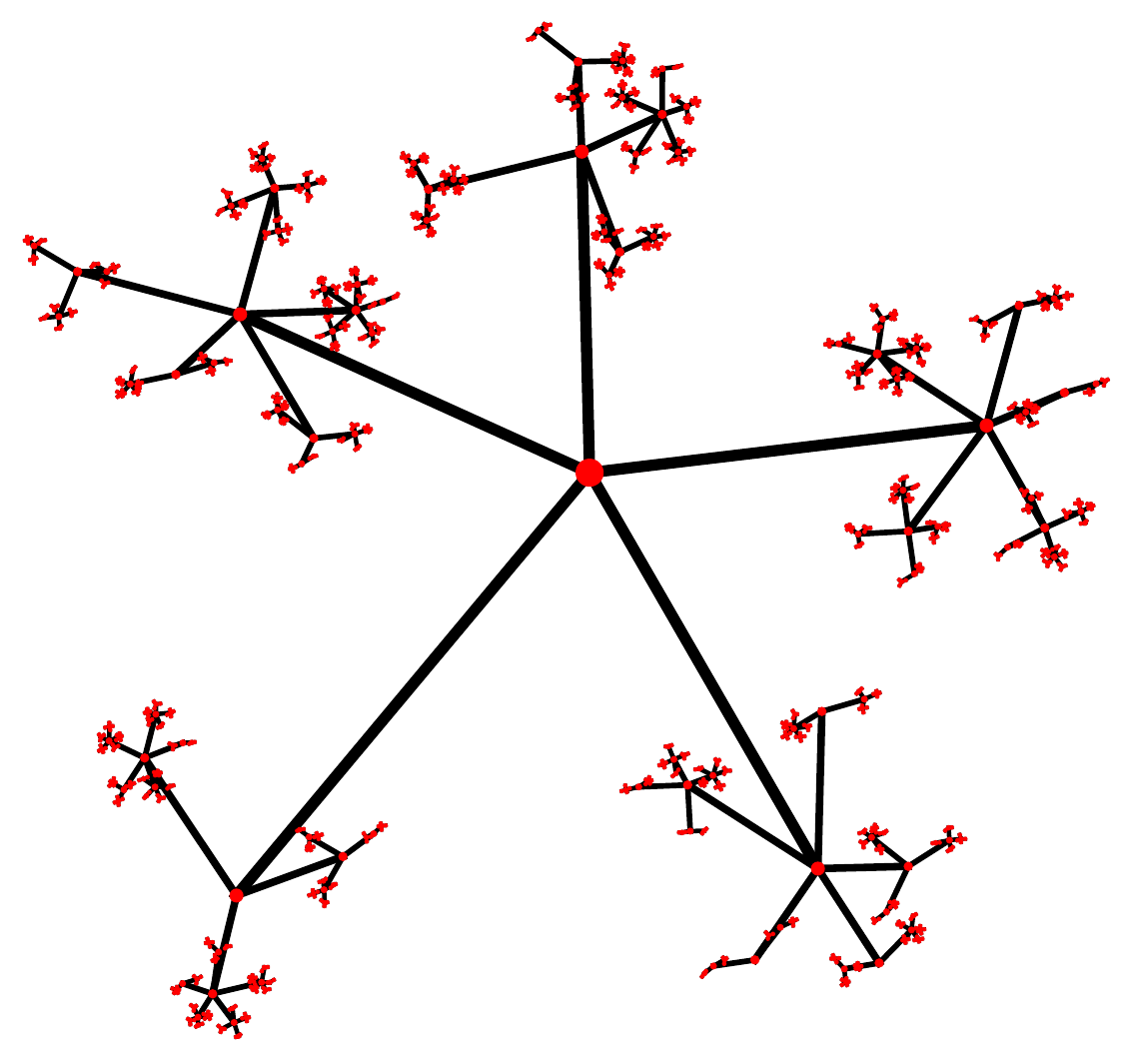}
    \caption{A schematic picture of a firework. The central point is the cone point of $\M_{0,n}^{\trop},$ the outgoing rays from this point are all along different rays of $\M_{0,n}^{\trop}$, the secondary offshoots are all in different 2-dimensional cones on $\M_{0,n}^{\trop},$ and so on.}
    \label{fig:FireworkPicture}
\end{figure}

\subsubsection{Relation of this paper to \cite{GillespieGriffinLevinson2022}} \label{sec:limits-and-intersections}
It is straightforward (using standard techniques for working with the Chow ring of $\Mbar_{0,n}$, see Section \ref{sec:KapranovDegreesIntro}) to express products of classes of the form $\pi_S^*(\psi_i)$ as sums of boundary strata of $\Mbar_{0,n}$ in various ways.  This calculation in Chow does not, however, imply that any particular such expression is ``geometric'' in the strong sense of arising as the cycle of a complete intersection of hypersurfaces, each representing one $\pi_S^*(\psi_i)$ factor. In \cite{GillespieGriffinLevinson2022}, Gillespie--Griffin--Levinson give expressions for products of $\omega$-classes (certain special classes of the form $\pi_S^*(\psi_i)$) as sums of boundary strata, which are geometric in the weaker sense of arising as cycles of \emph{limits} of complete intersections. Explicitly, they find degenerating $1$-parameter families of hypersurfaces in $\Mbar_{0,n}$, each representing one $\omega$ factor, such that the generic intersection is transverse, and such that the limit cycle of the $1$-parameter family of intersections is the corresponding sum of boundary strata. This paper arose as an attempt to make a precise connection between their results and tropical intersection theory. Specifically, we realized that 
\begin{enumerate}
\item the degenerating hypersurfaces chosen by \cite{GillespieGriffinLevinson2022} could be used to define tropical $\psi$-classes \textit{with non-trivial valuations}, as in Section \ref{sec:oldandnewpsi}, 
\item the choices made in \cite{GillespieGriffinLevinson2022} were essentially tropical in nature, and implied that $r$ such tropical hypersurfaces intersected ``nicely" \textit{along the $r$-skeleton $\Sigma_r$ of $\M_{0,n}^{\trop}$} (in the sense of Theorem \ref{thm:introthmabouttropicalintersection}),
\item the algorithm in \cite{GillespieGriffinLevinson2022} admitted an elegant metric/tropical augmentation that computed the finitely many intersection points of the $r$ tropical hypersurfaces with $\Sigma_r$, and 
\item everything generalized without issue from $\omega$-classes to all classes of the form $\pi_S^*(\psi_i)$. 
\end{enumerate}
Carrying out this calculation led us to develop the general tropical technique of
Part \ref{part:MainTheorem} (Theorem \ref{thm:intro1}) to efficiently prove algebraic results analogous to those in \cite{GillespieGriffinLevinson2022}, in particular to conclude Theorem \ref{thm:introapplypartonetoparttwo} from Theorem \ref{thm:introthmabouttropicalintersection}.

\subsection{Relationship between the ``new'' tropical \texorpdfstring{$\psi$}{psi}-hypersurfaces and the standard tropical \texorpdfstring{$\psi$}{psi}-classes} \label{sec:oldandnewpsi}

Early in the history of tropical geometry, Mikhalkin and Kerber-Markwig defined the tropical $\psi$-class $\psi_i^{\trop}$, as a codimension-$1$ balanced sub-fan of $\M_{0,n}^{\trop}$. Implicit in this definition is the fact that $\psi_i^{\trop}$ is actually the tropicalization of \textbf{any} sufficiently general hypersurface in $\M_{0,n}$, defined over a \textbf{trivially valued field}, whose closure represents $\psi_i$. This $\psi_i^{\trop}$ encodes the class $\psi_i$ in terms of all its intersection numbers:
the weight assigned by $\psi_i^{\trop}$ to a codimension-$1$ cone of $\M_{0,n}^{\trop}$ equals the intersection number of $\psi_i$ with the corresponding $1$-dimensional boundary stratum of $\Mbar_{0,n}$. There are well-known formulas that express the class $\psi_i$ as a positive sum of boundary divisors, but these expressions are not encoded by the Mikhalkin/Kerber-Markwig $\psi_i^{\trop}$. 

In contrast, each of our tropical $\psi$-hypersurfaces $H^{\trop}$ is the tropicalization of a sufficiently general hypersurface $H$ in $\M_{0,n}$, defined over a \textbf{non-trivially valued field}, whose closure in $\Mbar_{0,n}$ represents $\psi_i$. Any such sufficiently general hypersurface $H$ has flat limit some union of boundary divisors --- exactly which union depends on the valuations of the coefficients of the defining equation. Correspondingly, the tropical hypersurface $H^{\trop}$ depends on the valuations of the coefficients of the defining equation of $H$. The intersection $H^{\trop}\cap\Sigma_1$ encodes the union of boundary divisors that $H$ limits to; it accordingly recovers the known expressions for $\psi_i$ as an effective sum of boundary divisors.

The original tropical $\psi$-classes do not intersect tranversely, being all subfans of $\M_{0,n}^{\trop}$. 
Kerber and Markwig used tropical intersection theory to compute the stable intersection of $r$ tropical  $\psi$-classes as a balanced codimension-$r$ weighted subfan of $\M_{0,n}^{\trop}$ --- again, the weight assigned to a given codimension-$r$ cone equals the intersection number of the corresponding product of algebraic $\psi$-classes with the corresponding $r$-dimensional boundary stratum. They proved in particular that the resulting $0$-dimensional tropical intersections corroborated the well-known formula
\[\int_{\Mbar_{0,n}}\psi_1^{a_1}\cdots\psi_n^{a_n}=\binom{n-3}{a_1,\ldots,a_n}.\]
The stable intersection of $r$ tropical $\psi$-classes does not in any direct way express the corresponding product of algebraic $\psi$-classes as a sum of codimension-$r$ boundary strata. In contrast to Kerber-Markwig's setup, our tropical $\psi$-hypersurfaces, when sufficiently well-chosen, quite likely do intersect transversely, although we do not describe the $r$-fold tropical intersection as a whole. What we do is prove that the intersection of the $r$ hypersurfaces is transverse \emph{along the $r$-skeleton of $\M_{0,n}^{\trop}$}, and we compute all the (finitely many) points of intersection of the $r$ tropical hypersurfaces with the $r$-skeleton. This information encodes the limit cycle of the $r$-fold algebraic intersection and expresses it as an effective sum of boundary strata. (We note that the resulting expression for the product of $\psi$-classes in cohomology can be deduced more or less directly from Keel's presentation for the Chow ring of $\Mbar_{0, n}$; its content, particularly from a combinatorial perspective, is not new. It is noteworthy mainly in that it is ``geometric'' in the sense of Section \ref{sec:limits-and-intersections}.)

\textbf{We see therefore that these two approaches --- over trivially versus non-trivially valued fields --- encode ``dual" information.}

\begin{remark}
     Our tropical $\psi$-hypersurfaces are specific to the setting of genus $0$. There is recent work of Cavalieri--Gross--Markwig \cite{CavalieriGrossMarkwig2023tropical} that defines tropical $\psi$-classes in arbitrary genus, and work of Cavalieri--Gross \cite{cavalieri2024tropicalization} that shows that under certain conditions, the tropical $\psi$-classes arise as tropicalizations of algebraic $\psi$-classes. 
\end{remark}

\subsection{Intersection theory of \texorpdfstring{$\Mbar_{0,n}$}{M\_{0,n}-bar}, \texorpdfstring{$\psi$}{psi}-classes, and combinatorics}\label{sec:KapranovDegreesIntro} The intersection theory of $\Mbar_{0,n}$ is well-studied, and is an important ingredient in many enumerative problems involving rational curves, see e.g. \cite{KockVainsencher}. Keel and Kontsevich-Manin gave presentations \cite{Keel1992,KontsevichManin1994} for the Chow ring $A^*(\Mbar_{0,n})$, as a result of which any intersection product is computable via certain basic combinatorial rules. Nonetheless, $A^*(\Mbar_{0,n})$ has an extraordinary combinatorial richness, and decades later, its combinatorics are still of significant interest. The specific case of intersection numbers of pullbacks of $\psi$-classes, or \emph{Kapranov degrees}, have appeared in many recent works \cite{GalletGraseggerSchicho2020,CavalieriGillespieMonin2021,Goldner2021,Silversmith2022,GillespieGriffinLevinson2022,GillespieGriffinLevinson2023,Silversmith2024,ReinkeSilversmith2024,BrakensiekEurLarsonLi2023}, with connections to many well-known combinatorial objects (rigid graphs, parking functions, matching theory of bipartite graphs, triangulations of polygons, mixed volumes, chromatic polynomials of graphs).

A byproduct of Part \ref{part:Application} is an algorithm for computing intersection products of pullbacks of $\psi$-classes. As noted above, however, this computation in Chow/cohomology is not the point of the present paper.

\subsection{More about the tropical intersection multiplicity} \label{sec:TropicalMultiplicityIntro} We now describe the new notion of tropical intersection multiplicity $\mult_{\Gamma}(\Trop(X),\Sigma_r;\Sigma)$ required in Theorem \ref{thm:intro1}.

Recall the setup: $\Gamma$ is an isolated point of $\Trop(X)\cap\Sigma_r$, which lies in the relative interior of $\Sigma_r$, hence in a unique maximal cone $\sigma$ of $\Sigma_r$. 

\begin{definition}[Def. \ref{def:MultIndependentOfChoice}] \label{def:MultIndependentOfChoice-Intro}
Choose any maximal cone $\zeta$ of $\Sigma$ whose closure contains $\Gamma$ and, in a small neighborhood $U$ of $\Gamma$, choose a generic translation $\tilde\sigma$ of $\sigma$ into $\zeta$. (See Figure \ref{fig:LocallyPerturbSigma}.) We define
\[
\mult_{\Gamma}(\Trop(X),\Sigma_r;\Sigma) = \frac{1}{m_\zeta} \cdot i(\Gamma,\Trop(X)\cdot\tilde\sigma;\zeta),
\]
where the right-hand side is the tropical intersection number of $\tilde\sigma$ with $\Trop(X)$, computed inside ambient space $U\cap\zeta$ using the usual notion of tropical intersection multiplicity (and divided by the weight $m_\zeta$). For the purposes of this computation, $\tilde\sigma$ is assigned weight 1.
\end{definition}
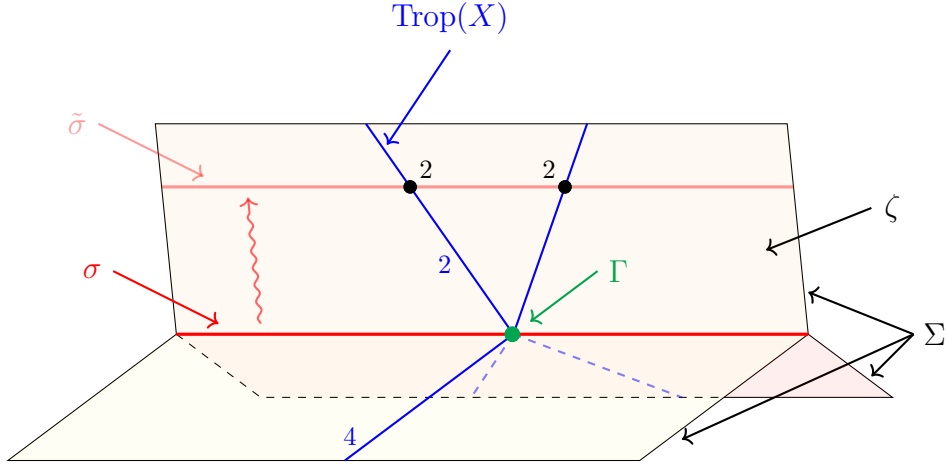
\begin{figure}
        \centering
        \begin{tikzpicture}[scale=2.8]
            \draw (0,0)--++(-.1,1)--++(3,0)--++(.1,-1);
            \draw (0,0)--++(-.8,-.6)--++(3,0)--++(.8,.6);
            \draw[dashed] (0,0)--++(.4,-.3)--++(2.2,0);
            \draw (2.6,-.3)--++(.8,0)--++(-.4,.3);
            \draw (3.6,0) node {\Large $\Sigma$};
            \draw[thick,->] (3.5,0)--(3.3,-.2);
            \draw[thick,->] (3.5,0)--(3,.2);
            \draw[thick,->] (3.5,0)--(2.4,-.5);
            \draw (3.4,0.6) node {\Large $\zeta$};
            \draw[thick,->] (3.3,0.6)--(2.8,.4);

            \filldraw[orange,opacity=0.05] (0,0)--++(-.1,1)--++(3,0)--++(.1,-1)--cycle;
            \filldraw[yellow,opacity=0.04] (0,0)--++(-.8,-.6)--++(3,0)--++(.8,.6)--cycle;
            \filldraw[red,opacity=0.03] (0,0)--++(.4,-.3)--++(2.2,0)--(3,0)--cycle;
            \filldraw[red,opacity=0.07] (3,0)--(2.6,-.3)--++(.8,0)--cycle;

            \draw[very thick,red] (0,0)--++(3,0);
            \draw[red] (-.4,0.3) node {\Large $\sigma$};
            \draw[red,thick,->] (-.3,.3)--(.2,.05);
            
            \draw[thick,blue] (1.6,0)--++(-.8,-.6);
            \draw[blue] (.9,-.4) node[below left] {$4$};
            \draw[thick,blue] (1.6,0)--++(-.7,1);
            \draw[blue] (1.35,.33) node[left] {$2$};
            \draw[thick,blue] (1.6,0)--++(.35,1);
            \draw[thick,blue,opacity=0.5,dashed] (1.6,0)--++(-.2,-.3);
            \draw[thick,blue,opacity=0.5,dashed] (1.6,0)--++(.8,-.3);
            \draw[blue] (1.3,1.5) node {\Large $\Trop(X)$};
            \draw[blue,thick,->] (1.3,1.35)--(1,.9);
            
            \filldraw[Green] (1.595,0) circle(0.035);
            \draw[Green] (2.1,.3) node {\Large $\Gamma$};
            \draw[thick,->,Green] (2,.3)--(1.68,.06);

            \draw[very thick,red,opacity=0.4] (-.07,0.7)--++(3,0);
            \filldraw (1.6,0)++(-.49,0.7) circle(0.03) node[above right] {$2$};
            \filldraw (1.6,0)++(.245,0.7) circle(0.03) node[above left] {$2$};
            \draw[red,opacity=0.4] (-.4,0.3)++(-.07,0.7) node {\Large $\tilde\sigma$};
            \draw[red,thick,->,opacity=0.4] (-.3,.3)++(-.07,0.7)--(.13,.75);
            \draw[red, decorate, decoration={snake,amplitude=.3mm,segment length=2.25mm,post length=1mm},thick, opacity=0.6,->] (0.4,0.05) --++ (-0.06, 0.6); 
        \end{tikzpicture}
        \caption{A schematic illustration of the multiplicity $\mult_{\Gamma}(\Trop(X),\Sigma_r;\Sigma)$ appearing in Theorem \ref{thm:intro1}. The point $\Gamma$ ({\color{ForestGreen}green}) is an isolated intersection point of $\Trop(X)$ ({\color{Blue}blue}) with the $1$-skeleton ($r=1$) of $\Sigma$, lying in the relative interior of the cone $\sigma$ ({\color{red}red}) of dimension $r=1$. The local geometry of $\Sigma$ near $\Gamma$ is shown. The multiplicity $\mult_{\Gamma}(\Trop(X),\Sigma_r;\Sigma)=4$ is found by translating $\sigma$ into a nearby maximal cone $\zeta$ of $\Sigma$, then computing the usual tropical intersection number of $\Trop(X)$ with the translate $\tilde\sigma$ ({\color{Salmon}pink}). The result is independent of the choices of $\zeta$ and $\tilde \sigma$. (Note that the multiplicities cannot be seen from the picture; they depend on certain lattice indices, cf. Section \ref{sec:BackgroundTropicalIntersectionTheory}.) 
        The map $\Sigma\to\Sigma/\sigma$ is the horizontal projection, which sends $\Trop(X)$ to $\Sigma/\sigma$ via a ``degree-4 tropical map.''}
        \label{fig:LocallyPerturbSigma}
    \end{figure}

It is nontrivial that $\mult_{\Gamma}(\Trop(X),\Sigma_r;\Sigma)$ is well-defined (see Theorem \ref{thm:MultIndependentOfChoice} and Definition/Proposition \ref{prop:GenericallyFinite}), as the definition involved making two choices: the maximal cone $\zeta,$ and the perturbation $\tilde\sigma.$ This phenomenon is of a flavor that is familiar in tropical intersection theory, and like all such phenomena, it ultimately boils down to the balancing condition for tropical varieties. (See Example \ref{ex:BezoutExample}.) More directly, the reader should compare to the theorem of Sturmfels-Tevelev that the degree of a ``tropical generically finite map'' $\Trop(X_1)\to\Trop(X_2)$ is the cardinality (counted with multiplicity) of the preimage of a general point of \emph{any} polyhedron of $\Trop(X_2)$ \cite[Thm 3.12]{SturmfelsTevelev2008}. Indeed, we use this theorem to prove that $\mult_{\Gamma}(\Trop(X),\Sigma_r;\Sigma)$ is well-defined (Theorem \ref{thm:MultIndependentOfChoice}). To see roughly why it applies, observe that in Figure \ref{fig:LocallyPerturbSigma}, the quotient map $\Sigma\to\Sigma/\sigma$ appears to send $\Trop(X)$ to $\Sigma/\sigma$ as a finite-to-one cover, whose degree is the same over each polyhedron of $\Sigma/\sigma$.

\begin{remark}
    It is worth highlighting again the contrast between what we are doing and the tropical intersection theory of Osserman-Payne \cite{OssermanPayne2013}. There, they consider a subvariety $Y\subseteq T$ of a torus and two subvarieties $X,X'\subseteq Y$. They define a tropical intersection multiplicity of $X$ and $X'$ in $Y$, \emph{provided $X\cap X'$ lies generically in the locus where $Y$ is locally affine-linear and of weight $1$}. We have defined a stable intersection of $\Trop(X)$ and $\Sigma_r$, provided their intersection is finite and lies in the relative interior of $\Sigma_r$, even if the ambient variety $\Trop(\M)=\Sigma$ is \emph{not} locally affine-linear along the intersection --- indeed $\Sigma$ is essentially never locally affine-linear along $\Sigma_r$ in our application. It would be interesting to know whether this idea has other applications/generalizations beyond Theorem \ref{thm:intro1}.
\end{remark}

\begin{ex}\label{ex:Degeneration}
    Let $\M=V(x+y+z+1)\subseteq\bfT=(\C^*)^3$ with $Y_\Sigma=\P^3\setminus Z$ where $Z$ is the set of four coordinate points, which are the $\bfT$-fixed points under the action scaling the first three homogeneous coordinates. Thus, $\Mbar$ is isomorphic to a copy of $\P^2$ sitting inside of $Y_\Sigma$. It can be checked that $\Mbar$ is a tropical compactification of $\M$, where $\Trop(\M)=|\Sigma|$ is the $2$-skeleton of the fan in $\R^3$ associated to the $\bfT$-action on $\P^3$, and each facet of $\Trop(\M)$ has weight $1$.

    Let $X = V(x^2yz+t,x+y+z+1)\subseteq \M\times_{\Spec(\C)}\Spec(K)$ where $K = \C\{\!\{t\}\!\}$. Then we see that the class of the ``$t=0$ limit'' is $[X_0] = 2[\Mbar_{\sigma_1}] + [\Mbar_{\sigma_2}] + [\Mbar_{\sigma_3}]\in A^\bullet(\Mbar)$, where $\sigma_i = \R_{>0} \langle e_i\rangle$ for $i=1,2,3$. Indeed, let us check that Theorem~\ref{thm:intro1} agrees with this computation.

    A simple calculation (using tropical transversality for example) shows that $\Trop(X)$ is the intersection of $\Trop(x+y+z+1)$ with the plane $2x+y+z=1$ in $\R^3$, see Figure~\ref{fig:Degeneration Example}.
    Then $\Trop(X)$ intersects $\Sigma_1$ at $e_1/2$, $e_2$, and $e_3$ which are in the relative interior of their respective cones. 

    \begin{figure}
    \centering
    \begin{tikzpicture}[scale=2]
            \pgfmathsetmacro{\xconst}{.5}
            \pgfmathsetmacro{\yconst}{sqrt(1-\xconst^2)}
            \pgfmathsetmacro{\tscale}{0.2}
            \pgfmathsetmacro{\tscaletwo}{1.6}
            \filldraw[yellow,opacity=0.1] (0,0)--(-3,0)--(0,2)--cycle;
            \filldraw[fill={rgb:red,4;green,1;yellow,1},opacity=0.12] (0,0)--(-3,0)--({2*\xconst},{-2*\yconst})--cycle;
            \filldraw[orange,opacity=0.2] (0,0)--(0,2)--({2*\xconst},{-2*\yconst})--cycle;
            \draw[->] (0,0)--(0,2);
            \draw[->] (0,0)--({2*\xconst},{-2*\yconst});
            \draw[->] (0,0)--(-3,0);
            \draw[blue, thick] (-0.5,0)--(0,1) node[pos=0.25,right] {$r_2$};
            \draw[blue, thick] (\xconst,-\yconst)--(0,1);
            \draw[blue, thick] (-0.5,0)--(\xconst,-\yconst) node[pos=0.15, below] {$r_1$};
            \draw[blue,dashed,opacity=.3, thick] (0,1)--({\tscale*(1-\xconst)},{\tscale*\yconst+(3*\tscale+1)});
            \draw[blue,dashed,opacity=.3, thick] (\xconst,-\yconst)--({\tscale+\xconst*(3*\tscale+1)},{-\yconst*(3*\tscale+1)-\tscale});
            \draw[blue, dashed,opacity=0.3, thick] (-0.5,0)--({-.5-0.6*\tscaletwo-0.6*\tscaletwo*\xconst},{0.6*\tscaletwo*\yconst-0.6*\tscaletwo}) node[pos=0.25, below] {$r_3$};
            \draw (0,1) node {$\bullet$} node[left] {$e_3$};
            \draw (\xconst,-\yconst) node {$\bullet$} node[right] {$e_2$};
            \draw (-1,0) node {$\bullet$} node[above] {$e_1$};
            \draw[name path=pathr3, blue] (-0.5,0) node {$\bullet$} node[above,xshift=-2] {$\Gamma$};
            \draw[blue,opacity=.3] ({-1.5-\xconst},{\yconst-1}) node {$\bullet$} node[below] {$\Gamma'$};
            \coordinate (Gamma) at (-.5,0);
            \coordinate (endr3) at ({-.5-\tscaletwo-\tscaletwo*\xconst},{\tscaletwo*\yconst-\tscaletwo});
            \coordinate (xaxisend) at (-3,0);
            \coordinate (yaxisend) at ({2*\xconst},{-2*\yconst});
            \coordinate (xyint) at (intersection of  Gamma--endr3 and xaxisend--yaxisend);
            \coordinate (diagtip) at ({3*(1-\xconst)},{3*(\yconst-1)-.5});
            \filldraw[gray,opacity=0.1] (0,0)--(-3,0)--(diagtip)--cycle;
            \filldraw[gray,opacity=0.1] (0,0)--(yaxisend)--(diagtip)--cycle;
            \filldraw[gray,opacity=0.1] (0,0)--(0,2)--(diagtip)--cycle;
        \end{tikzpicture}
    \caption{The tropicalization $\Trop(X) \subseteq \mathbb{R}^3$ of the variety in Example~\ref{ex:Degeneration}.}
    \label{fig:Degeneration Example}
    \end{figure}
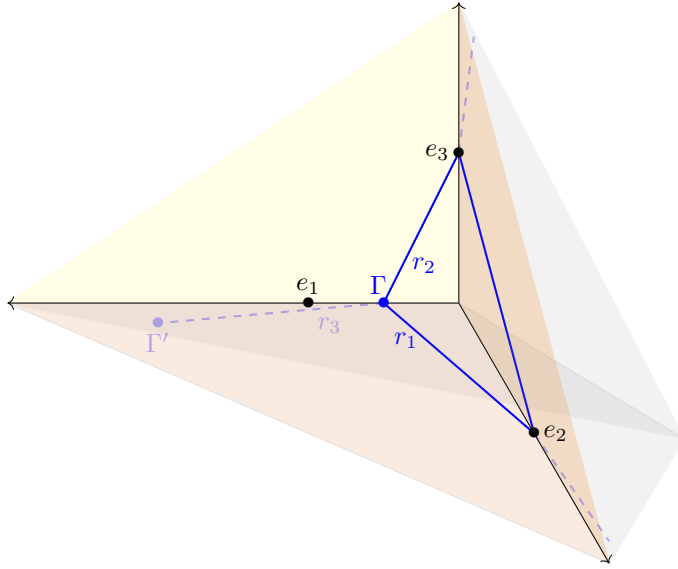

    Let $\Gamma = e_1/2$, and let $r_1,r_2,r_3$ be the $1$-dimensional cells of $\Trop(X)$ whose closures contain $\Gamma$ and are parallel to $\langle -1,2,0\rangle$, $\langle -1,0,2\rangle$, and $\langle 1,-1,-1\rangle$, respectively. It is easy to check $r_1$ and $r_2$ have weight $1$. For the ray $r_3$, let us take $\Gamma'=(3/2,-1,-1)$ in its relative interior. Then $\init_{\Gamma'}(I) = \langle x^2yz+1,y+z\rangle = \langle x^2y^2-1,y+z\rangle\subseteq \C[x^\pm,y^\pm,z^\pm]$, whose vanishing set is a union of two reduced copies of $\C^*$. Hence, the weight of $r_3$ is $2$.

    Observe that the unique cone of $\Sigma$ containing $\Gamma$ is $\sigma_1$. Generically translating $\sigma_1$ to $\tilde\sigma_1$ in the facet $\zeta=\R_{>0}\langle e_1,e_2\rangle$, then $\tilde\sigma_1$ intersects $r_2$ at a unique point. Thus, the formula \eqref{eq:MultSumIntro} predicts the coefficient of $[\Mbar_{\sigma_1}]$ to be
    \[
    \frac{1}{m_\zeta}m_{r_1}[N_\zeta:L_\Z(\langle -1,2,0\rangle) + L_\Z(\langle 1,0,0\rangle)] =  \frac{1}{1}\cdot 1 \cdot 2 = 2,
    \]
    which indeed matches our prediction.

    Alternatively, we may generically translate $\sigma_1$ to $\tilde\sigma_1'$ in the facet $\zeta' = \R_{>0}\langle e_1,-e_1-e_2-e_3\rangle$ of $\Sigma$. In this case, $\tilde\sigma_1'$ intersects $r_3$ in a unique point, and the formula \eqref{eq:MultSumIntro} predicts the coefficient of $[\Mbar_{\sigma_1}]$ to be 
    \[
    \frac{1}{m_\zeta}m_{r_3}[N_{\zeta'}:L_\Z(\langle 1,-1,-1\rangle) + L_\Z(\langle 1,0,0\rangle)] = \frac{1}{1}\cdot 2\cdot 1 = 2,
    \]
    which again matches our prediction.
\end{ex}

Finally, a special case of Theorem \ref{thm:intro1} is when $X$ is of the form $\widetilde X\cap\M$, for some subscheme $\widetilde X$ of a larger ambient space $Y$ (such as $\bfT$), such that $\Trop(\widetilde{X})$ meets $\Sigma_r$ transversely in $\Trop(Y)$. In this case, using \cite{OssermanPayne2013}, we prove that $\mult_{\Gamma}(\Trop(X),\Sigma_r;\Sigma)$ coincides with a tropical intersection multiplicity in the sense of Osserman-Payne:
\begin{thmx}[Thm. \ref{thm:XTilde2-OP}]\label{thm:XTilde}
    Suppose $\M$ is Cohen-Macaulay and contained in a subvariety $Y \subseteq \bfT$ of dimension $n'$. Let $\widetilde X\subseteq Y^K$ be a Cohen-Macaulay subscheme, and let $X = \widetilde X \cap \M$ be the scheme-theoretic intersection.
    
    Let $\Gamma$ be an isolated point of $\Trop(\widetilde X)\cap \Sigma_r$ that is moreover in the relative interior of:
    \begin{enumerate}
        \item a facet of $\Trop(Y)$ of weight $1$,
        \item a facet $\sigma$ of $\Sigma_r$, and
        \item an $(n'-r)$-dimensional facet of $\Trop(\widetilde X)$,
    \end{enumerate}
    Then $\Gamma$ is an isolated point of $\Trop(X) \cap \sigma$ and \begin{align*}
        \mult_{\Gamma}(\Trop(X),\sigma;\Sigma) = i(\Gamma, \Trop(\widetilde{X}) \cdot\sigma;\Trop(Y)),
    \end{align*}
    where on the right-hand side, $\sigma$  is taken to have weight 1 (as in Definition \ref{def:MultIndependentOfChoice-Intro}), and $i(\Gamma, \Trop(\widetilde{X}) \cdot\sigma;\Trop(Y))$ is the local tropical intersection multiplicity of \cite{OssermanPayne2013}.
\end{thmx}

See Remark \ref{rem:discussion-of-thm-E} for additional discussion.

\subsection{Outline of the paper} Part \ref{part:MainTheorem} consists of Section \ref{sec:BackgroundTropical}, which contains background material on tropical geometry, Section \ref{sec:StratificationOfTropicalCompactification}, which contains some results on tropical compactifications, and Section \ref{sec:DegenerationsInTropicalCompactifications}, which contains the proof of Theorems \ref{thm:intro1} and \ref{thm:XTilde}. Part \ref{part:Application} begins with Section \ref{sec:BackgroundModuli}, which contains background material on the moduli spaces $\Mbar_{0,n}$ and $\Mbar_{0,n}^{\trop}$ of (algebraic and tropical) marked genus-0 stable curves. In Section \ref{sec:TropicalKapranov} we define tropical Kapranov maps and tropical $\psi$-hypersurfaces. In Section \ref{sec:TropIntisnice}, we introduce a collection of $\psi$-hypersurfaces $H_q$, and prove Theorem \ref{thm:introthmabouttropicalintersection}. In Section \ref{sec:RelateToAlgebraic}, we apply the results of Part \ref{part:MainTheorem} to prove Theorem \ref{thm:introapplypartonetoparttwo}. In Section \ref{sec:RecursionandFirework}, we develop the recursive Firework algorithm that computes $\Sigma_r\cap\Trop(\bigcap_{q=1}^r H_q)$.

\subsection{Acknowledgements}The authors are grateful to Renzo Cavalieri for the initial observation that the results of \cite{GillespieGriffinLevinson2022} seemed to be tropical in nature, and to Diane Maclagan, Hannah Markwig, Sam Payne, and Bal\'azs Szendr\H{o}i for enlightening discussions. Ramadas and Silversmith thank SLMath for their hospitality during Spring 2022, when this project was initiated. Griffin, Ramadas, and Silversmith thank ICERM for their hospitality in Summer 2023, when part of this work was conducted.

\part{Tropical intersection theory and degenerating families of subvarieties}\label{part:MainTheorem}
\section{Background 1: Tropical geometry}\label{sec:BackgroundTropical}
Unless otherwise specified, see \cite{MaclaganSturmfels2015} for further details on everything in this section.

\subsection{Polyhedral geometry}\label{sec:BackgroundPolyhedralGeometry}
We review some of the basic terminology in polyhedral geometry. (See \cite[Sec. 2.3]{MaclaganSturmfels2015}.) A \textbf{polyhedron} $\sigma$ in $\R^n$ is a (possibly unbounded) subset that is the intersection of finitely many closed half-spaces. A polyhedron has \textbf{faces} of various dimensions. The \textbf{affine span} of $\sigma$ is the smallest affine subspace of $\R^n$ containing $\sigma$ --- it is a translate of the \textbf{linear space parallel to $\sigma$}, which we denote $L_{\R}(\sigma)$. The \textbf{dimension} $\dim(\sigma)$ of $\sigma$ is the dimension of $L_{\R}(\sigma)$ as a real vector space. The \textbf{relative interior} $\relint(\sigma)$ of $\sigma$ is the interior of $\sigma$ inside its affine span. A \textbf{polyhedral complex} in $\R^n$ is a set $\Sigma$ of polyhedra in $\R^n$ that are ``glued along faces''; precisely, a set such that (1) $\Sigma$ is closed under taking faces, and (2) if $\sigma_1,\sigma_2\in\Sigma$ are polyhedra, then $\sigma_1\cap\sigma_2$ is empty or is a face of both $\sigma_1$ and $\sigma_2$. The elements $\sigma\in\Sigma$ are called \textbf{cells} (or $d$-cells, where $d=\dim(\sigma)$), and the maximal cells with respect to inclusion are called \textbf{facets}. $\Sigma$ is called \textbf{pure dimension} $m$ if all facets have dimension $m$, and in this case we write $\dim(\Sigma)=m$. The \textbf{$r$-skeleton} of $\Sigma$ is the polyhedral complex $\Sigma_r=\{\sigma\in\Sigma:\dim(\sigma)\le r\}.$ The union $\abs{\Sigma}\subseteq\R^n$ of all cells $\sigma\in\Sigma$ is called the \textbf{support} of $\Sigma.$ For polyhedral complexes $\Sigma_1,\Sigma_2,$ we say $\Sigma_1$ \textbf{refines} $\Sigma_2$ if $\abs{\Sigma_1}=\abs{\Sigma_2}$ and every cell $\sigma\in\Sigma_2$ is a union of cells of $\Sigma_1.$ The union of the relative interiors of all facets of $\Sigma$ is called the \textbf{relative interior} $\relint(\Sigma)$ of $\Sigma$. The \textbf{lineality space} of $\Sigma$ is the largest linear subspace of $\R^n$ with respect to which $\Sigma$ is translation-invariant. If all cells of $\Sigma$ are \textbf{cones} (i.e. stable under multiplication by $\R_{\ge0}$), then $\Sigma$ is called a \textbf{fan}. 

All polyhedral complexes $\Sigma$ in this paper will be \textbf{rational}, i.e. for every cell $\sigma\in\Sigma,$ the linear space $L_{\R}(\sigma)$ is the span of integer vectors. (Note that the affine span of $\sigma$ need not be a translate of $L_{\R}(\sigma)$ by a vector with rational entries.) In this case, each cell $\sigma$ has an \textbf{associated lattice} $L_{\Z}(\sigma)=L_{\R}(\sigma)\cap\Z^n$.

Our polyhedral complexes will also be \textbf{weighted}, i.e. have a positive integer weight $m_{\Sigma;\zeta}$ assigned to each facet $\zeta\in\Sigma$. For a weighted rational polyhedral complex $\Sigma$, there is a notion called \textbf{balancedness}, see \cite[Def. 3.3.1]{MaclaganSturmfels2015}. 
A rational fan $\Sigma$ pure of dimension $d$  is \textbf{balanced} if for every cone $\tau$ of dimension $d-1$, $\sum_{\sigma \supsetneq \tau} m_\sigma v_\sigma = 0$ where $v_\sigma$ is the primitive vector of the ray $(\sigma + L_\mathbb{Z}(\tau))/L_\mathbb{Z}(\tau)$. A rational polyhedral complex $\Sigma$ pure of dimension $d$ is balanced if its star at every point is a balanced fan. (See just below for the definition of the star of $\Sigma$ at a point.)
In the simplest case, suppose $\Sigma$ is pure 1-dimensional, let $\sigma\in\Sigma$ be a 0-cell. Then $\Sigma$ is balanced at $\sigma$ if $$\sum_{\zeta}m_\zeta\cdot v_\zeta=0\in\R^n,$$ where $\zeta$ runs over 1-cells in $\Sigma$ that contain $\sigma$ --- i.e. rays based at $\sigma$ --- and $v_\zeta$ is the generator of $L_{\Z}(\zeta)$ pointing outward from $\sigma$ along $\zeta$.

\begin{ex}
     The reader may verify that $\Trop(X)$ in Figure \ref{fig:P2Case} is balanced, noting that the leftmost ray has slope 2.
\end{ex}
\begin{remark}\label{rem:RelativeInteriorOfSupport}
    We will often encounter subsets $W\subseteq\R^n$ that are the support of a non-unique rational polyhedral complex. In this case, by a slight abuse of notation, we define the \textbf{relative interior} $\relint(W)$ to be the set of points $\Gamma\in W$ such that there exists a neighborhood $U\subseteq\R^n$ of $\Gamma$ and an affine subspace $\mathbf{A}\subseteq\R^n$ such that $\mathbf{A}\cap U=W\cap U$, i.e. points where $W$ is locally linear. If $\Sigma$ is a polyhedral complex with support $W$, and $\Gamma\in\relint(W)$ is in the intersection of maximal cones $\sigma_1,\sigma_2$, then we must have $L_{\R}(\sigma_1)=L_{\R}(\sigma_2)$ and $L_{\Z}(\sigma_1)=L_{\Z}(\sigma_2)$. Thus $\Gamma$ defines a \textbf{linear space} $L_{\R}(W;\Gamma):=L_{\R}(\sigma_1)$, and an \textbf{associated lattice} $L_{\Z}(W;\Gamma):=L_{\Z}(\sigma_1)$.
\end{remark}
\begin{remark}\label{rem:RelativeInteriorOfSupportWeights}
    If $\Sigma$ is a balanced weighted rational polyhedral complex and $\Gamma\in\relint(\abs{\Sigma})$ (see Remark \ref{rem:RelativeInteriorOfSupport}), then balancing implies that each maximal cone of $\Sigma$ containing $\Gamma$ has the same weight. Thus we may think of a weighting as a locally constant function $\relint(\abs{\Sigma})\to\Z_{>0}.$ We may recover the weight $m_{\Sigma;\zeta}$ of a maximal cell $\zeta\in\Sigma$ as the weight assigned to any point in the interior of $\zeta.$ The subsets $W$ in Remark \ref{rem:RelativeInteriorOfSupport} will typically have a natural weight function $\relint(W)\to\Z_{>0}$, such that for any choice of polyhedral complex structure on $W$, the associated weighted rational polyhedral complex is balanced. We may therefore refer to the \textbf{weight} $m_{W;\Gamma}$ of any \emph{point} $\Gamma\in\relint(W)$.
\end{remark}

If $\Sigma$ is a (balanced weighted) polyhedral complex in $\R^n$, and $\Gamma\in\abs{\Sigma}$ is a point in the support of $\Sigma$, the \textbf{star} of $\Sigma$ at $\Gamma$ is the (balanced weighted) fan $$\Star_\Gamma\Sigma:=\{\R_{>0}\cdot(\sigma-\Gamma):\sigma\in\Sigma,\sigma\ni\Gamma\}.$$ That is, $\Star_\Gamma\Sigma$ is obtained by translating $\Sigma$ so that $\Gamma$ is at the origin, and taking the positive spans of all cells $\sigma\in\Sigma$ containing $\Gamma.$ We write $\Star_\Gamma\sigma:=\R_{>0}\cdot(\sigma-\Gamma)$ for these positive spans. This is the natural notion of the ``local picture near a point'' for polyhedral complexes. The lineality space of $\Star_\Gamma\Sigma$ contains $L_{\R}(\sigma)$ for $\sigma\in\Sigma$ the unique cell containing $\Gamma$ in its relative interior. In particular, if $\Gamma\in\relint(\zeta)$ for $\zeta\in\Sigma$ a facet, then $\Star_\Gamma\Sigma=L_{\R}(\zeta)$.

Remarks \ref{rem:RelativeInteriorOfSupport} and \ref{rem:RelativeInteriorOfSupportWeights} are compatible with taking stars, as follows. If $W\subseteq\R^n$ is the support of a rational polyhedral complex, as in Remark \ref{rem:RelativeInteriorOfSupport}, and $\Gamma\in W,$ then there is a well-defined subset $\Star_\Gamma W\subseteq\R^n$ that is the support of a rational polyhedral fan. Furthermore, a balanced weighting function $\relint(W) \to \Z_{>0}$ as in Remark \ref{rem:RelativeInteriorOfSupport} induces a balanced weighting function on $\Star_\Gamma(W)$ for any choice of fan structure.

\subsection{Basic tropical geometry} \label{sec:BackgroundTropicalBasics}
\begin{notation}
    In this section, we use the following notation.
    \begin{itemize}
    \item Let $K$ be a complete valued field (possibly with the trivial valuation) with valuation $\val : K^* \to \R$. Let $\alpha\mapsto t^{\alpha}$ be a splitting $\val(K^*)\to K^*$ of the valuation map. Let $R$ be the valuation ring of $K$, with maximal ideal $\m$ and residue field $\k=R/\m$. We assume that $\mathrm{char}(\k)=0$.
    \item Let $\bfT^R$ be an $n$-dimensional torus over $R$.
    Let $\bfT^K$ be the associated torus over $K$ and $\bfT$ the associated torus over $\k$.

    \item We have a canonical identification of the character and cocharacter lattices of $\bfT^K$ and $\bfT$; we write $M$ for the character lattice and $N$ for the cocharacter lattice. Let $N_{\R}=N\otimes\R=\R^n.$ 
    \end{itemize}
    Recall that taking cocharacter lattices gives an equivalence of categories from algebraic tori over an algebraically closed field to finite-rank free $\Z$-modules; in particular, there is a bijection between rational subspaces of $N_{\R}$ and subtori of $\bfT^K$ (or $\bfT$).
\end{notation}

\begin{definition}
    Let $\Gamma\in N_\R$. Given $f\in K[x_1^{\pm},\dots, x_n^{\pm}]$ with expansion $f = \sum_{u\in \mathbb{Z}^n} c_u x^u$, its \textbf{initial form} is
    \[
    \init_\Gamma(f) \coloneqq \sum_{u:\val(c_u) + \Gamma\cdot u= m} \overline{t^{-\val(c_u)}c_u}x^u\in\k[x_1^{\pm},\dots, x_n^{\pm}],
    \]
    where $m= \min\{\val(c_u) + \Gamma\cdot u\}$ and $\overline{\cdot}$ denotes taking the image in the residue field $\k$.
    
    The \textbf{initial ideal} of an ideal $I\subseteq K[x_1^{\pm},\dots, x_n^{\pm}]$ is $\init_\Gamma(I) \coloneqq \langle \init_\Gamma(f): f\in I\rangle\subseteq \k[x_1^\pm,\dots, x_n^\pm]$. The \textbf{initial subscheme} of $X$ is $\init_\Gamma X \coloneqq V(\init_\Gamma I(X))$, a subscheme of $\bfT$.
\end{definition}

\begin{definition}
    Let $X\subseteq\bfT^K$ be a (closed) subvariety of pure dimension $m$. Then the \textbf{tropicalization} of $X$, denoted by $\Trop(X)$, is a subset of $N_{\R}=\R^n$ with the following equivalent characterizations (\cite[Thm.~3.2.3]{MaclaganSturmfels2015}):
\begin{enumerate}
    \item $\Trop(X)$ is the set of vectors $\Gamma\in N_\R$ such that $$\init_{\Gamma}I(X)\ne\langle1\rangle=\k[x_1^{\pm},\ldots,x_n^{\pm}],$$
    or equivalently $\init_\Gamma X\ne\emptyset.$\label{it:char1}
    \item $\Trop(X)$ is the set of vectors in $N_\R$ arising as $$\Trop(y)=(\val(y_1),\ldots,\val(y_n))$$ for some $\widetilde K$-valued point $y=(y_1,\ldots,y_n):\Spec \widetilde K\to X$, where $\widetilde K$ is a valued extension of $K$ whose valuation extends the valuation on $K$.
\end{enumerate}
\end{definition}

Note that tropicalizations are stable under base extension.

\begin{remark}\label{rem:RescueDegeneration}
    It is helpful to keep in mind the following geometric intuition behind Characterization \eqref{it:char1}. Take $K=\C((t))$ with the $t$-adic valuation, so that $\Spec(K)$ is the ``formal punctured disk''. A vector $\Gamma\in N_\R$ corresponds (at least, if its entries are integers) to a cocharacter (i.e. 1-parameter subgroup) $\Spec(K[x^{\pm}])\to\bfT^K.$ We also have a map $\Spec(K)\to\Spec(K[x^{\pm}])$ given by $x\mapsto t,$ and composing gives a $K$-valued point $t^\Gamma:\Spec(K)\to\bfT^K$, i.e. a ``family of torus elements over the formal punctured disk''. Consider the translation $t^{-\Gamma}\cdot X$ of $X$ by $t^{-\Gamma}$. Then the initial subscheme $\init_\Gamma X$ can be characterized as the ``limit as $t\to0$ of $t^{-\Gamma}\cdot X$'' --- precisely, $\init_\Gamma X=(\overline{t^{-\Gamma}\cdot X})\times_{\Spec(R)}\Spec(\k)$, where $\overline{t^{-\Gamma}\cdot X}$ denotes closure in $\bfT^R.$ So $\Trop(X)$ contains $\Gamma$ if and only if this limit is nonempty.
\end{remark}
\begin{ex}
    We illustrate Remark \ref{rem:RescueDegeneration}. Consider $X_1=V(x^2y^2+x^3+y^3+1)$ from Example \ref{ex:BezoutExample}. We can see that $(-1,-2)\in\Trop(X_1)\subseteq\R^2$, because translating $X_1$ along the 1-parameter subgroup $(t^1,t^2)\in(\C^*)^2$ and taking the limit as $t\to0$ yields $$\lim_{t\to0}V(t^{-6}x^2y^2+t^{-3}x^3+t^{-6}y^3+1)=\lim_{t\to0}V(x^2y^2+t^3x^3+y^3+t^6)=V(x^2y^2+y^3)=V(x^2+y)\subseteq(\C^*)^2,$$ which is nonempty, as expected.
\end{ex}

The ``structure theorem for tropical varieties'' \cite[Thm. 3.3.5]{MaclaganSturmfels2015} asserts that $\Trop(X)$ is the support of a rational polyhedral complex of pure dimension $m$. Furthermore, there is a natural weight function $\relint(\Trop(X))\to\Z_{>0}$ (see Remark \ref{rem:RelativeInteriorOfSupportWeights}), with respect to which any rational polyhedral complex structure on $\Trop(X)$ is balanced.

\begin{definition}[{\cite[Remark 3.4.4]{MaclaganSturmfels2015}}]
    Given $\Gamma\in \relint(\Trop(X))$, $V(\init_\Gamma(I(X)))\subseteq\bfT$ is a subscheme supported on a disjoint union of copies of subtori. Then the \textbf{weight} of $\Gamma$ is the sum of the multiplicities of these subtori.
\end{definition}

We note two important special cases:
\begin{enumerate}
    \item If $K$ is trivially valued, then $\Trop(X)$ has the structure of a \emph{fan} \cite[Cor. 3.5.5]{MaclaganSturmfels2015}.
    \item If $X=V(f)\subseteq\bfT^K$ is a hypersurface, with defining equation $f=\sum_{\mathbf u}c_{\mathbf{u}}\mathbf{x}^{\mathbf{u}}\in K[x_1^{\pm},\ldots,x_n^{\pm}]$, the tropicalization $\Trop(X)$ is combinatorially equivalent to the (negative of the) dual polyhedral complex to the regular subdivision of the Newton polygon $\Newt(f)$ defined by the height function $\mathbf u\mapsto\val(c_{\mathbf{u}})$. (See \cite[Def. 2.3.8]{MaclaganSturmfels2015} and \cite[Prop. 3.1.6]{MaclaganSturmfels2015}.) In this case, the maximal polyhedra of $\Trop(X)$ correspond to edges of the regular subdivision, and the weight of a maximal cone is equal to the lattice length of the corresponding edge. We illustrate further with Example \ref{ex:P2Case}.
\end{enumerate}

\begin{ex}\label{ex:P2Newton}
    We illustrate the calculation of weights of maximal cones of a tropical hypersurface. In Example \ref{ex:P2Case}, the polynomial $$f=x_0^3+t^{-2}x_0x_1x_2+tx_0x_2^2+tx_1x_2^2+t^{2}x_2^3,$$ 
    after dehomogenizing with respect to $x_2$, has Newton polygon $\Newt(f)$ shown on the left in Figure \ref{fig:P2Newton}, where we have labeled each lattice point in $\Newt(f)$ with the valuation of its coefficient in $f$. Viewing these valuations as a height function on the lattice points, we get a collection of points in $\R^3$, and we take the ``upper convex hull'' $P$, i.e. the set of points in or above the convex hull (Figure \ref{fig:P2Newton}, center). Note that 
    excluding the point $(2,0,+\infty)$ does not affect the \emph{upper} convex hull. $P$ is an unbounded polyhedron, whose ``lower faces'' (i.e. those whose outer normal vector has negative last coordinate) project down to define a subdivision of $\Newt(f)$ (Figure \ref{fig:P2Newton}, right) --- this is the regular subdivision induced by the given height function. We have drawn the dual complex (weighted by lattice length of edges) of the subdivision in red. Note that after taking the negative (i.e. $180^\circ$ rotation), and up to translation and edge lengths, this complex is equal to $\Trop(X)$ as drawn in Figure \ref{fig:P2Case}.

    \begin{figure}
        \centering
        \raisebox{-25pt}{\begin{tikzpicture}
            \draw (0,0) node {$\bullet$} node[below] {$2$};
            \draw (1,0) node {$\bullet$} node[below] {$1$};
            \draw (2,0) node {$\bullet$} node[below] {$+\infty$};
            \draw (3,0) node {$\bullet$} node[below] {$0$};
            \draw (0,1) node {$\bullet$} node[above] {$1$};
            \draw (1,1) node {$\bullet$} node[above] {$-2$};
            \draw (0,0)--(3,0)--(1,1)--(0,1)--cycle;
        \end{tikzpicture}}
        \quad\quad$\leadsto$\quad\quad
        \raisebox{-45pt}{\begin{tikzpicture}
            \filldraw[fill=black!5,thick] (0,2)--(1.5,0)--(.5,2.1)--cycle;
            \filldraw[fill=black!20,thick] (0,2)--(1,1.5)--(1.5,0)--cycle;
            \filldraw[fill=black!10,thick] (3,1)--(1,1.5)--(1.5,0)--cycle;
            \draw[very thin] (0,2)--(0,4);
            \draw[very thin] (1,1.5)--(1,4);
            \draw[very thin] (3,1)--(3,4);
            \draw[very thin] (.5,2.1)--(.5,4);
            \draw[very thin] (1.5,1.375)--(1.5,4);
            \draw[dashed] (1.5,0)--(1.5,1.375);
            \draw (0,2) node {$\bullet$} node[below] {$2$};
            \draw (1,1.5) node {$\bullet$} node[left] {$1$};
            \draw (3,1) node {$\bullet$} node[below] {$0$};
            \draw (.5,2.1) node {$\bullet$} node[above] {$1$};
            \draw (1.5,0) node {$\bullet$} node[below] {$-2$};
            \draw[dashed] (.5,2.1)--(1.5,0);
        \end{tikzpicture}}
        \quad\quad$\leadsto$\quad\quad
        \raisebox{-45pt}{\begin{tikzpicture}
            \draw (0,0) node {$\bullet$} node[below] {$2$};
            \draw (1,0) node {$\bullet$} node[below] {$1$};
            \draw (2,0) node {$\bullet$} node[below] {$+\infty$};
            \draw (3,0) node {$\bullet$} node[below] {$0$};
            \draw (0,1) node {$\bullet$} node[above] {$1$};
            \draw (1,1) node {$\bullet$} node[above] {$-2$};
            \draw (0,0)--(3,0)--(1,1)--(0,1)--cycle;
            \draw(0,0)--(1,1)--(1,0);
            \draw[red] (.25,2)--(.25,.75)--(.75,.25)--(1.25,.25)--(2.125,2);
            \draw[red] (.25,.75)--(-1,.75);
            \draw[red] (.75,.25)--(.75,-1);
            \draw[red] (1.25,.25)--(1.25,-1) node[below] {$2$};
        \end{tikzpicture}}
        \caption{Assigning each lattice point in $\Newt(f)$ a height according to the valuation of its coefficient, then projecting the lower faces of the upper convex hull to get a subdivision of $\Newt(f)$, as in Example \ref{ex:P2Newton}.}
        \label{fig:P2Newton}
    \end{figure}
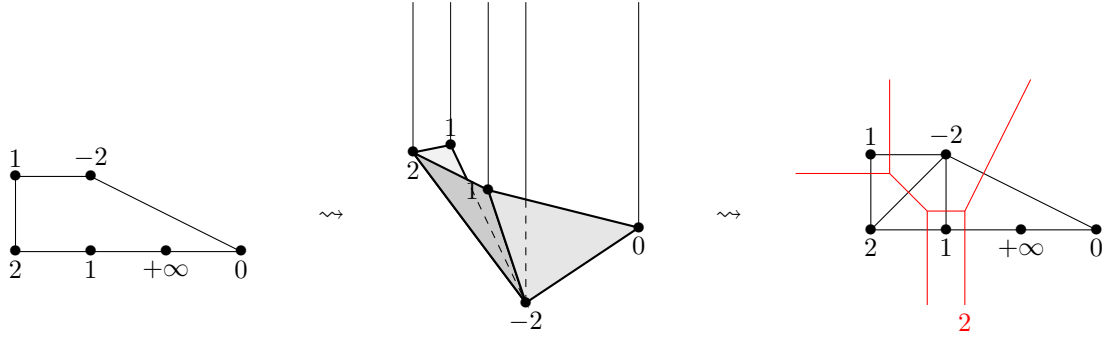
\end{ex}

We will also need the following standard facts:
\begin{lem}[{\cite[Lem. 3.3.6]{MaclaganSturmfels2015}}]\label{lem:InitialTropicalization}
    For any $\Gamma\in \Trop(X)$ we have $\Trop(\init_\Gamma X)=\Star_\Gamma\Trop(X)$.
\end{lem}

\begin{lemma}\label{lem:HypersurfacePullback}
    Let $\varphi:\bfT\to\bfT'$ be a monomial map of tori, such that the corresponding map $\widehat\varphi$ of lattices is surjective. For any hypersurface $V(f)\subseteq \bfT'$, 
    \begin{equation}\label{eq:PullbackProduct}
    \Trop(V(\varphi^*(f))) = (\widehat\varphi)^{-1}(\Trop(V(f))),
    \end{equation}
    and moreover the weight function on $\Trop(V(\varphi^*(f)))$ is pulled back from the weight function on $\Trop(V(f))$.
\end{lemma}
(Lemma \ref{lem:HypersurfacePullback} follows from the fact that a surjective map of finite-rank free $\Z$-modules splits, which provides a splitting of $\varphi$.)

\subsection{Tropical intersection theory} \label{sec:BackgroundTropicalIntersectionTheory}
The most general formulation of tropical intersection theory is given in \cite{OssermanPayne2013}. We only need a few results, which we now state.

Let $\{X_j\}\subseteq \bfT^K$ be subvarieties of pure codimensions $r_j.$ We say the $\Trop(X_j)$s \textbf{intersect tropically transversely at $\Gamma\in\bigcap_j\Trop(X_j)$} if, locally near $\Gamma$, the intersection looks like the transverse intersection of affine-linear spaces. In this case, the intersection $\bigcap_j\Trop(X_j)$ is locally an affine-linear space of codimension $\sum_j r_j.$

Suppose $\Trop(X_1)$ and $\Trop(X_2)$ intersect tropically transversely at $\Gamma.$ Let $\zeta_1$ and $\zeta_2$ be the facets of $\Trop(X_1)$ and $\Trop(X_2)$, respectively, that contain $\Gamma$ in their relative interiors. Transversality implies that the lattice $L_\mathbb{Z}(\zeta_1) + L_\mathbb{Z}(\zeta_2)$ is of finite index in $N$ and, dually, $L_\mathbb{Z}(\zeta_1)^\perp \cap L_\mathbb{Z}(\zeta_2)^\perp = 0$. 

The \textbf{tropical intersection multiplicity} of $\Trop(X_1)$ and $\Trop(X_2)$ at $\Gamma$ (or rather, along the facet of $\Trop(X_1)\cap\Trop(X_2)$ containing $\Gamma$) is 
\begin{align}\label{eq:TropIntMultDef}
    i(\Gamma,\Trop(X_1)\cdot\Trop(X_2))=[N:L_{\Z}(\zeta_1)+L_{\Z}(\zeta_2)]\cdot m_{\Trop(X_1);\zeta_1}\cdot m_{\Trop(X_2);\zeta_2},
\end{align}
or equivalently,
\begin{align}\label{eq:TropIntMultDef2}
    i(\Gamma,\Trop(X_1)\cdot\Trop(X_2))=[L_{\Z}(\zeta_1\cap\zeta_2)^{\perp}:L_{\Z}(\zeta_1)^\perp\oplus L_{\Z}(\zeta_2)^\perp]\cdot m_{\Trop(X_1);\zeta_1}\cdot m_{\Trop(X_2);\zeta_2}.
\end{align}
% Jake's argument: We have $L_\Z(\zeta_1) \cap L_\Z(\zeta_2) = L_\Z(\zeta_1 \cap \zeta_2)$ (an intersection of saturated sublattices is saturated).  Recall that dualizing a saturated sublattice produces
% \[0 \to A \to N \to \frac{N}{A} \to 0 \quad \leadsto \quad 0 \to \underbrace{\Big(\frac{N}{A}\Big)^\vee}_{=:A^\perp} \to N^\vee \to A^\vee \to 0.\]
% So with $A_i = L_\Z(\zeta_i)$, we start with
% \[
% 0 \to \frac{N}{A_1 \cap A_2} \to \frac{N}{A_1} \oplus \frac{N}{A_2} \to \underbrace{\frac{N}{A_1+A_2}}_{F} \to 0,
% \]
% the left and middle are free (since $A_1$, $A_2$ and $A_1 \cap A_2$ are saturated in $N$) and we have assumed the right to be finite, call it $F$. We apply $\Hom(-,\Z)$. We get
% \[
% 0 \to \underbrace{\Hom(F, \Z)}_{=0} \to \underbrace{\Big(\frac{N}{A_1}\Big)^\vee \oplus \Big(\frac{N}{A_2}\Big)^\vee}_{A_1^\perp \oplus A_2^\perp} \to 
% \underbrace{\Big(\frac{N}{A_1 \cap A_2}\Big)^\vee}_{(A_1 \cap A_2)^\perp} \to 
% \underbrace{\mathrm{Ext}_\Z^1(F, \Z)}_{\approx F} \to 0
% \]
% since $F$ is finite.

The \textbf{stable tropical intersection} $\Trop(X_1)\cdot\Trop(X_2)$ of $\Trop(X_1)$ and $\Trop(X_2)$ is a balanced weighted polyhedral complex contained in $\Trop(X_1)\cap\Trop(X_2)$, defined via the \emph{local fan displacement rule}, see \cite[Sec. 2.6]{OssermanPayne2013}. If $\Trop(X_1)$ and $\Trop(X_2)$ intersect tropically transversely at $\Gamma$, then locally near $\Gamma$, $\Trop(X_1)\cap\Trop(X_2)$ is a facet of $\Trop(X_1)\cdot\Trop(X_2)$ which has weight equal to the tropical intersection multiplicity $i(\Gamma,\Trop(X_1)\cdot\Trop(X_2)).$

Osserman and Payne \cite{OssermanPayne2013} generalized all of the above notions, allowing tropical intersections to take place inside an ambient tropical subvariety $\Trop(Y)\subseteq N_{\R},$ where $Y\subseteq\bfT^K$ is a subvariety containing $X_1$ and $X_2.$ In this setup, $\Trop(X_1)$ and $\Trop(X_2)$ intersect tropically transversely at $\Gamma\in\Trop(X_1)\cap\Trop(X_2)$ if $\Gamma$ is in a weight-1 facet of $\Trop(Y)$ (that is, $\Gamma$ is a \emph{simple point} of $\Trop(Y)$) and the intersection locally looks like the transverse intersection of affine subspaces of $\Trop(Y).$ The tropical intersection multiplicity $i(\Gamma,\Trop(X_1)\cdot\Trop(X_2);\Trop(Y))$ at such a point is defined as in \eqref{eq:TropIntMultDef}, replacing $N$ with $L_{\Z}(\Trop(Y),\Gamma)$. Similarly there is a notion of stable tropical intersection $\Trop(X_1)\cdot\Trop(X_2)$ in $\Trop(Y)$, provided every facet of $\Trop(X_1)\cap\Trop(X_2)$ intersects $\relint(Y).$
\begin{ex}
    The reader may verify that the tropical intersections in Figures \ref{fig:BezoutExample} (right) and \ref{fig:P2Case} are tropically transverse, with tropical intersection multiplicities as shown, noting that
    \begin{itemize}
        \item The leftmost red rays in Figure \ref{fig:BezoutExample} have slopes $1/2$ and $2$, with all other slopes in $\{0,\pm1,\infty\}$.
        \item The leftmost red ray in Figure \ref{fig:P2Case} has slope 2, with all other slopes in $\{0,\pm1,\infty\}$.
    \end{itemize}
\end{ex}
There are various results of the form ``tropicalization commutes with intersection'', see \cite[Lem. 3.2]{BogartJensenSpeyerSturmfelsThomas2007} for an early such theorem. We will repeatedly use refinements due to Osserman-Payne \cite[Sec. 5.1]{OssermanPayne2013}.

\subsection{Tropical and toric geometry} We consider how tropicalizations are studied in the context of toric geometry. See \cite{Fulton1993} for an introduction to toric geometry. 

In this section, we work over $\Z$, so that we may later base change to any of $\k,$ $K$, or $R$. Let $M\cong\Z^n$ be a lattice, and let $N$ be the dual lattice. Let $\bfT^{\Z}=\Spec(\Z[M])$ denote the torus with character lattice $M$.
\begin{notation}\label{not:ToricNotation}
    Let $\Sigma$ be a smooth fan in $N_{\R}$, and let $Y_\Sigma$ denote the corresponding toric scheme. (The base change of $Y_\Sigma$ to any field is a smooth toric variety.)
    For an $r$-dimensional cone $\sigma\in\Sigma,$
    \begin{enumerate}
        \item Let $\sigma^{\vee}=\{w\in M:\sigma(w)\ge0\}\subseteq M$ denote the dual monoid and $\sigma^{\perp}=\{w\in M:\sigma(w)=0\}\subseteq M$ denote the annihilator sublattice.
        \item Let $Y_\sigma\subseteq Y_\Sigma$ denote the corresponding codimension-$r$ $\bfT^\Z$-orbit, and let $\bar Y_\sigma$ denote its closure. We have a canonical isomorphism $Y_\sigma=\Spec(\Z[\sigma^\perp])$.
        \item Let $\bfT^\Z_\sigma\subseteq\bfT^\Z$ denote the $r$-dimensional subtorus corresponding to the subspace $L_\R(\sigma)\subseteq N_\R.$ Note that $\bfT^\Z_\sigma\subseteq\bfT^\Z$ is the set of elements that fix $Y_\sigma$ pointwise and we have a natural isomorphism $\bfT^\Z_\sigma=\Spec(\Z[M/\sigma^\perp])$. There is a natural $\bfT^\Z$-equivariant map $\rho_\sigma:\bfT^\Z\to Y_\sigma$, which factors through the projection along $\bfT^\Z_\sigma$:
        
        $$\begin{tikzcd}
            \bfT^\Z\arrow[r]\arrow[rr,bend left,"\rho_\sigma"]&\bfT^\Z/\bfT^\Z_\sigma\arrow[r,"\cong",swap]&Y_\sigma
        \end{tikzcd}$$
        We have a distinguished point $\rho_\sigma(\id_{\bfT^\Z})\in Y_\sigma$, where $\id_{\bfT^\Z}\in\bfT^\Z$ is the identity element; the canonical isomorphism $\bfT^\Z/\bfT^\Z_\sigma\cong Y_\sigma$ identifies $\rho_\sigma(\id_{\bfT^\Z})$ with $\id_{\bfT^\Z/\bfT^\Z_\sigma}.$
        Note that the sequence of maps $\bfT^\Z_\sigma\into\bfT^\Z\xrightarrow{\rho_\sigma}{} Y_\sigma$ is induced by the sequence of lattice homomorphisms $\sigma^\perp\to M\to M/\sigma^\perp$.\label{not:TorusProjection}

        \item Let $U_\sigma\subseteq Y_\Sigma$ denote the toric affine chart $\Spec(\Z[\sigma^\vee])$ corresponding to $\sigma$. We have a natural inclusion $\bfT^\Z\into U_\sigma$ induced by the monoid homomorphism $\sigma^{\vee}\into M$. The map $\rho_\sigma:\bfT^\Z\to Y_\sigma$ extends to a $\bfT^\Z$-equivariant map $\rho_\sigma:U_\sigma\to Y_\sigma$. Since $U_\sigma$ is smooth, the fibers of this map are isomorphic to $\A_\Z^r.$ 
        
        In fact, we have (noncanonical) $\bfT^\Z$-equivariant splittings

        $$\begin{tikzcd}
            \bfT^\Z\arrow[r,"\cong"]\arrow[d,hook]&(\bfT^\Z/\bfT^\Z_\sigma)\times \bfT^\Z_\sigma\arrow[r,"="]&Y_\sigma\times \bfT^\Z_\sigma\arrow[d,hook]\arrow[dr]\\
            U_\sigma\arrow[rr,"\cong"]\arrow[rrr,bend right=20,"\rho_\sigma"]&&Y_\sigma\times\A^r_\Z\arrow[r]&Y_\sigma
        \end{tikzcd}$$
        where $\bfT^\Z_\sigma\into\A^r_\Z$ is the inclusion of the usual torus. Note that $\rho_\sigma$ is induced by the monoid homomorphism $\sigma^\perp\into\sigma^\vee.$
        \label{not:ToricAffineProjection}
    \end{enumerate}
    We will use these notations heavily in the proof of Theorem \ref{thm:intro1}.
\end{notation}
\begin{prop}\label{prop:FanOfToricStratum}
    The fan of the toric scheme $\bar{Y}_\sigma$ is equal to $$(\Star_\Gamma\Sigma)/L_{\R}(\sigma)\subseteq N_{\R}/L_{\R}(\sigma),$$ for any $\Gamma\in\relint(\sigma).$
\end{prop}
The following is a fundamental relationship between tropicalization and toric geometry. Let $K$ be a valued field, and let $R,\m,\k$ denote the valuation ring, maximal ideal, and residue field. A $K$-valued point $y : \Spec K \to \bfT^\Z$ induces a map $M\to K^*,$ and composing with the valuation defines a point of $\Hom(M,\R)=N_\R$, which we denote by $\Trop(y).$
\begin{prop}\cite[Thm. 6.6.20]{MaclaganSturmfels2015}\label{prop:TropicalizationIntersectsOrbits}
\begin{enumerate}
\item[(i)] $\Trop(y) \in \abs{\Sigma}$ if and only if $y$ extends to an $R$-valued point $\overline{y} : \Spec R \to Y_\Sigma$.
\item[(ii)] Under the conditions of (i), there is a unique cone $\sigma \in \Sigma$ whose relative interior contains $\Trop(y)$; it is also the unique cone corresponding to the locally closed stratum $Y_\sigma$ containing the image of $\Spec \k \to Y_\Sigma$ by $\overline{y}$.
\end{enumerate}
\end{prop}

\subsection{Tropical compactifications}\label{sec:BackgroundTropicalCompactifications}
\begin{notation}
    In this section we work over a \emph{trivially valued} algebraically closed field $\k$ of characteristic zero. Let $\bfT$ be a torus over $\k$. 
\end{notation}
Let $\M\subseteq\bfT$ be a subvariety. As $\k$ is trivially valued, the tropicalization $\Trop(\M)$ is the support of a fan (Section \ref{sec:BackgroundTropicalBasics}), and it is natural to consider the relationship between $\M$ and the corresponding toric variety. (In fact, $\Trop(M)$ is generally the support of many fans.) This idea is the origin of the notion of tropical compactification \cite{Tevelev2007}. 

Let $\Sigma$ be a fan in $N_\R$, and let $\Mbar$ denote the closure of $\M$ in the corresponding toric variety $Y_\Sigma\supseteq\bfT$. By \cite[Prop. 6.4.7(1)]{MaclaganSturmfels2015}, $\Mbar$ is proper if and only if $\Trop(\M)\subseteq\abs{\Sigma}.$ We say that $\Sigma$ is a \textbf{tropical fan for $\M$} if $\Mbar$ is proper \emph{and} the multiplication map $\bfT\times\Mbar\to Y_\Sigma$ is faithfully flat \cite[Def 6.4.13]{MaclaganSturmfels2015}. This condition implies $\abs{\Sigma}=\Trop(\M)$, but the converse is not true \cite[Prop. 6.4.14, Ex. 6.4.16]{MaclaganSturmfels2015}. If $\Mbar\supseteq\M$ is a compactification arising from some tropical fan, we say $\Mbar$ is a \textbf{(flat) tropical compactification} of $\M$. Note that from here-on we will drop the word `flat' from `tropical compactification' but will assume flatness as part of the definition.

The variety $\Mbar$ inherits a stratification from $Y_\Sigma.$ For each cone $\sigma$ of $\Sigma$, let $\M_{\sigma}=\Mbar\cap Y_\sigma$ (resp. $\Mbar_{\sigma}=\Mbar\cap\bar Y_\sigma$) denote the corresponding locally closed (resp. closed) boundary stratum in $\Mbar$.

\begin{prop}[See {\cite[Sec. 6.4]{MaclaganSturmfels2015}}]\label{prop:TropicalCompactificationsBasicProperties}
    Tropical compactifications satisfy the following basic properties:
\begin{enumerate}
    \item If $\Sigma$ is a tropical fan for $\M\subseteq \bfT$, the boundary $\Mbar\setminus\M$ has \emph{combinatorial normal crossings}, i.e. for every $\sigma\in\Sigma,$ the stratum $\Mbar_{\sigma}$ has pure codimension $\dim\sigma$ in $\Mbar$.\label{eq:TropicalCompactificationStratumCodimension}
    \item $\Mbar_\sigma$ is the closure of $\M_\sigma$ in $Y_\Sigma$.
    \item Every subvariety $\M\subseteq\bfT$ admits a tropical fan. 
    \item The refinement of a tropical fan for $\M$ is a tropical fan for $\M$.
    \item If $\Sigma$ is a tropical fan for $\M$, then for any cone $\sigma\in\Sigma,$ $\Mbar_\sigma\subseteq\bar Y_\sigma$ is a flat tropical compactification of $\M_\sigma\subseteq Y_\sigma$, with tropical fan $\Star_\sigma\Sigma$. \label{it:ExtendedTropicalStratum}
\end{enumerate}
\end{prop}

The main result of \cite{SturmfelsTevelev2008} shows that tropical compactifications give a tropical-geometric way of detecting the degree of a generically finite map. Let $\bfT_1$ and $\bfT_2$ be tori over $\k$ of dimensions $n_1$ and $n_2$, with cocharacter lattices $N_1$ and $N_2$, respectively. Let $\M_1\subseteq \bfT_1$ and $\M_2\subseteq \bfT_2$ be subvarieties, with $\M_2$ irreducible. Let $f:\bfT_1\to \bfT_2$ be a torus homomorphism, with induced linear map $F:(N_1)_\R\to (N_2)_\R$, such that the restriction $f|_{\M_1}$ is a generically finite map from $\M_1$ to $\M_2$ of degree $d$.
\begin{thm}[{\cite[Thm. 3.12]{SturmfelsTevelev2008}}]\label{thm:Sturmfels-Tevelev}
    Suppose $\Sigma_1$ and $\Sigma_2$ are tropical fans for $\M_1$ and $\M_2$ respectively, such that for all $\sigma\in\Sigma_1$, $F(\sigma)\subseteq\abs{\Sigma_2}$ is a union of cones. Let $\zeta$ be a maximal cone of $\Sigma_2$. Then the degree of $f|_{\M_1}$ is 
    \begin{equation}\label{eq:sturmfels-tevelev}
    d = \frac{1}{m_{\Sigma_2;\zeta}} \sum_{\substack{\sigma \in \Sigma_1 \\ F(\sigma) \supseteq \zeta}} m_{\Sigma_1;\sigma} \cdot[L_{\Z}(\zeta) : F(L_{\Z}(\sigma))],
    \end{equation}
    where $m_{\Sigma_2;\zeta}$ and $m_{\Sigma_1;\sigma}$ are the weights of $\zeta$ and $\sigma$, and $L_{\Z}(\zeta)$ and $L_{\Z}(\sigma)$ are the associated lattices (Section \ref{sec:BackgroundPolyhedralGeometry}).
\end{thm}
Geometrically, $F|_{\Sigma_1}:\Sigma_1\to\Sigma_2$ looks like a ``generically finite map of polyhedral complexes of degree $d$''.

We will need the following minor modification of Theorem \ref{thm:Sturmfels-Tevelev}.

\begin{cor}\label{cor:Sturmfels-Tevelev-New}
    With $d$ as above, fix an element $\Phi\in\relint(\Trop(\M_2))$ such that $F^{-1}(\Phi)\cap\Trop(\M_1)\subseteq\relint(\Trop(\M_1)).$ Then
    \begin{equation} \label{eq:ST-multiplicities-at-point}
    d = \frac{1}{m_{\Trop(\M_2);\Phi}} \sum_{\substack{\Lambda \in \Trop(\M_1) \\ F(\Lambda) = \Phi}} m_{\Trop(\M_1);\Lambda} \cdot[L_{\Z}(\Trop(\M_2);\Phi) : F(L_{\Z}(\Trop(\M_1);\Lambda))],
    \end{equation}
    where $m_{\Trop(\M_2);\Phi}$ and $m_{\Trop(\M_1);\Lambda}$ are the weights at $\Phi$ and $\Lambda$, and $L_{\Z}(\Trop(\M_2);\Phi)$ and $L_{\Z}(\Trop(\M_1);\Lambda)$ are the associated lattices. (See Remarks \ref{rem:RelativeInteriorOfSupport} and \ref{rem:RelativeInteriorOfSupportWeights}.)
    In particular, Equation \eqref{eq:ST-multiplicities-at-point} holds for $\Phi$ in a dense open subset of $\Trop(\M_2)$.
\end{cor}

\begin{proof}
    By \cite[Thm. 1.2]{Tevelev2007}, there exist tropical fans $\Sigma_1,\Sigma_2$ for $\M_1$ and $\M_2$ respectively.
    By standard techniques in convexity theory, we may refine $\Sigma_2$ (say to $\Sigma_2'$) so that for each $\sigma\in\Sigma_1,$ $F(\sigma)$ is a union of cones of $\Sigma_2'$. By \cite[Prop. 2.5]{Tevelev2007}, $\Sigma_2'$ is still a tropical fan for $\M_2$.
    
    Let $\Phi\in\relint(\Trop(\M_2))$. By Remark \ref{rem:RelativeInteriorOfSupport} and the fact that $F^{-1}(\Phi)\cap\Trop(\M_1)\subseteq\relint(\Trop(\M_1))$, the quantity \eqref{eq:ST-multiplicities-at-point} is locally constant at $\Phi$. Therefore we may assume (for the purposes of showing \eqref{eq:ST-multiplicities-at-point}) that $\Phi$ is in the relative interior of a maximal cone $\zeta \subseteq \Sigma_2'$, and likewise $\Lambda \in F^{-1}(\Phi) \cap \Trop(\M_1)$ is in the relative interior of a maximal cone $\sigma \subseteq \Sigma_1$. For such a $\Phi$, the right side of \eqref{eq:ST-multiplicities-at-point} agrees with the right side of \eqref{eq:sturmfels-tevelev}; therefore Theorem \ref{thm:Sturmfels-Tevelev} implies \eqref{eq:ST-multiplicities-at-point}.
\end{proof}

\begin{remark}
    Corollary \ref{cor:Sturmfels-Tevelev-New} makes no reference to tropical compactifications, and simply formulates the degree $d$ as a sum over the preimages of a sufficiently nice point $\Phi$, in particularly one such that $\Phi$ and its preimages in $\Trop(\M_1)$ are all locally linear. In this sense, Theorem \ref{thm:Sturmfels-Tevelev} is not really about tropical compactifications after all, but about detecting the degree of a generically finite map of subvarieties of tori under a torus homomorphism. Corollary \ref{cor:Sturmfels-Tevelev-New} (and its generalization to the case of nontrivially valued fields \cite{OssermanPayne2013,BakerPayneRabinoff2016}) is used to define the correct notion of pushforward for tropical cycles, see \cite{Gubler2013}.
\end{remark}

\section{Tropical compactifications: Boundary strata and initial subschemes}\label{sec:StratificationOfTropicalCompactification}

In this section, we prove some results on boundary strata in tropical compactifications. Lemma \ref{lem:InitialIdealInStratum}, which relates the boundary stratification of a tropical compactification $\Mbar$ with initial subschemes of $\M$, is likely known to experts, but does not appear in the literature as far as we know.

\begin{notation}\label{not:TropicalCompactificationSetup}
In this section and in Section \ref{sec:DegenerationsInTropicalCompactifications}, we fix the following data (matching the notation in the introduction):
\begin{itemize}
    \item \textbf{Base field:} Let $\k$ be an algebraically closed field of characteristic zero, with trivial valuation.
    \item \textbf{Torus:} Let $\bfT$ be an $n$-dimensional torus over $\k$.
    \item \textbf{Subvariety of torus:} Let $\M\subseteq\bfT$ be a subvariety of pure dimension $m$.
    \item \textbf{Fan, ambient toric variety and tropical compactification:} Let $Y_\Sigma$ be a smooth toric variety with fan $\Sigma,$ such that $\Mbar\subseteq Y_{\Sigma}$ is a tropical compactification of $\M$. 
\end{itemize}
Recall that for a cone $\sigma\in\Sigma$, we write $Y_\sigma\subseteq Y_\Sigma$ for the corresponding torus orbit, $\bar Y_\sigma$ for its closure, and $\bfT_\sigma\subseteq\bfT$ for the corresponding subtorus, as in Notation \ref{not:ToricNotation}. We also write $\M_{\sigma}=\Mbar\cap Y_\sigma,$ and $\Mbar_\sigma=\Mbar\cap \bar Y_\sigma$ as in Section \ref{sec:BackgroundTropicalCompactifications}.
\end{notation}

\begin{lem}\label{lem:InitialIdealInStratum}
    Let $\sigma\in\Sigma$ be a cone, and let $\Gamma\in\relint(\sigma)$. Let $\bfT_\sigma$ and $\rho_\sigma:\bfT\to Y_\sigma$ be as in Notation \ref{not:ToricNotation}\eqref{not:TorusProjection}. Then as subschemes of $\bfT$, we have $$\init_\Gamma\M=\rho_\sigma^{-1}(\M_\sigma).$$ In particular, the noncanonical isomorphism $\bfT \cong \bfT/\bfT_\sigma \times \bfT_\sigma\cong Y_\sigma\times \bfT_\sigma$ induces an identification $\init_\Gamma\M\cong\M_\sigma\times \bfT_\sigma,$ with respect to which $\rho_\sigma$ is the first projection.
\end{lem}

\begin{proof}[Proof of Lemma \ref{lem:InitialIdealInStratum}]

Consider the isomorphism $\bfT \times U_\sigma \to \bfT \times U_\sigma$ given by $(x, z) \mapsto (x, x^{-1}z)$. Restricting it to $\bfT \times (\Mbar \cap U_\sigma)$ yields the following diagram:
$$\begin{tikzcd}
    \bfT\times(\Mbar\cap U_\sigma)\arrow[rrr,"{(x,z)\mapsto(x,x^{-1}z)}",hook]\arrow[rrrr,bend left,"\mathcal{F}","\mathrm{(flat)}"']&&&\bfT\times U_\sigma\arrow[r,"\mathrm{pr}_2"]&U_\sigma
\end{tikzcd}$$
where flatness of $\mathcal{F}$ is by the assumption that $\Mbar$ is a tropical compactification. The first map gives an isomorphism onto its image, hence we may view $\mathcal{F}$ as a flat family of subschemes of $\bfT$ over $U_\sigma$ whose fiber over $z\in U_\sigma$ is $$\{x\in \bfT:x\cdot z\in\Mbar\cap U_\sigma\}.$$ In particular, the fiber over $z\in \bfT\subseteq U_\sigma$ is $z^{-1}\cdot\M\subseteq \bfT$, and the fiber over the distinguished point $\rho_\sigma(\id_{\bfT})\in Y_\sigma$ is $\rho_\sigma^{-1}(\M_\sigma)\subseteq \bfT$, since $$\{x\in \bfT:x\cdot \rho_\sigma(\id_{\bfT})\in\Mbar\}=\{x\in \bfT:x\cdot \rho_\sigma(\id_{\bfT})\in\M_\sigma\}=\{x\in \bfT:\rho_\sigma(x)\in\M_\sigma\}=\rho_\sigma^{-1}(\M_\sigma).$$ Here the first equality is because the image of $\rho_\sigma$ is contained in $Y_\sigma$ (which is a $\bfT$-orbit), and the second equality is by $\bfT$-equivariance of $\rho_\sigma$.

We use Gr\"obner theory to show that the fiber $\rho_\sigma^{-1}(\rho_\sigma(\id_{\bfT}))$ can also be identified with $\init_\Gamma \M$. Let $K:=\k((\R))$ denote the ring of generalized power series\footnote{Alternatively, we may take $K$ to be any valued field extension of $\k$ with residue field $\k$, with a splitting, and over whose value group $\Gamma$ is defined.} over $\k$ in a variable $t$, with its natural valuation \cite[Ex. 2.1.7]{MaclaganSturmfels2015}. Note that $K$ has residue field $\k$, and natural splitting $\alpha\mapsto t^\alpha$. The $\k$-algebra morphism $\k[M]\to K$ given by $x\mapsto t^{\Gamma(x)}$ for $x\in M$ induces a morphism $y:\Spec(K)\to \bfT$ that transparently satisfies $\Trop(y)=\Gamma$. By Proposition \ref{prop:TropicalizationIntersectsOrbits}, $y$ extends to an $R$-valued point $\bar{y}:\Spec(R)\to U_\sigma$, such that $\bar{y}$ maps $\Spec(\k)$ into $Y_\sigma$. Furthermore, since $\Gamma\in\sigma$, we have that $y$ factors through $\bfT_\sigma\into \bfT$. Thus $$\bar{y}(\Spec(\k))\in Y_\sigma\cap\bar{\bfT_\sigma}=\{\rho_\sigma(\id_{\bfT})\}.$$

Let $\mathcal{F}|_{\bar{y}}$ denote the pullback of the flat family $\mathcal{F}:\bfT\times(\Mbar\cap U_\sigma)\to U_\sigma$ along $\bar{y}:\Spec(R)\to U_\sigma$. Then $\mathcal{F}|_{\bar{y}}$ 
 is a flat family of subschemes of $\bfT$ over $\Spec(R),$ whose general fiber $\mathcal{F}|_{y}$ is $t^{-\Gamma}\cdot\M$ (see \cite[Sec. 2.4]{MaclaganSturmfels2015}). By the above computations, the special fiber is $\mathcal{F}|_{\Spec(\k)}=\rho_\sigma^{-1}(\M_\sigma)$. On the other hand, by \cite[Lem. 2.4.14]{MaclaganSturmfels2015}, the special fiber is $$\mathcal{F}|_{\Spec(\k)}=\init_\Gamma\M.$$ We conclude $\rho_\sigma^{-1}(\M_\sigma)=\init_\Gamma(\M)$ as subschemes of $\bfT$.

The last sentence of the Lemma follows from Notation \ref{not:ToricNotation}\eqref{not:TorusProjection}.
\end{proof}

\begin{cor}\label{cor:ExtendedTropicalStratum}
    For any $\Gamma\in\relint(\sigma),$ we have $\Trop(\M_\sigma)=\Star_\Gamma\Sigma/L_{\R}(\sigma)$ as (supports of) weighted polyhedral complexes.
\end{cor}

\begin{proof}
    The fact that the two sets agree follows from Lemma \ref{prop:TropicalCompactificationsBasicProperties}\eqref{it:ExtendedTropicalStratum}.

    To see the fact about the weight function, let $\zeta$ be a maximal cell in $\Star_\Gamma\Sigma/L_\R(\sigma)$ corresponding to a maximal cell $\hat\zeta$ of $\Sigma$, and let $\Gamma'\in \relint(\zeta)$ and $\hat\Gamma'$ a lift of $\Gamma'$ to $\hat \zeta$. By Lemma~\ref{lem:InitialIdealInStratum}, 
    \begin{equation}\label{eq:InitInit}
    \init_{\hat\Gamma'}(\init_\Gamma \M) = \init_{\hat\Gamma'}(\rho_\sigma^{-1}(\M_\sigma)) = \rho_\sigma^{-1}(\init_{\Gamma'} \M_\sigma) \cong (\init_{\Gamma'} \M_\sigma)\times \bfT_\sigma.
    \end{equation}
    
    On the left-hand side, by \cite[Corollary 2.4.10]{MaclaganSturmfels2015}, for sufficiently small $\epsilon>0$ we have that $\init_{\hat\Gamma'}(\init_\Gamma \M) \cong \init_{\Gamma + \epsilon\Gamma'}\M$. For sufficiently small $\epsilon>0$, then $\Gamma + \epsilon\Gamma'\in \relint(\hat\zeta)$, thus $\init_{\Gamma+\epsilon\Gamma'}\M$ is a union of tori whose multiplicities add to the weight $m_{\Trop(\M);\hat\zeta}$.

    On the right-hand side of \eqref{eq:InitInit}, $(\init_{\Gamma'} \M_\sigma)\times \bfT_\sigma$ is a union of tori whose multiplicities add to $m_{\Trop(\M_\sigma);\zeta}$, since taking the product with $\bfT_\sigma$ does not change this number. Thus, $m_{\Trop(\M);\hat\zeta} = m_{\Trop(\M_\sigma);\zeta}$.
\end{proof}

\section{Degenerations in tropical compactifications, and tropical intersections with skeleta}\label{sec:DegenerationsInTropicalCompactifications}

Let $\mathbf{k},\bfT,n,\M,m,\Sigma,Y_\Sigma,\Mbar$ be as in Notation \ref{not:TropicalCompactificationSetup}. This section deals with a degenerating subvariety in a tropical compactification --- more precisely, with a codimension-$r$ subvariety of $\M$ defined over a nontrivially valued extension of the base field, as in Notation \ref{not:DegenerationSetup} just below. 

In Section \ref{sec:MultDef}, we introduce a ``transversality'' condition (Condition \ref{cond:SigmaGamma}) for a point $\Gamma\in\Trop(X)\cap\Sigma_r$, where $\Sigma_r$ is the $r$-skeleton of $\Sigma$. If $\Gamma$ satisfies this condition, we define (Definition \ref{def:MultIndependentOfChoice}) a natural associated local tropical intersection multiplicity $\mult_\Gamma(\Trop(X),\Sigma_r;\Sigma)$ at $\Gamma$, within $\Sigma$, between $\Sigma_r$ and $\Trop(X)$. Both of these notions --- transversality and tropical multiplicity --- are new. Existing versions do not apply, for the same reason in both cases --- namely, because $\Sigma_r$ is not generally in $\relint(\Sigma)$. We prove that $\mult_\Gamma(\Trop(X),\Sigma_r;\Sigma)$ is well-defined by showing (Definition/Proposition \ref{prop:GenericallyFinite}, Theorem \ref{thm:MultIndependentOfChoice}) that it is equal to the degree $d_\Gamma$ of a certain generically finite map of schemes, obtained as a certain projection onto a boundary stratum. Finally, we show (Theorem \ref{thm:XTilde2-OP}) that if $X$ is a (sufficiently ``tropically transverse'') intersection of some variety $\widetilde X\subseteq Y\subseteq\bfT^K$ with $\M^K,$ then $\mult_{\Gamma}(\Trop(X),\Sigma_r;\Sigma)$ coincides with a tropical intersection multiplicity in the usual sense (as in Section \ref{sec:BackgroundTropicalIntersectionTheory}).

In Section \ref{sec:ProofOfMainTheorem}, we study the limit $(\bar X)_0$ of $\bar X$. Under a further \emph{global} assumption (Condition \ref{cond:GlobalIntersectionCondition}) on $\Trop(X)\cap\Sigma_r,$ we show (Proposition \ref{prop:Set-Theoretic}) that $(\bar X)_0$ is a union of codimension-$r$ boundary strata $\Mbar_\sigma$ of $\Mbar.$ Then, in Theorem \ref{thm:mainthm}, we use the tropical multiplicities $\mult_{\Gamma}(\Trop(X),\Sigma_r;\Sigma)$ to express the limit \emph{cycle} $[(\bar X)_0]$ of $\bar X$ explicitly as a weighted sum of boundary strata $[\Mbar_\sigma].$ These two statements comprise Theorem \ref{thm:intro1}.

\begin{notation}\label{not:DegenerationSetup}
We fix the following for the remainder of this section:
\begin{itemize}
    
    \item \textbf{Valued field extension $K$ of $\mathbf{k}$:} Let $K$ be a nontrivially valued extension of $\k$ with valuation ring $R$ and maximal ideal $\m$, such that $\k \cong R/\m$. (All of the following theorems will hold vacuously if the valuation is trivial.)
    \item \textbf{Base changes from $\mathbf{k}$ to $K$ and $R$:} For a scheme $V$ over $\mathbf{k},$ we will write $V^K$ for the base change $V\times_{\Spec(\mathbf{k})}\Spec(K)$ and $V^R$ for the base change $V\times_{\Spec(\mathbf{k})}\Spec(R)$.
    \item \textbf{Subscheme $X\subseteq\M$ defined over $K$ of codimension $\le r$:} Let $r\ge0$ be an integer, and let $X\subseteq\M^K$ be a closed subscheme such that all irreducible components of $X$ have codimension at most $r$. Let $\bar X$ denote the closure of $X$ in $\Mbar^K$, let $\X$ denote the closure of $\bar X$ in $\Mbar^R$, and let $(\bar X)_0=\X\times_{\Spec(R)}\Spec(\k)$. ($X_0$ is often referred to as the \emph{flat limit} of $\bar X$.)
\end{itemize}
\end{notation}

We will need two further assumptions on $\M$, $\Sigma$, and $X$:
\begin{assumption}\label{ass:MSigmaNice}
    For the rest of this section, we assume that for each $r$-dimensional cone $\sigma\in\Sigma$, $\Mbar_\sigma$ is reduced and irreducible, and a general point of $\Mbar_{\sigma}$ is a smooth point of $\Mbar$ at which $\Mbar$ intersects $Y_\sigma$ transversely.
\end{assumption}

\subsection{Local tropical multiplicities associated to the degeneration \texorpdfstring{$X\subset \M^{K}$}{X in M\^{}K}}\label{sec:MultDef}

The notion of tropical multiplicity we define in this section can be defined at a point of $\Trop(X)\cap\Sigma_r$ where the following local condition is satisfied.

\begin{condition}\label{cond:SigmaGamma}
    Let $\sigma\in\Sigma_r$ be a facet, and let $\Gamma$ be an isolated point of $\Trop(X)\cap\sigma$. Suppose
    \begin{itemize}
        \item $\Gamma\in\relint(\sigma)$, and
        \item every facet of $\Trop(X)$ containing $\Gamma$ has dimension $m-r$.
    \end{itemize}
\end{condition}

\begin{defprop}\label{prop:GenericallyFinite}
    If $\sigma$ and $\Gamma$ are as in Condition \ref{cond:SigmaGamma}, then the composition 
    \begin{equation}\label{eqn:rho-sigma-on-init-gamma-X}
    \init_\Gamma X\into \bfT\into U_\sigma\xrightarrow{\rho_\sigma} Y_\sigma
    \end{equation}
    factors through $M_\sigma\subseteq Y_\sigma$, and the resulting map $\rho_\sigma|_{\init_\Gamma X}:\init_\Gamma X\to\M_\sigma$ is generically finite. We denote its degree by $d_\Gamma$.
\end{defprop}
\begin{proof}
    Recall the natural projection $\rho_\sigma : U_\sigma \to Y_\sigma$ (see Notation \ref{not:ToricNotation}\eqref{not:ToricAffineProjection}). Since $\Gamma\in\relint(\sigma),$ Lemma \ref{lem:InitialIdealInStratum} implies $$\init_\Gamma\M=\rho_\sigma^{-1}(\M_\sigma).$$ Since $X\subseteq\M,$ we have $\init_\Gamma X\subseteq\init_\Gamma\M=\rho_\sigma^{-1}(\M_\sigma),$ i.e.
    we may consider the restriction $$\rho_\sigma|_{\init_\Gamma X}:\init_\Gamma X\to\M_\sigma.$$ By Lemma \ref{lem:InitialTropicalization}, we have $\Trop(\init_\Gamma X)=\Star_\Gamma\Trop(X)$. By the structure theorem for tropical varieties, $\init_\Gamma X$ and $\Star_\Gamma\Trop(X)$ have the same dimension, which is $m-r$ by the assumption that every facet of $\Trop(X)$ containing $\Gamma$ has dimension $m-r.$ We also know that $\M_\sigma$ is $(m-r)$-dimensional by Proposition \ref{prop:TropicalCompactificationsBasicProperties}\eqref{eq:TropicalCompactificationStratumCodimension}. Thus $\rho_\sigma|_{\init_\Gamma X}$ is a morphism of $(m-r)$-dimensional schemes, whose codomain $\M_\sigma$ is irreducible by Assumption \ref{ass:MSigmaNice}. Hence $\rho_\sigma|_{\init_\Gamma X}$ is generically finite over $\M_\sigma$.
\end{proof}

In fact, the degree $d_\Gamma = \deg(\init_\Gamma X \to \M_\sigma)$ can be computed entirely via tropical geometry --- this follows from a general result of Sturmfels-Tevelev \cite{SturmfelsTevelev2008} (Corollary \ref{cor:Sturmfels-Tevelev-New}), as we now show.

\begin{thm}\label{thm:MultIndependentOfChoice}
    Let $\sigma$ and $\Gamma$ be as in Condition \ref{cond:SigmaGamma}. Choose a facet $\zeta\in\Sigma$ containing $\Gamma$, and choose a generic vector $v_\zeta\in\Star_\Gamma\zeta$. Then the tropical $0$-cycle 
    \begin{align}\label{eq:MultCycleDef}
        \frac{1}{m_{\Trop(\M);\zeta}}((v_\zeta+\Star_\Gamma\sigma)\cdot\Star_\Gamma\Trop(X)),
    \end{align}
    has degree $d_\Gamma$. Here ``$\cdot$'' denotes stable tropical intersection taken inside the ambient space $\Star_\Gamma\Sigma$, and the affine-linear space $v_\zeta+\Star_\Gamma\sigma$ is assigned weight 1. In particular, the degree is independent of $\zeta$ and $v_\zeta$.
\end{thm}

\begin{remark}
    The $\Star_\Gamma$ operations in \eqref{eq:MultCycleDef} essentially say that $\mult_\Gamma(\Trop(X),\sigma;\Sigma)$ may be computed locally near $\Gamma$.
\end{remark}

\begin{definition}\label{def:MultIndependentOfChoice}
    Let $\sigma$ and $\Gamma$ be as in \ref{cond:SigmaGamma}. We define the \textbf{tropical intersection multiplicity of $\Trop(X)$ and $\sigma$ in $\Sigma$ at $\Gamma$}, denoted $\mult_{\Gamma}(\Trop(X),\sigma;\Sigma)$, as the degree of the tropical cycle in \eqref{eq:MultCycleDef}, for any $\zeta$ and any generic $v_\zeta$.
\end{definition}
\begin{remark}
Our main goal is to give a tropical method to compute limit cycles. As we will see below in Lemma \ref{lem:mainlem}, these limits are more directly expressed in terms of the algebraic data $d_\Gamma$. With the equality
\[\mult_{\Gamma}(\Trop(X),\sigma;\Sigma)=d_\Gamma\]
of Theorem \ref{thm:MultIndependentOfChoice},
we may instead compute these limits purely tropically (Theorem \ref{thm:mainthm}, the main theorem of the paper). We use the latter in our application to $\Mbar_{0, n}$.
\end{remark}

\begin{proof}[Proof of Theorem \ref{thm:MultIndependentOfChoice}]
    By definition, $\rho_\sigma|_{\init_\Gamma X}$ is induced by the homomorphism $\rho_\sigma:\bfT\to \bfT/\bfT_\sigma\cong Y_\sigma$ of tori. Let $\rho_\sigma^{\trop}:N_{\R}\to N_{\R}/L_{\R}(\sigma)$ denote the corresponding linear map. By Corollary \ref{cor:Sturmfels-Tevelev-New}, there exists a dense open subset of $\Trop(\M_\sigma)\subseteq N_{\R}/L_{\R}(\sigma)$ such that for any $\Phi$ in this subset, we have
    \begin{align}\label{eq:AppliedST}
        d_\Gamma&=\frac{1}{m_{\Trop(\M_\sigma);\Phi}}\sum_{\substack{\Lambda\in\Trop(\init_\Gamma X)\\\rho_\sigma^{\trop}(\Lambda)=\Phi}}m_{\Trop(\init_\Gamma X);\Lambda}\cdot[L_{\Z}(\Trop(\M_\sigma);\Phi):\rho_\sigma^{\trop}(L_{\Z}(\Trop(\init_\Gamma X);\Lambda))].
    \end{align}
    We now describe the pieces of this expression one by one. By Lemma \ref{lem:InitialTropicalization}, we have $$\Trop(\init_\Gamma X)=\Star_\Gamma\Trop(X)\subseteq N_{\R}.$$ By Corollary \ref{cor:ExtendedTropicalStratum}, we have $\Trop(\M_\sigma)=\abs{\Star_\Gamma\Sigma/L_{\R}(\sigma)}$, and the weighting function on $\relint(\Star_\Gamma\Sigma)$ is pulled back from the weighting function on $\relint(\Trop(\M_\sigma))$. 

    These quantities are all local, so by passing to a sufficiently small open set around $\Gamma$, and translating so that $\Gamma = \vec{0}$, we may assume that $\sigma$ is a linear space and that $\Sigma=\Trop(\M)$, $\Trop(\M_\sigma)$ and $\Trop(X)$ are cone complexes, with $\sigma$ in the lineality space of $\Sigma$. In particular, we have:
    \begin{align}\label{eq:substitutions}
        L_\R(\sigma) &= \sigma,&\Star_\Gamma \Sigma &= \Sigma&\Trop(\M_\sigma) &= \Sigma/\sigma&&\text{and}&\Star_\Gamma \Trop(X) &= \Trop(X).
    \end{align}
    
    Let $\zeta \in \Sigma$ be a maximal cone containing $\Gamma$ and let $v_\zeta \in \zeta$ be generic. Note that, since $\sigma$ is in the lineality space of $\Sigma$, $v_\zeta + \sigma \subseteq \zeta$. The map
    \[
    \rho_\sigma^\trop : \Sigma \to \Sigma/\sigma
    \]
    is surjective and takes $v_\zeta$ to a generic vector in $\Trop(\M_\sigma) = \Sigma/\sigma$, so we may apply Equation \eqref{eq:AppliedST} with $\Phi = \rho_\sigma^{\trop}(v_\zeta)$ and with the substitutions \eqref{eq:substitutions}. We have $(\rho_\sigma^\trop)^{-1}(\rho_\sigma(v_\zeta)) = v_\zeta + \sigma$, so the condition on $\Lambda$ is just $\Lambda \in \Trop(X) \cap (v_\zeta + \sigma)$. Finally, by Corollary \ref{cor:ExtendedTropicalStratum}, we also have an equality of multiplicities
    \[
    m_{\Sigma/\sigma;\rho_\sigma^{\trop}(v_\zeta)}=m_{\Sigma;\Lambda}=m_{\Sigma;\zeta}
    \] 
    for all such $\Lambda$. Substituting all of this into \eqref{eq:AppliedST} yields:
    \begin{equation}
    \label{eq:AppliedST2}
    d_\Gamma = 
    \sum_{\Lambda\in \Trop(X) \cap (v_\zeta+\sigma)}\frac{m_{\Trop(X);\Lambda}}{m_{\Sigma;\zeta}}\cdot[L_{\Z}(\Sigma/\sigma;\rho_\sigma^\trop(v_\zeta)):\rho_\sigma^{\trop}(L_{\Z}(\Trop(X);\Lambda))].
    \end{equation}
    The two lattices appearing in the index calculation are both sublattices of $N/L_\zeta(\sigma)$. Their preimages under $\rho_\sigma^\trop : N \to N/L_\Z(\sigma)$ are, respectively,
    \begin{align*}
        (\rho_\sigma^{\trop})^{-1}(L_{\Z}(\abs{\Sigma/\sigma};\rho_\sigma^\trop(v_\zeta))) &= L_{\Z}(\Sigma;\Lambda)\\
        (\rho_\sigma^{\trop})^{-1}(\rho_\sigma^{\trop}(L_{\Z}(\Trop(X);\Lambda))) &= L_{\Z}(\Trop(X);\Lambda)\oplus L_{\Z}(\sigma).
    \end{align*}
    By the third isomorphism theorem for groups, we may replace the lattice index in \eqref{eq:AppliedST2} by
    \begin{equation}
        [L_{\Z}(\Sigma;\Lambda) : L_{\Z}(\Trop(X);\Lambda)\oplus L_{\Z}(\sigma)],
    \end{equation}
    which by definition is
    \[
    \frac{1}{m_{\Trop(X); \Lambda}} i(\Lambda, \Trop(X) \cdot (v_\zeta + \sigma)),
    \]
    where the tropical intersection multiplicity is taken in $\Sigma$ and $(v_\zeta + \sigma)$ is given weight $1$. (See Section \ref{sec:BackgroundTropicalIntersectionTheory}.) Thus, we may write \eqref{eq:AppliedST2} as
    \[
    d_\Gamma = \frac{1}{m_{\Sigma; \zeta}} \sum_{\Lambda \in \Trop(X) \cap (v_\zeta + \sigma)} i(\Lambda, \Trop(X) \cdot (v_\zeta + \sigma)).
    \]
    Summing over $\Lambda$, we obtain
    \begin{align*}
        d_\Gamma =
        \frac{1}{m_{\Sigma;\zeta}} \deg((v_\zeta + \sigma) \cdot \Trop(X)),
    \end{align*}
    which is exactly the degree of the cycle \eqref{eq:MultCycleDef}. In particular, the degree of this stable tropical intersection is independent of the choices of $\zeta$ and $v_\zeta$.
\end{proof}

We next prove Theorem \ref{thm:XTilde}, restated here:

\medskip 

\begin{thm}\label{thm:XTilde2-OP} 
    Suppose $\M$ is Cohen-Macaulay and contained in a subvariety $Y \subseteq \bfT$ of dimension $n'$. Let $\widetilde X\subseteq Y^K$ be a Cohen-Macaulay subscheme, and let $X = \widetilde X \cap \M$ be the scheme-theoretic intersection.
    
    Let $\Gamma$ be an isolated point of $\Trop(\widetilde X)\cap \Sigma_r$ that is moreover in the relative interior of:
    \begin{enumerate}
        \item a facet $\eta$ of $\Trop(Y)$ of weight $1$,
        \item a facet $\sigma$ of $\Sigma_r$, and
        \item an $(n'-r)$-dimensional facet of $\Trop(\widetilde X)$,
    \end{enumerate}
    By \cite[Thm. 1.2]{OssermanPayne2013}, $\Gamma \in \Trop(X)$. It follows that $\Gamma$ is an isolated point of $\Trop(X) \cap \sigma$ and satisfies Condition \ref{cond:SigmaGamma}. Then \begin{align}\label{eq:XTilde2}
        \mult_{\Gamma}(\Trop(X),\sigma;\Sigma) = i(\Gamma, \Trop(\widetilde{X}) \cdot\sigma;\Trop(Y)),
    \end{align}
    where on the right-hand side, $\sigma$  is taken to have weight 1 (as in Theorem
    \ref{thm:MultIndependentOfChoice}), and $i(\Gamma, \Trop(\widetilde{X}) \cdot\sigma;\Trop(Y))$ is the local tropical intersection multiplicity of \cite{OssermanPayne2013}, as described in Section \ref{sec:BackgroundTropicalIntersectionTheory}. Note that this multiplicity is computed in ambient space $\eta$.
\end{thm}

\begin{remark}\label{rem:discussion-of-thm-E} \
\begin{enumerate}

\item The assumption that $\Gamma$ is in the relative interior of a facet of $\Trop(Y)$ of weight $1$ is necessary, see \cite[Examples 6.3 and 6.4]{OssermanPayne2013}. In their language, $\Gamma$ must be a \emph{simple point} of $\Trop(Y)$.

\item At one extreme, if $Y = \bfT$, condition (1) is vacuous (but note that not every $X \subset \M$ can be expressed as a sufficiently transverse intersection $\widetilde X\cap\M$). In this case, the statement of Theorem \ref{thm:intro1}, applied to $X$, can be deduced from \cite[Thm. 10.1]{Katz2009}. Specifically, conditions (2) and (3) imply that the limit cycle of $\widetilde X$ is a sum of \emph{toric} boundary strata, and using that fact that specialization commutes with refined intersection product, one may verify that the limit cycle of $X$ is the corresponding sum of boundary strata of $\M$.\label{it:UseKatz}

\item At the other extreme, one might hope to deduce Theorem \ref{thm:mainthm} directly by taking $Y = \M$ and letting $\widetilde{X} = X$, since then condition (3) is vacuous. But this fails precisely because a point $\Gamma \in \Sigma_r$ is essentially never a simple point of $\Trop(\M)$.

\item In the theorem statement, we can remove the assumption that $\M$ and $\widetilde X$ are Cohen-Macaulay, provided we replace $X$ in the statement with the cycle-theoretic intersection $[\widetilde X]\cdot[\M]$. We have stated the theorem as above because in our setup, a subscheme $X$ is given from the beginning.
\end{enumerate}
\end{remark}

\begin{proof}

    By passing to a sufficiently small neighborhood of $\Gamma$ (and translating $\Gamma$ to $\vec{0}$), we may assume $\eta = \Trop(Y)$ is a linear space of dimension $n'$, and $\sigma$ and $\Trop(\widetilde X)$ are linear subspaces of $\eta$ of dimensions $r$ and $n'-r$ that meet transversely (in $\eta$). We may further assume that $\Trop(X)$ is a cone complex, and that $\sigma$ is in the lineality space of $\Sigma$. In particular, $\Star_\Gamma \sigma = \sigma$. We write $\tilde \kappa$ for the linear space $\Trop(\widetilde{X})$.

    Let $\zeta \in \Sigma$ be a maximal cone containing $\vec{0}$. Since $\tilde \kappa$ and $\sigma$ are linear spaces intersecting in the expected dimension in $\eta$, $\tilde \kappa$ intersects $\zeta$ in a cone $\kappa$ of the expected dimension. By \cite[Cor 5.1.3]{OssermanPayne2013} (noting that $\widetilde X$ and $\M$ are Cohen-Macaulay and that $\Gamma$ is a simple point of $\eta = \Trop(Y)$), $\kappa$ is a facet of $\Trop(X)$, of multiplicity \begin{align}\label{eq:MultiplicityOfKappa}
        m_{\Trop(X);\kappa}=m_{\tilde \kappa;\Gamma}\cdot m_{\Sigma;\zeta}\cdot[L_{\Z}(\eta) : L_{\Z}(\zeta)+L_{\Z}(\tilde \kappa)].
    \end{align}

    We first evaluate the left-hand side of \eqref{eq:XTilde2}. Let $v_\zeta \in \zeta$ be a generic vector. By definition
    \begin{align*}
    \mult_{\Gamma}(\Trop(X),\sigma;\Sigma) &=
    \frac{1}{m_{\Trop(\M);\zeta}} \deg ((v_\zeta+\Star_\Gamma\sigma)\cdot\Star_\Gamma\Trop(X)) \\
    &=
    \frac{1}{m_{\Sigma;\zeta}}\deg((v_\zeta+\sigma)\cdot \Trop(X)),
    \end{align*}
    where the intersection product takes place in ambient space $\zeta$, and where $\sigma$ is assigned weight 1. The intersection $(v_\zeta+\sigma)\cap\Trop(X)$ is a single transverse intersection point of the affine-linear space $(v_\zeta+\sigma)$ and the cone $\kappa$. Thus
    \begin{align*}
    \mult_{\Gamma}(\Trop(X),\sigma;\Sigma) &=\frac{1}{m_{\Sigma;\zeta}} \cdot m_{\Trop(X),\kappa}\cdot[L_{\Z}(\zeta):L_{\Z}(v_\zeta+\sigma)+ L_{\Z}(\kappa)]\\
    &=\frac{1}{m_{\Sigma;\zeta}} \cdot m_{\Trop(X),\kappa}\cdot[L_{\Z}(\zeta):L_{\Z}(\sigma)+ L_{\Z}(\kappa)].
    \end{align*}
    Substituting in \eqref{eq:MultiplicityOfKappa}, we have
    \begin{align}
      \mult_{\Gamma}(\Trop(X),\sigma;\Sigma)  
    &=m_{\tilde\kappa;\Gamma}\cdot[L_{\Z}(\eta) : L_{\Z}(\zeta)+L_{\Z}(\tilde \kappa)]\cdot[L_{\Z}(\zeta):L_{\Z}(\sigma)+ L_{\Z}(\kappa)]\nonumber\\
    &=m_{\tilde\kappa;\Gamma}\cdot[L_{\Z}(\eta) : L_{\Z}(\zeta)+L_{\Z}(\tilde \kappa)]\cdot[L_{\Z}(\zeta):L_{\Z}(\sigma)+ (L_\Z(\tilde \kappa)\cap L_{\Z}(\zeta))]\label{eq:SubstituteMultiplicityOfKappa}\\
    &=m_{\tilde\kappa;\Gamma}\cdot[L_{\Z}(\eta) : L_{\Z}(\zeta)+L_{\Z}(\tilde \kappa)]\cdot[L_{\Z}(\zeta):(L_{\Z}(\sigma)+ L_{\Z}(\tilde \kappa))\cap L_{\Z}(\zeta)].\nonumber
\end{align}
where the last line follows from $L_\Z(\sigma) \subseteq L_\Z(\zeta)$ and the equality $(A+B) \cap C = A + (B \cap C)$ when $A \subseteq C$. By the second isomorphism theorem for groups, we have 
    \[
    \frac{L_{\Z}(\zeta)}{(L_{\Z}(\sigma)+ L_{\Z}(\tilde \kappa))\cap L_{\Z}(\zeta)}
    \cong \frac{L_{\Z}(\zeta)+(L_{\Z}(\sigma)+ L_{\Z}(\tilde \kappa))}{L_{\Z}(\sigma)+ L_{\Z}(\tilde \kappa)}
    = \frac{L_{\Z}(\zeta)+L_{\Z}(\tilde \kappa)}{L_{\Z}(\sigma)+ L_{\Z}(\tilde \kappa)},
    \]
    so \eqref{eq:SubstituteMultiplicityOfKappa} gives
    \begin{align*}
    \mult_{\Gamma}(\Trop(X),\sigma;\Sigma)&=m_{\tilde \kappa;\Gamma}\cdot[L_{\Z}(\eta):L_{\Z}(\zeta)+L_{\Z}(\tilde \kappa)]\cdot[L_{\Z}(\zeta)+L_{\Z}(\tilde \kappa):L_{\Z}(\sigma)+ L_{\Z}(\tilde \kappa)]\\
    &=m_{\tilde \kappa;\Gamma}\cdot[L_{\Z}(\eta):L_{\Z}(\sigma)+ L_{\Z}(\tilde \kappa)]\\
    &=i(\Gamma, \tilde \kappa \cdot\sigma;\eta) \\
    &=i(\Gamma, \Trop(\widetilde{X}) \cdot \sigma;\Trop(Y)).\qedhere
    \end{align*}
\end{proof}

\subsection{Main result: The limit cycle of a degeneration in a tropical compactification}\label{sec:ProofOfMainTheorem}

We now replace the local assumptions of the previous section by global assumptions on $X$. 

\begin{condition}\label{cond:GlobalIntersectionCondition}
Let $X$ be as in Notation \ref{not:DegenerationSetup}. We suppose Assumption~\ref{ass:MSigmaNice} holds, and that $\Trop(X)$ intersects the $r$-skeleton $\Sigma_r$ in finitely many points, all contained in the relative interior $\relint(\Sigma_r).$
\end{condition}

\begin{prop}\label{prop:Set-Theoretic}
    Let $X$ satisfy Condition \ref{cond:GlobalIntersectionCondition}. Then the limit scheme $(\bar X)_0$ is supported on the union of the codimension-$r$ boundary strata $\bigcup_{\sigma}\overline{\M}_\sigma \subseteq \Mbar$, where $\sigma$ ranges over the $r$-dimensional cones of $\Sigma$.
\end{prop}
\begin{proof}
    We may assume without loss of generality that $X$ is irreducible.

    Let $V\subseteq(\bar X)_0$ be an irreducible component not contained in the codimension-$r$ boundary strata of $\Mbar.$ Let $L$ denote the function field of $\X$. Then $V$ defines a divisorial valuation on $L$, which extends the valuation on $K$ (possibly after taking a positive multiple). By Proposition \ref{prop:TropicalizationIntersectsOrbits}, the tropicalization of the $L$-valued point $\Spec(L)\to\bfT$ lies in $\Trop(X)\cap\Sigma_{r-1}$. However, this intersection is empty by assumption, which completes the proof.
\end{proof}

\begin{cor}\label{cor:NoTooBigComponents}
    Let $X$ satisfy Condition \ref{cond:GlobalIntersectionCondition}. Then $X$ has pure codimension $r$, and for every $\Gamma\in\Trop(X) \cap \Sigma_r$, $\Trop(X)$ and  $\Gamma$ satisfy  Condition \ref{cond:SigmaGamma}.  
\end{cor}

\begin{proof}
    We have assumed (in Notation \ref{not:DegenerationSetup}) that $X$ has no irreducible components of codimension $>r$, and Proposition \ref{prop:Set-Theoretic} implies that $(\bar X)_0$, and therefore (by properness of $\Mbar$) $X$ itself, has no irreducible components of codimension $<r$, i.e. $X$ has pure codimension $r$. By the structure theorem for tropical varieties, $\Trop(X)$ has pure codimension $r$. Putting this together with the fact that $\Trop(X) \cap \Sigma_r$ is finite shows that Condition \ref{cond:SigmaGamma} is satisfied.
\end{proof}

In particular, the values $d_\Gamma$ of Proposition \ref{prop:GenericallyFinite} and $\mult_\Gamma(\Trop(X), \sigma;\Sigma)$ of Definition \ref{def:MultIndependentOfChoice} are well-defined.

Our main theorem computes the limit cycle of $\bar X$ in terms of the tropical intersection multiplicity of Definition \ref{def:MultIndependentOfChoice}.
\begin{thm}\label{thm:mainthm}
Let $X$ satisfy Condition \ref{cond:GlobalIntersectionCondition}.
    Then \[
    [(\bar X)_0] = \sum_{\sigma \in \Sigma_r} \bigg(\sum_{\Gamma \in \Trop(X) \cap \sigma} \mult_{\Gamma}(\Trop(X),\sigma;\Sigma) \bigg) [\Mbar_\sigma].
    \]
\end{thm}
Theorem \ref{thm:mainthm} follows from Theorem \ref{thm:MultIndependentOfChoice} and the following algebro-geometric lemma:
\begin{lem}\label{lem:mainlem}
    Let $X$ satisfy Condition \ref{cond:GlobalIntersectionCondition}.
    Let $\sigma \in \Sigma_r$ be an $r$-dimensional cone. Then the multiplicity of $[(\bar X)_0]$ along the stratum $\Mbar_\sigma$ is given by:\label{it:FormulaForX0}
    \begin{align}\label{eq:mainlem}
    \mult_{\Mbar_\sigma}([(\bar X)_0]) &= \sum_{\Gamma\in\Trop(X)\cap\sigma}d_\Gamma,
    \end{align}
    where $d_\Gamma = \deg(\init_\Gamma(X) \to \M_\sigma)$ is as in Proposition \ref{prop:GenericallyFinite}.
\end{lem}

We first prove the following, where, for $Q\in\bfT^K$, $\displaystyle\lim_{t\to0}Q$ denotes the special fiber of the closure of $Q$~in~$Y_\Sigma^R$.

\begin{lemma}\label{lem:ProjectionAntiRescue}
    Let $Q \in \bfT^K$ such that $\displaystyle\lim_{t\to0} Q\in \bfT$. Let $\sigma \in \Sigma$ be a cone and let $\rho_\sigma : U_\sigma \to Y_\sigma$ be as in Notation \ref{not:ToricNotation}\eqref{not:ToricAffineProjection}. Then, for all $\Gamma \in \relint(\sigma)$,
    \begin{equation}\label{eqn:a-iff-b}
    \rho_\sigma\big(\lim_{t \to 0} Q\big) = \lim_{t \to 0} (t^\Gamma Q).
    \end{equation}
\end{lemma}
\begin{proof}
    Since $\Gamma \in \relint(\sigma)$, we have $\displaystyle{\lim_{t \to 0} (t^\Gamma Q)} \in Y_\sigma \subset U_\sigma$ by Proposition \ref{prop:TropicalizationIntersectsOrbits}.

    Recall that $\rho_\sigma : U_\sigma^R \to Y_\sigma^R$ is $\bfT_\sigma$-invariant and restricts to the identity on $Y_\sigma^R$. Also, by continuity
    \begin{equation}
    \rho_\sigma(\lim_{t \to 0} Q') = \lim_{t \to 0} \rho_\sigma(Q')
    \end{equation}
    for any $Q' \in U_\sigma^K$ whenever the limit exists in $U_\sigma$ on the left-hand side.
    
    To prove \eqref{eqn:a-iff-b}, we first have $\rho_\sigma(Q) = \rho_\sigma(t^\Gamma Q)$ by $\bfT_\sigma$-invariance and the fact that $t^\Gamma \in \bfT_\sigma$. Taking the limit, we obtain
    \begin{equation}
    \lim_{t \to 0} \rho_\sigma(Q) = \lim_{t \to 0} \rho_\sigma(t^\Gamma Q).
    \end{equation}
    Since $\displaystyle{\lim_{t \to 0} Q \in \bfT}$ and $\displaystyle{\lim_{t\to 0}t^\Gamma Q} \in Y_\sigma \subset U_\sigma$ both exist, by continuity, we can move both limits inside $\rho_\sigma$:
    \begin{equation}
    \rho_\sigma(\lim_{t \to 0} Q) =
    \rho_\sigma(\lim_{t \to 0} t^\Gamma Q).
    \end{equation}
    Finally, we can simplify the right-hand side because $\displaystyle{\lim_{t \to 0} t^\Gamma Q \in Y_\sigma}$ and $\rho_\sigma|_{Y_\sigma}$ is the identity. We conclude
    \[
    \rho_\sigma(\lim_{t \to 0} Q) =
    \lim_{t \to 0} t^\Gamma Q. \qedhere
    \]
\end{proof}

\begin{proof}[Proof of Lemma \ref{lem:mainlem}]

Both sides of \eqref{eq:mainlem} are unaffected by a base change from $K$ to any valued extension of $K$ that induces an isomorphism on residue fields. Therefore, by \cite[Cor. 5]{Poonen1993}, we may assume from now on that $K$ is complete, algebraically closed, has residue field $\mathbf{k}$, and has value group $\R$. We may therefore fix a splitting \cite[Lemma 2.1.15]{MaclaganSturmfels2015}. We note that whenever $\Gamma \in \sigma$, the corresponding element $t^\Gamma$ satisfies $t^\Gamma \in \bfT_\sigma$.

Both sides of Equation \eqref{eq:mainlem} can be calculated on $U_\sigma$, so for the remainder of the proof, we work entirely in $U_\sigma$. Let $\rho_\sigma : U_\sigma \to Y_\sigma$ be the projection, and write $\rho_\sigma|_{\init_\Gamma X} : \init_\Gamma X \to \M_\sigma$ for its restriction for each $\Gamma \in \Trop(X) \cap \sigma$. Note that, by definition, $\init_\Gamma X \subseteq \bfT\subseteq U_\sigma$.

Both sides of Equation \eqref{eq:mainlem} are local on $\M_\sigma$. We pass to a dense open subset as follows.
\begin{itemize}
\item Let $D = \overline{X} \setminus X \subset (\overline{\M} \setminus \M)^K$ and let $D_0$ be its flat limit in $U_\sigma$, i.e. the limit of the boundary of $\overline{X}$. Note that $\dim D \leq \dim X - 1 = \dim \M_\sigma - 1$ and that $\dim D_0 = \dim D$, see e.g. \cite[Section 4]{OssermanPayne2013} and/or \cite[Theorem 12.1.1(i)]{EGA66}. In particular, $\M_\sigma \setminus D_0$ is nonempty. Replacing $Y_\sigma$ by $Y_\sigma \setminus D_0$, $\M_\sigma$ by $\M_\sigma \setminus D_0$, $\bfT$ by $\bfT \setminus \rho_\sigma^{-1}(D_0)$ and $U_\sigma$ by $U_\sigma \setminus \rho_\sigma^{-1}(D_0)$, we may assume the flat limit of $D$ in $U_\sigma$ is empty.
\item Let $X' \subsetneq X$ be the locus where $X$ is not Cohen-Macaulay; note that $X'$ is a closed subscheme of $X$ whose complement is dense in $X$ \cite[\href{https://stacks.math.columbia.edu/tag/00RG}{Tag 00RG}]{stacks-project}. We may similarly assume that the flat limit of $X'$ in $U_\sigma$ is empty.
\item Similarly, by replacing $Y_\sigma$ by a dense open subset $U$ intersecting $\M_\sigma$, replacing $\M_\sigma$ by $\M_\sigma \cap U$, $\bfT$ by $\bfT \cap \rho_\sigma^{-1}(U)$ and $U_\sigma$ by $\rho_\sigma^{-1}(U)$, we may assume that $\M_\sigma$ is smooth and that all the maps $\rho_\sigma|_{\init_\Gamma X}$ are finite and flat. Note that this implies each $\init_\Gamma X$ is Cohen-Macaulay \cite[\href{https://stacks.math.columbia.edu/tag/045J}{Tag 045J}]{stacks-project}.
\end{itemize}
Note that shrinking $Y_\sigma$, $\M_\sigma$, $\bfT$ and $U_\sigma$ in this way retains the action of $\bfT_\sigma$ but loses the action of $\bfT$ (indeed $\bfT$ is no longer itself a torus), and that $\rho_\sigma$ only remains $\bfT_\sigma$-equivariant. We will accordingly only use the action of $\bfT_\sigma$ in the remainder of the proof.

Let $P \in \M_\sigma$ be arbitrary and let $\underline{W} \subset Y_\sigma$ be a smooth $(n-m)$-dimensional subvariety that intersects $\M_\sigma$ transversely at $P$ in an isolated reduced point. By shrinking $Y_\sigma$, $\M_\sigma$, $\bfT$ and $U_\sigma$ again, we may assume $\underline{W} \cap \M_\sigma = P$. Let $W = \rho_\sigma^{-1}(\underline{W})$. We now examine the following schemes:
\begin{itemize}
    \item The intersection $W^R \cap \X$ in $U_\sigma^R$, and
    \item The intersection $W^R \cap \overline{t^{-\Gamma} X}$ in $\bfT^R$ (here $W^R$ denotes $(W \cap \bfT)^R$), for each $\Gamma \in \Trop(X) \cap \sigma$.
\end{itemize} 

By \cite[\href{https://stacks.math.columbia.edu/tag/03QH}{Tag 03QH(4)}]{stacks-project}, since $R$ is a complete local ring, in particular henselian, every finite-type $R$-scheme has finitely many connected components $\mathcal{Z}$, each of which has exactly one of the two properties below:
\begin{enumerate}
    \item every irreducible component of the special fiber of $\mathcal{Z}$ has positive dimension, or else
    \item $\mathcal{Z}$ is finite over $\Spec R$ and its special fiber is a single closed point.
\end{enumerate}
Let $\mathcal{Z} \subseteq W^R \cap \X$ be the connected component containing $P$. We have the following equalities of sets: 
\[(W^R \cap \X)_0 = W \cap X_0 = W \cap \M_\sigma = P,\]
since $X_0 = \M_\sigma$ as a set, and using our hypothesis on $W$. We see that the second conclusion applies; in particular $\mathcal{Z}$ is finite over $\Spec R$. Let $Z$ denote its generic fiber. We note that every point of $Z$ has flat limit $P$; in particular $Z \subseteq X$, since the flat limit of $\overline{X} \setminus X$ in $U_\sigma$ is empty.

Similarly, for each $\Gamma\in\Trop(X)\cap\sigma,$ the special fiber of $W^R \cap \overline{t^{-\Gamma} X}$ is 
\[
W \cap \init_\Gamma X = \rho_\sigma^{-1}(\underline{W}) \cap \init_\Gamma X
=
(\rho_\sigma|_{\init_\Gamma X})^{-1}(\underline{W} \cap \M_\sigma)
=
(\rho_\sigma|_{\init_\Gamma X})^{-1}(P),
\]
as a scheme. Since $\rho_\sigma|_{\init_\Gamma X}$ is finite, this fiber is also finite. Let $\mathcal{V}_\Gamma \subseteq W^R \cap \overline{t^{-\Gamma}X}$ denote the union of all connected components containing points of $(\rho_\sigma|_{\init_\Gamma X}^{-1})(P)$. Then $\mathcal{V}_\Gamma$ is again finite over $\Spec R$, by case (2) above. Let $V_\Gamma$ denote its generic fiber.

We claim that $\mathcal{Z}$ and $\mathcal{V}_\Gamma$ are both flat over $\Spec R$. To see this, we apply \cite[\href{https://stacks.math.columbia.edu/tag/0470}{Tag 0470}]{stacks-project}: since $\mathcal{X}$ and $\overline{t^{-\Gamma} X}$ are themselves flat over $\Spec R$, with special fibers $X_0$ and $\init_\Gamma X$, it suffices to show that the equations for $W$ form a regular sequence at $P \in X_0$ and at each point of $(\rho_\sigma|_{\init_\Gamma X})^{-1}(P) \subset \init_\Gamma X$. This holds since $X_0$ and $\init_\Gamma X$ are Cohen-Macaulay, $W, U_\sigma$ and $\bfT$ are smooth, and the intersections $W\cap X_0\subseteq U_\sigma$ and $W\cap\init_\Gamma X\subseteq \bfT$ are of the expected dimension. Thus $\mathcal{Z}$ and $\mathcal{V}_\Gamma$ are flat over $\Spec R$.

We will compare the generic fibers $Z$ and $V_\Gamma$. We first decompose $Z$ using tropical data. Since $Z\subseteq X,$ we have $\Trop(Z)\subseteq\Trop(X).$ Also, since the special fiber of $\mathcal{Z}$ is contained in $Y_\sigma$, we have $\Trop(Z) \subseteq \sigma$. Since $Z$ is a zero-dimensional subscheme of $\bfT^K$ and $K$ is algebraically closed, it decomposes as a disjoint union
\[
Z = \bigsqcup_{\Gamma \in \Trop(X) \cap \sigma} Z_\Gamma,
\]
where $\Trop(Z_\Gamma)$ is supported at $\Gamma$. (See Figure \ref{fig:rescuing} for an illustration of this decomposition.) In particular, we have
\begin{align}\label{eq:LengthAsSumOverGamma}
\len_K(Z)=\sum_{\Gamma\in\Trop(X)\cap\sigma}\len_K(Z_\Gamma).
\end{align}

\begin{figure}
    \centering

\begin{tikzpicture}[scale=0.75]
  \begin{axis}[
    domain=0:1,
    ymin=0, ymax=3,
    ticks=none,
    axis x line=none,
    axis y line=none,
    clip = true,
    ]
    \addplot[black, thick] {0};
    \addplot[orange, thick] {x^10/3};
    \addplot[blue, thick] {x*x};
    \addplot[blue, thick] {x*x*0.75};
    \addplot[red, thick] {1.3*x};
    \addplot[red, thick] {1.4*x};
    \addplot[red, thick] {1.5*x};
    \addplot[mark=*] coordinates {(0,0)} node {$\bullet$};    
  \end{axis}
  \node[red] at (6.8, 2.65) {$Z_{\Gamma_3}$};
  \node[blue] at (6.8, 1.65) {$Z_{\Gamma_2}$};
  \node[orange] at (6.8, 0.65) {$Z_{\Gamma_1}$};
  \node at (6.8, 0) {$P^R$};
  \node at (0, 0) {$P$};
  \draw (0.57, 3) node[left] {$\rho_\sigma^{-1}(P)$};
  \node at (7.8, 1.65) {\raisebox{-.46\height}{\rotatebox{90}{\makebox[13ex]{\upbracefill}}} $\ Z$};
  \draw (4,-1) node {$\Spec(R)$};
  \draw (1,-1) node {$0\leftarrow t$};
  \draw (0.57,0)--(0.57,3);
\end{tikzpicture}
\quad
\begin{tikzpicture}[scale=0.75]
  \begin{axis}[
    domain=0:1,
    ymin=0, ymax=3,
    ticks=none,
    axis x line=none,
    axis y line=none,
    clip = true,
    ]
    \addplot[black, ultra thick] {0};
    \addplot[orange, thick] {x^3/3};
    \addplot[blue, thick] {1} node[pos=0] {$\bullet$}  node[pos=0.12,above] {$\init_{\Gamma_2}(Z)$};
    \addplot[blue, thick] {0.75} node[pos=0] {$\bullet$};
    \addplot[red, thick, samples=50] {1.3/(x+0.01)^0.1} node[pos=0] {$\circ$};
    \addplot[red, thick, samples=50] {1.4/(x+0.015)^0.1} node[pos=0] {$\circ$};
    \addplot[red, thick, samples=50] {1.5/(x+0.02)^0.1} node[pos=0] {$\circ$};
    \addplot[mark=o, orange, thick] coordinates {(0,0)};
  \end{axis}
  \node[red] at (7.3, 2.65) {$\textcolor{blue}{t^{-\Gamma_2}} Z_{\Gamma_3}$};
  \node[blue] at (7.3, 1.65) {$t^{-\Gamma_2} Z_{\Gamma_2}$};
  \node[orange] at (7.3, 0.65) {$\textcolor{blue}{t^{-\Gamma_2}} Z_{\Gamma_1}$};
  \node at (9.23, 1.65) {\raisebox{-.46\height}{\rotatebox{90}{\makebox[13ex]{\upbracefill}}} $\ \textcolor{blue}{t^{-\Gamma_2}} Z$};
  \draw (4,-1) node {$\Spec(R)$};
  \draw (1,-1) node {$0\leftarrow t$};
  \node at (6.8, 0) {$P^R$};
  \node at (0, 0) {$P$};
  \draw (0.57, 3) node[left] {$\rho_\sigma^{-1}(P)$};
  \draw (0.57,0)--(0.57,3);
\end{tikzpicture}
    
    \caption{{\bf Left:} A finite scheme $Z \subset U_\sigma$ with $\len_K(Z) = 6$ and flat limit supported at $P \in \M_{\sigma}$. Here $\Trop(Z) \cap \sigma$ consists of three distinct points $\{\Gamma_1, \Gamma_2, \Gamma_3\}$, with $\len_K Z_{\Gamma_i} = i$ for each $i$. {\bf Right:} Acting by $t^{-\Gamma_2}$ moves the limits of the points with tropicalization $\Gamma_2$ back into the torus $\bfT \subset U_\sigma$. We have $\init_{\Gamma_2}(Z_{\Gamma_i}) = \varnothing$ for $i \ne 2$, so $\init_{\Gamma_2}(Z) = \init_{\Gamma_2}(Z_{\Gamma_2}) \subseteq \bfT$. }
    \label{fig:rescuing}
\end{figure}
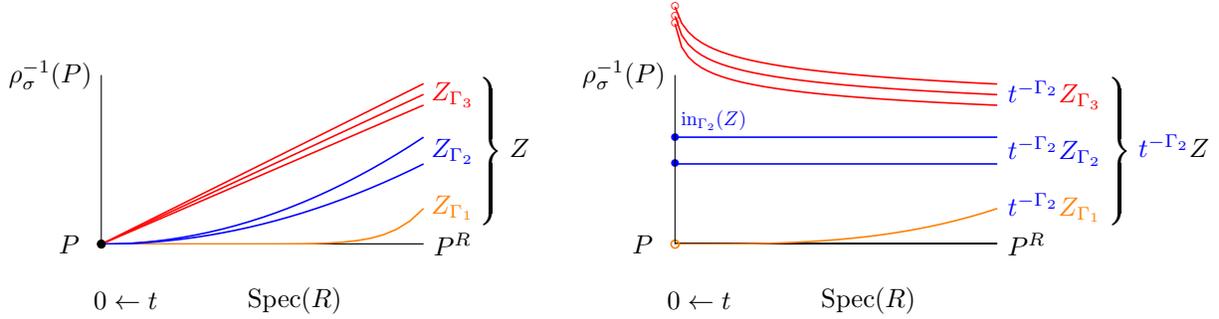

\noindent \textit{Claim 1.} We have $V_\Gamma = t^{-\Gamma} Z_\Gamma$. In particular, $\len_K Z = \sum_\Gamma \len_K(V_\Gamma)$.

\begin{proof}
    Since $W^K$ is $\bfT_\sigma$-invariant and $t^\Gamma \in \bfT_\sigma$, we have
    \[t^{-\Gamma} (W^K \cap X) = W^K \cap t^{-\Gamma} X.\]
    By definition, $Z_\Gamma$ is an open subscheme of $W^K \cap X$, so $t^{-\Gamma}Z_\Gamma$ is an open subscheme of $W^K \cap t^{-\Gamma}X$. Likewise, $V_\Gamma$ is by definition an open subscheme of $W^K \cap t^{-\Gamma}X$. It is therefore sufficient to show $V_\Gamma = t^{-\Gamma} Z_\Gamma$ as sets.

    Consider an arbitrary geometric point $Q \in W^K \cap t^{-\Gamma} X$. By the definitions of $V_\Gamma$ and $t^{-\Gamma} Z_\Gamma$, we have the following characterizations:
    \begin{enumerate}
        \item[(a)] $Q \in V_\Gamma$ if and only if $\displaystyle{\lim_{t \to 0} Q \in \bfT}$ and $\displaystyle{\rho_\sigma\big(\lim_{t \to 0} Q\big) = P}$.
        \item[(b)] $Q \in t^{-\Gamma}Z_\Gamma$ if and only if
        $\displaystyle{\lim_{t \to 0} Q \in \bfT}$ and $\displaystyle{\lim_{t \to 0} (t^\Gamma Q) = P}$.
    \end{enumerate}

    To show (a) iff (b), it is thus sufficient to prove that for all $Q \in \bfT^K$ such that $\displaystyle\lim_{t\to0} Q\in \bfT$,
    \begin{equation}
    \rho_\sigma\big(\lim_{t \to 0} Q\big) = \lim_{t \to 0} (t^\Gamma Q).
    \end{equation}
    This equation holds in the original $U_\sigma$ (prior to passing to an open subset) by Lemma \ref{lem:ProjectionAntiRescue}. In fact, this point lies in our smaller open subset because we deleted a union of fibers of $\rho_\sigma$, and we are assuming $\displaystyle{\lim_{t\to 0}Q}$ exists in $\bfT$ after this deletion.
\end{proof}

\noindent \textit{Claim 2.} $d_\Gamma = \len_K V_\Gamma$.

\begin{proof}
    Since $\mathcal{V}_\Gamma$ is finite and flat over $\Spec R$, $\len_K(V_\Gamma)$ is the same as the length of the special fiber of $\mathcal{V}_\Gamma$, which is $(\rho_\sigma|_{\init_\Gamma X})^{-1}(P)$ (as a scheme). Since $\rho_\sigma|_{\init_\Gamma X}$ is finite and flat, this fiber has length $d_\Gamma = \deg(\rho_\sigma|_{\init_\Gamma X})$.
\end{proof}

\noindent \textit{Claim 3.} $\mult_{\M_\sigma}([X_0]) = \len_K Z$.

\begin{proof}
    We apply Osserman-Payne's specialization map on cycles over $\Spec R$ \cite[Theorem 4.4.5]{OssermanPayne2013}, which generalizes \cite[Sec. 20.3]{Fulton1998} to the case of arbitrary rank-1 valuation rings.
    
    By definition, $Z$ is a union of connected components of $W^K \cap \overline{X}$ intersecting in $U_\sigma^K$ with the correct dimension, and all of whose points have limit $P$. We have $Z \subseteq X$ since the flat limit of $\overline{X} \setminus X$ is empty. Similarly, since the flat limit of the non-Cohen-Macaulay locus of $X$ is empty, $X$ is Cohen-Macaulay along $Z$. (We are using two properties of the dense open subset we passed to at the beginning of the proof.) Finally, $W^K$ is smooth by construction. Thus the intersection cycle agrees with the intersection product along $Z$:
\begin{equation}
    ([W^K \cap \overline{X}])_Z = ([W^K] \cdot [\overline{X}])_Z \in A_*(Z).
\end{equation}
We apply the specialization map to this equation. The left-hand side is $[Z]$. Since $\mathcal{Z}$ is finite and flat over $\Spec R$ and its special fiber is supported at $P$, $[Z]$ specializes to $\len_K(Z) \cdot [P]$.

As for the right-hand side, specialization of cycles commutes with refined intersection product by \cite[Theorem 4.4.5]{OssermanPayne2013}, so the right-hand side specializes to
\begin{equation}\label{eqn:W-times-X0}
    ([W] \cdot [X_0])_P \in A_*(P).
\end{equation}
On $U_\sigma$, by Proposition \ref{prop:Set-Theoretic}, we have $[X_0]=\mult_{\M_\sigma}([X_0])\cdot[\M_\sigma]$. Since $W$ intersects $\M_\sigma$ transversely at $P$, \eqref{eqn:W-times-X0} is thus $\mult_{\M_\sigma}([X_0]) \cdot [P]$. Comparing the specializations gives the claim.
\end{proof}

Combining \eqref{eq:LengthAsSumOverGamma} with Claims 1, 2, and 3, we conclude
\[
\mult_{\M_\sigma}([X_0]) = \sum_\Gamma d_\Gamma,
\]
as desired.
\end{proof}

\begin{remark}
We have assumed that $\Mbar\supseteq\M$ is a tropical compactification, but versions of Proposition \ref{prop:Set-Theoretic}, Corollary \ref{cor:NoTooBigComponents}, and Lemma \ref{lem:mainlem} hold under weaker hypotheses on $\Mbar$, as well as under a weaker version of Condition \ref{cond:GlobalIntersectionCondition}, as we now explain.

Assume that the boundary of $\Mbar$ has combinatorial normal crossings in the sense of Proposition \ref{prop:TropicalCompactificationsBasicProperties}\eqref{eq:TropicalCompactificationStratumCodimension}, but is not necessarily a tropical compactification, i.e. $\Mbar$ need not be proper, and $T\times\Mbar\to Y_\Sigma$ need not be faithfully flat. We obtain the following altered set-theoretic conclusions.
\begin{enumerate}[label=(\alph*)]
    \item \textbf{(Local version of Proposition \ref{prop:Set-Theoretic})} Suppose that for some \emph{fixed} $r$-dimensional cone $\sigma\in\Sigma$, $\Trop(X)$ intersects $\sigma$ in finitely many points, all in $\relint(\sigma)$. (This is a local version of Condition \ref{cond:GlobalIntersectionCondition}.) Then the conclusion of Proposition \ref{prop:Set-Theoretic} holds ``locally near $\M_\sigma,$'' i.e. inside $\Mbar\cap U_\sigma$, the limit scheme is supported on $\M_{\sigma}.$ The proof is unchanged. \label{it:SetTheoreticWorksLocally}
    \item \textbf{(Local version of Corollary \ref{cor:NoTooBigComponents})} Under the same hypothesis as in \ref{it:SetTheoreticWorksLocally}, the conclusion of Corollary \ref{cor:NoTooBigComponents} holds with the following modification. $X$ may have irreducible components of codimension $<r$, but any such component has empty limit in $\Mbar\cap U_\sigma$. Again, the proof is essentially unchanged.
    \item \textbf{(Global version of Corollary \ref{cor:NoTooBigComponents}, without assuming properness of $\Mbar$)} Finally, if $X$ satisfies the global Condition \ref{cond:GlobalIntersectionCondition}, then the limit cycle of $\bar X$ is supported on the codimension-$r$ boundary. ($X$ may have irreducible components of codimension $<r$, but any such component has empty limit in $\Mbar$.)
\end{enumerate}
To obtain the statements about multiplicities, we suppose in addition that $T\times\Mbar\to Y_\Sigma$ is faithfully flat (again without a properness assumption on $\Mbar$); this implies $\Mbar$ has combinatorial normal crossings.
\begin{enumerate}
    \item[(d)] \textbf{(Lemma \ref{lem:mainlem} under local version of Condition \ref{cond:GlobalIntersectionCondition}, without assuming properness of $\Mbar$)} With faithful flatness and the hypothesis as in \ref{it:SetTheoreticWorksLocally},
    the conclusion of Lemma \ref{lem:mainlem} holds. Our proof of Lemma \ref{lem:mainlem} uses properness of $\Mbar$ to rule out irreducible components of $X$ of large dimension, but this can be bypassed by ignoring them, since by the above, they have empty limit in $\Mbar\cap U_\sigma$ and do not contribute to $\mult_{\Mbar_{\sigma}}([(\bar{X})_0]).$
\end{enumerate}
\end{remark}

% BACKGROUND
\part{Application: Computing intersections on \texorpdfstring{$\Mbar_{0,n}$}{M\_{0,n}-bar}}\label{part:Application}
In the rest of the paper, we give an extended application of Theorem \ref{thm:intro1} and \ref{thm:XTilde} to computing limits of cycles on $\Mbar_{0,n}$. We first (Section \ref{sec:TropicalKapranov}) describe the tropicalization of the map to projective space associated to a $\psi$-class, which we call the \emph{tropical Kapranov map} (Definition \ref{def:TropicalPsi}), as well as \emph{tropical $\psi$-hypersurfaces} (Definition \ref{def:BasicHyperplane}). The tropical $\psi$-hypersurfaces we described are tropicalizations of hypersurfaces representing the corresponding $\psi$-classes, and defined over non-trivially valued fields. They are described in terms of path lengths along metric trees, and may be of independent interest. (See Example  \ref{ex:TropicalKapranov}.) We next (Section \ref{sec:TropIntisnice}) introduce our intersection setup --- a collection of $\psi$-hypersurfaces $H_q$ that are sufficiently general with respect to each other. We show that the intersection $\Sigma_r\cap\Trop(\bigcap_{q=1}^r H_q)$ is finite and contained in the relative interior of $\Sigma_r$. This allows us to apply the results of Part \ref{part:MainTheorem}, which we do in Section \ref{sec:RelateToAlgebraic}, proving that the limit cycle of $\bigcap_{q=1}^r \overline{H}_q$ is a sum of boundary strata, each with multiplicity 1. In Section \ref{sec:RecursionandFirework}, we show that $\Sigma_r\cap\Trop(\bigcap_{q=1}^r H_q)$, as $r$ varies, has an elegant recursive structure, and use this structure to give an algorithm to compute $\Sigma_r\cap\Trop(\bigcap_{q=1}^r H_q)$.

\section{Background 2: Moduli of curves, classical and tropical}\label{sec:BackgroundModuli}
For the rest of the paper, let $\k$ be an algebraically closed field of characteristic zero, and let $K=\k((t))$, which we consider as a discretely valued field with the $t$-adic valuation. We write $\bbT$ for the standard 1-torus over $\k$, and $\bbT_K$ for its base change to $K$.

We refer the reader to \cite{CavalieriNotes} for background on $\Mbar_{0,n}$, and to \cite[Sec. 2]{KerberMarkwig2009} and \cite[Secs. 4.3, 6.4]{MaclaganSturmfels2015} for background on $\M_{0,n}^{\trop}$.

\subsection{The moduli space \texorpdfstring{$\Mbar_{0,n}$}{M\_{0,n}-bar}}\label{sec:M0nBar}

For a finite set $S$, we write $\Mbar_{0,S}$ for the moduli space of stable rational curves with marked points indexed by $S$; we write $\M_{0, S} \subseteq \Mbar_{0, S}$ for the open locus of smooth $S$-marked curves. We often take $S=[n]$. 

An $S$-marked stable curve $(C, p_\bullet)$ has a dual tree $\tree$, a \emph{stable $S$-marked tree} (see \cite[Def. 8]{CavalieriNotes}), which has a vertex for each irreducible component of $C$, an edge for each nodal singularity of $C$, and additional \emph{legs} or \emph{half-edges} labeled by $S$, such that each vertex has valence at least 3 (counting legs). For a stable $S$-marked tree $\tree$, we write $\M_{\tree}\subseteq\Mbar_{0,S}$ for the locally closed locus of stable $S$-marked curves with dual tree $\tree$, and we write $\Mbar_{\tree}$ for the closure of $\M_{\tree}$. We have $(C',p_\bullet')\in \Mbar_{\tree}$ if and only if the dual tree $\tree'$ of $(C',p_\bullet')$ admits a sequence of edge contractions resulting in $\tree$. Note that $\M_{0,S}=\M_{\tree_0(S)},$ where $\tree_0(S)$ is the trivial $S$-marked tree with one vertex and no edges.

For any subset $S' \subseteq S$ with $|S'| \geq 3$, there is a forgetful map 
$\pi_{S'} : \Mbar_{0,S} \to \Mbar_{0,S'}$
given by forgetting the marked points $s \notin S'$, then \emph{stabilizing} the curve by repeatedly contracting components corresponding to vertices in the dual tree with valence less than 3. By abuse of notation, we also write $\pi_{S'}(\tree)$ for the resulting tree, obtained analogously by deleting legs and contracting edges of $\tree$. 

\subsection{The tropical moduli space \texorpdfstring{$\M_{0,n}^{\trop}$}{M\_{0, n}\^{}trop}}\label{sec:M0nTrop}

\newcommand{\Mtreetrop}[1]{\M_{0,n}^\trop(#1)}

An $S$-marked stable \emph{tropical curve} $\Gamma$ is a metric space whose underlying topological space is a stable $S$-marked tree $\tree$, obtained by assigning a positive real length $\length(e)$ to each edge of $\tree$, and assigning the legs infinite length. The space of tropical curves with underlying tree $\tree$ is isomorphic to $(\R_{> 0})^{\edges(\tree)}$. We allow edge lengths to go to $0$ by contracting the corresponding edges of $\tree$; we denote by $\M_{0,S}^{\trop}(\tree)\cong (\R_{\ge 0})^{\edges(\tree)}$ the resulting cone. We note that $\M_{0,S}^{\trop}(\tree)$ has a natural integral structure, namely the lattice points $(\Z_{\ge 0})^{\abs{\edges(\tree)}}\subset (\R_{\ge 0})^{\abs{\edges(\tree)}}$ corresponding to metric trees with integer edge lengths.

The moduli space $\M_{0,S}^{\trop}$ of $S$-marked tropical curves is the polyhedral cone complex obtained by taking the disjoint union of all cones $\M_{0,S}^{\trop}(\tree)$, and gluing along faces by identifying a tropical curve with length-zero edges with the tropical curve obtained by contracting those edges, as described above. The cones of $\M_{0,S}^{\trop}$ are in natural bijection with the boundary strata of $\Mbar_{0,S},$ and this bijection switches dimension and codimension.

If $S'\subset S$, there is a forgetful map $\pi_{S'}^{\trop}:\M_{0,S}^{\trop}\to\M_{0,S'}^{\trop}$, defined by
\begin{equation}
\pi_{S'}^\trop(\Gamma) := \text{ the convex hull of the legs in $S'$} \subseteq \Gamma.
\end{equation}
This convex hull is naturally equipped with the structure of an $S'$-marked tropical curve, with underlying $S'$-marked tree $\pi_{S'}(\tree)$, where $\tau$ is the underlying tree of $\Gamma$. The edge lengths of $\pi_{S'}^\trop(\Gamma)$ are sums of edge lengths of $\Gamma$. See Figure \ref{fig:TropForget}. The map $\pi_{S'}^{\trop}$ is a morphism of cone complexes preserving the integral structure. Along with the inclusion $\pi_{S'}^{\trop}(\Gamma)\into\Gamma$, there is a natural deformation retraction $\Gamma\to\pi_{S'}^{\trop}(\Gamma)$, which contracts all legs not in $S'$ (and possibly more).

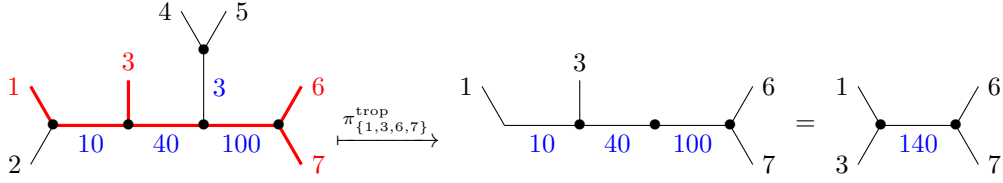
\begin{figure}
    \centering
\begin{tikzpicture}
        \draw[very thick,red] (0,0)--(3,0);
        \draw[very thick,red] (0,0)--++(120:.6) node[left] {1};
        \draw (0,0)--++(240:.6) node[left] {2};
        \draw[very thick,red] (1,0)--++(90:.6) node[above] {3};
        \draw (2,0)--++(90:1);
        \draw (2,1)--++(60:.6) node[right] {5};
        \draw (2,1)--++(120:.6) node[left] {4};
        \draw[very thick,red] (3,0)--++(60:.6) node[right] {6};
        \draw[very thick,red] (3,0)--++(-60:.6) node[right] {7};
        \draw (0,0) node {$\bullet$};
        \draw (1,0) node {$\bullet$};
        \draw (2,0) node {$\bullet$};
        \draw (2,1) node {$\bullet$};
        \draw (3,0) node {$\bullet$};
        \draw[blue] (.5,0) node[below] {10};
        \draw[blue] (1.5,0) node[below] {40};
        \draw[blue] (2.5,0) node[below] {100};
        \draw[blue] (2,0.5) node[right] {3};
        \draw (4.5,0) node {$\xmapsto{\pi_{\{1,3,6,7\}}^{\trop}}{}$};
        \draw (6,0)--(9,0);
        \draw (7,0) node {$\bullet$};
        \draw (8,0) node {$\bullet$};
        \draw (9,0) node {$\bullet$};
        \draw (6,0)--++(120:.6) node[left] {1};
        \draw (7,0)--++(90:.6) node[above] {3};
        \draw (9,0)--++(60:.6) node[right] {6};
        \draw (9,0)--++(-60:.6) node[right] {7};
        \draw[blue] (6.5,0) node[below] {10};
        \draw[blue] (7.5,0) node[below] {40};
        \draw[blue] (8.5,0) node[below] {100};
        \draw (10,0) node {$=$};
        \draw (11,0)--(12,0);
        \draw (11,0) node {$\bullet$};
        \draw (12,0) node {$\bullet$};
        \draw (11,0)--++(120:.6) node[left] {1};
        \draw (11,0)--++(240:.6) node[left] {3};
        \draw (12,0)--++(60:.6) node[right] {6};
        \draw (12,0)--++(-60:.6) node[right] {7};
        \draw[blue] (11.5,0) node[below] {140};
    \end{tikzpicture}
    \caption{Applying the tropical forgetful map $\pi_{\{1, 3, 6, 7\}}^\trop$ by taking the convex hull of the legs $1, 3, 6, 7$ in a tree. Since legs have infinite length, the \textcolor{blue}{10} disappears.}
    \label{fig:TropForget}
\end{figure}

\subsection{Embeddings of \texorpdfstring{$\Mbar_{0,S}$}{M\_{0,S}-bar} and \texorpdfstring{$\M_{0,S}^{\trop}$}{M\_{0, S}\^{}trop}}\label{sec:Embeddings}

Let $\Gr(2,S)^\circ$ denote the open subset of the Grassmannian $\Gr(2, \k^S)$ on which all Pl\"ucker coordinates are nonzero. Sending a $2 \times |S|$ matrix $M = (m_{ij})$ to the tuple of points $[m_{1i} : m_{2i}] \in \mathbb{P}^1$ induces a well-defined map $\Gr(2,S)^\circ \to \M_{0, S}$, which is identified with the quotient map by the column action of $\bbT^S$. Kapranov \cite{Kapranov1993Chow} showed moreover that the $\Mbar_{0, S}$ is identified with the resulting Chow quotient of $\Gr(2, \k^S)$. The Pl\"ucker embedding of $\Gr(2, \k^S)$ then induces an embedding
\[
\Plucker: \M_{0,S} = \Gr(2,S)^\circ/\bbT^S \into \bbT^{\binom{S}{2}}/\bbT^{S},
\] 
where $(\lambda_\ell)_{\ell\in S}\in \bbT^{S}$ acts on the coordinates $p_{ij}$ of $\bbT^{\binom{S}{2}}$ (for $i, j \in S$) by $p_{ij}\mapsto\lambda_i\lambda_jp_{ij}$. 
Explicitly, for $(C, p_\bullet) \in \M_{0, S}$, choose arbitrary coordinates on $\P^1$ and write $(p_\ell)_{\ell\in S}$ as the columns of a $2\times \abs{S}$ matrix. Then $\Plucker(C,p_\bullet)$ is represented by the vector of $2\times 2$ minors of this matrix.

The embedding $\Plucker$ then induces a tropicalization $\Trop(\M_{0,S}) \subset \R^{\binom{S}{2}}/\R^{S}$,
where the quotient is by the image of the map
\begin{align*}
\R^{S} &\to \R^{\binom{S}{2}} \\
(a_i)_{i\in S} &\mapsto (a_i + a_j)_{i,j \in S}.
\end{align*}
In fact, $\Trop(\M_{0,S})$ is naturally isomorphic to $\M_{0,S}^\trop$ as a cone complex with integral structure, via the map
\begin{align*}
\Plucker^{\trop}:
\M_{0,S}^\trop &\xrightarrow{\cong}\Trop(\M_{0,S})\subseteq
\R^{\binom{S}{2}}/\R^{S}, \\
\Gamma &\mapsto -\tfrac{1}{2}\vec{d},
\end{align*}
where $\vec{d}$ is the vector of pairwise distances  $d_{ij}$ between the legs of $\Gamma$. Note that the factor of $\tfrac{1}{2}$ implies that the integer points of $\R^{\binom{S}{2}} / \R^S$ lying on the image of $\M_{0, S}^\trop$ are precisely those coming from metric trees with integer edge lengths, i.e. the natural integer lattices on the cones of $\M_{0, S}^\trop$ and $\Trop(\M_{0, S})$ agree\footnote{This factor is also needed to ensure that $\Plucker^{\trop}$ is compatible with the tropicalization map in non-Archimedean geometry, see \cite[Thm. 1.2.1]{AbramovichCaporasoPayne2012}}, see \cite[Thm. 4.2]{SpeyerSturmfels2004}, \cite[Thm. 3.7]{GathmannKerberMarkwig2009}, and \cite[Sec. 4.3]{MaclaganSturmfels2015}. All maximal cones of $\Trop(\M_{0,S})$ have weight 1 \cite[Cor. 4.3.12]{MaclaganSturmfels2015}.

By \cite{Tevelev2007}, $\Mbar_{0,S}$ is a tropical compactification of $\M_{0, S}$, as we now explain. The embedding $\Plucker^{\trop}$ of the cone complex $\M_{0,S}^{\trop}$ gives a fan $\Sigma(S)\subseteq\R^{\binom{S}{2}}/\R^{S}$ in which cones are in bijection with stable $S$-marked trees. The Zariski closure of $\M_{0,S}$ in the corresponding toric variety $\TV_{\Sigma(S)}$ is isomorphic to $\Mbar_{0,S}.$ Furthermore, the torus orbits $Y_\tree$ in $\TV_{\Sigma(S)}$ are in bijection with stable $S$-marked trees $\tree$, and the boundary stratum $\Mbar_\tree$ is precisely the (transverse) intersection of $\Mbar_{0,S}$ with $Y_\tree$.

\subsection{Kapranov maps}\label{sec:Kapranov}
For any $i \in S$, $\Mbar_{0,S}$ admits a birational map to $\P^{\abs{S}-3}$ called the $i$-th Kapranov map,
induced by the $i$-th \emph{psi (cotangent) class} $\psi_i \in A^1(\Mbar_{0,S})$. To describe this map, it is convenient to fix a second marked point $j \in S\setminus \{i\}$. Let $(C, p_\bullet)$ be an $S$-marked stable curve. Contracting all components except for the one containing $p_i$ yields an $S$-marked irreducible rational curve $(C' \cong \P^1,p_\bullet)$, where markings may coincide with each other, but may not coincide with $p_i$. We coordinatize $C'$ so that $p_i = \infty$ and $p_j = 0$. Up to a common scalar, the remaining points $p_{k}$ have well-defined coordinates $q_k$, not all zero, and we define the \emph{$i$-th Kapranov map, relative to $j$} by
\[
\kap_{i \rel j}(C, p_\bullet) = [q_{k} : k \in S \setminus \{i, j\}] \in \P^{|S|-3},
\]
in which we view $\P^{|S|-3}$ as $\P(\k^S / \k^{\{i, j\}})$.
\begin{remark}\label{rem:KapranovNotWellDefined}
    For $j, j'\ne i$, $\kap_{i \rel j}$ and $\kap_{i \rel j'}$ differ only by post-composing with a linear automorphism of $\P^{\abs{S}-3}$. Indeed, there is a unique $i$-th Kapranov map $\kap_i$, up to change of coordinates, induced by the $i$-th cotangent line bundle. However, the automorphism induced by replacing $j$ with $j'$ effectively changes the choice of torus in $\P^{|S|-3}$. For non-tropical purposes this subtlety is unimportant (e.g. the pullback $\psi_i=\kap_{i\rel j}^*(H)$ of the hyperplane class is independent of $j$) but tropicalization of morphisms is well-behaved only when the morphisms are (locally) restrictions of toric morphisms. Thus, the second marked point will play a role in our algorithms later on, notably in our definition below of a tropical Kapranov map.
\end{remark}

\begin{ex}
Consider the curve below:
\begin{center}
    \includegraphics[scale=1]{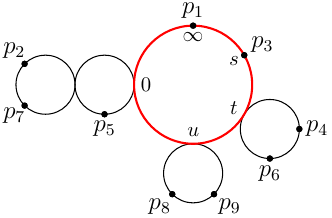}
\end{center}
Up to change of coordinates, all the special points on the non-highlighted components can be taken to be at $0, 1$ or $\infty$. We have $\kap_{1 \rel 2}(C) = [s : t : 0 : t : 0 : u : u]$.
\end{ex}

\section{The tropical Kapranov map and tropical \texorpdfstring{$\psi$}{psi}-hypersurfaces}\label{sec:TropicalKapranov}
In this section we introduce a \emph{tropical Kapranov map} $\Psi_{i\rel j}^\trop:\M_{0,S}^\trop\to\R^{\abs{S}-2}/\R$, in analogy with the Kapranov map $\Psi_{i\rel j}:\Mbar_{0,S}\to\P^{\abs{S}-3}$. 

\subsection{The tropical Kapranov map}\label{sec:Factorizations}

We remark that both $\Mbar_{0,S}$ and $\P^{\abs{S}-3}$ arise as spaces of \emph{weighted} stable curves in the sense of Hassett \cite{Hassett2003}, and the Kapranov map is then one of Hassett's reduction morphisms. These spaces, along with their reduction morphisms, have natural tropicalizations; unfortunately, the tropicalization of $\P^{\abs{S}-3}$ in this sense is a single point, so the tropical reduction morphism loses all information. This loss of information is related to the lack of a natural choice of torus on $\P^{|S|-3}$, as discussed in Remark \ref{rem:KapranovNotWellDefined}.

Thus, we remedy the difficulty by fixing the second marked point $j$ and identifying $\P^{|S|-3}$ with $\P(\k^S / \k^{\{i,j\}})$ with its standard coordinate torus. The Kapranov map now arises as the composition of the tropical compactification $\Mbar_{0, S} \xrightarrow{\Plucker} \TV_{\Sigma(S)}$ with a toric morphism $G_{i, j} : \TV_{\Sigma(S)} \to \P(\k^S / \k^{\{i,j\}})$, yielding a diagram
\begin{align}\label{eq:ToricMorphisms}
    \begin{tikzcd}[ampersand replacement=\&]
    \Mbar_{0,S}\arrow[r,"\Plucker"]\arrow[rr,bend right,"\kap_{i\rel j}"]\&\TV_{\Sigma(S)}\arrow[r,"G_{i,j}"]\& \P(\k^S / \k^{\{i,j\}}) \cong \P^{\abs{S}-3}.
    \end{tikzcd}
\end{align}
On the torus $\bbT^{\binom{S}{2}} / \bbT^S$, we define $G_{i,j}$ by
\begin{equation}\label{eq:ToricMorphism-formula}
(p_{ab} : a,b \in S) \mapsto 
\left(\frac{p_{j\ell}}{p_{i\ell}} : \ell\in S\setminus\{i,j\}\right).
\end{equation}
It is straightforward to see that $G_{i, j}$ is well-defined on the quotient by $\bbT^S$ and that $G_{i, j} \circ \Plucker = \kap_{i \rel j}$ on $\M_{0, S}$. Below, we verify that $G_{i, j}$ extends to a toric morphism.

We thus obtain a linear map
\begin{align}\label{eq:TropicalMorphism-formula}
G_{i,j}^\trop :\R^{\binom{S}{2}}/\R^{S} &\to \R^{S \setminus \{i, j\}}/\R \cong \R^{\abs{S}-3}, \nonumber \\
(p_{ab}^\trop : a,b \in S) &\mapsto \big(p_{j\ell}^\trop - p_{i \ell}^\trop : \ell \in S \setminus \{i, j\} \big),
\end{align}
where we quotient on the right by the all-ones vector $\mathbbm{1} \in \R^{S \setminus \{i, j\}}$. This gives a diagram of tropicalizations:
\begin{align}\label{eq:TropicalMorphisms}
    \begin{tikzcd}[ampersand replacement=\&]
         \M_{0,S}^{\trop}\arrow[r,"\Plucker^{\trop}"]\arrow[rr,bend right,"\kap_{i\rel j}^\trop"]\&\R^{\binom{S}{2}}/\R^{S}\arrow[r,"G_{i,j}^{\trop}"]\&\R^{S \setminus \{i, j\}}/\R.
    \end{tikzcd}
\end{align}

\begin{definition}\label{def:TropicalPsi}
    The \emph{$i$-th tropical Kapranov map relative to $j$} is $\Psi_{i\rel j}^\trop = G_{i,j}^\trop\circ\Plucker^\trop$. Concretely, let $\Gamma$ be a tropical curve. Choose any isometry between $\R$ and the convex hull of legs $i$ and $j$, oriented from $i$ to $j$. For each $\ell \in S \setminus \{i, j\}$, let $d_\ell$ be the coordinate of the nearest point to the $\ell$-th leg in this convex hull. Then
    \begin{equation}\label{eqn:PsiCoords}
    \Psi_{i \rel j}^\trop(\Gamma) := (d_\ell : \ell \in S \setminus \{i, j\}).
    \end{equation}
    We note that $\Psi_{i \rel j}^\trop(\Gamma)$ is defined only up to translation by $\mathbbm{1} \in \R^{S \setminus \{i, j\}}$, which accounts for the choice of isometry. Equation \eqref{eq:TropicalMorphism-formula} naturally sets the attachment points of the legs $i$ and $j$ at $-\tfrac{1}{2}d_{ij}$ and $\tfrac{1}{2}d_{ij}$, where $d_{ij}$ is the distance between the $i$-th and $j$-th legs. Adding $\tfrac{1}{2}d_{ij} \mathbbm{1}$ in \eqref{eq:TropicalMorphism-formula} instead sets the attachment point of leg $i$ at $0$ and that of leg $j$ at $+d_{ij}$.
    \end{definition}

We refer the reader to Example \ref{ex:TropicalKapranov} for an illustration.    

\begin{prop}\label{prop:ToricFactorization}
    The morphism $G_{i,j}$ extends to a toric morphism $\TV_{\Sigma(S)} \to \P(\k^S/\k^{\{i, j\}})$.
\end{prop}
\begin{proof}
    We show that the linear map $G_{i,j}^\trop$ of \eqref{eq:TropicalMorphisms} induces a morphism of fans, i.e. that for each cone $\sigma$ of $\Sigma(S),$ $G_{i,j}^\trop(\sigma)$ is contained in a cone of the fan $\Sigma'$ of $\P(\k^S/\k^{\{i,j\}}).$ The cones of $\Sigma'$ are in bijection with nonempty subsets of $S \setminus \{i, j\}$.
    
    Every point of $\sigma$ is of the form $\Plucker^\trop(\Gamma)$ for some $\Gamma\in\M_{0,S}^\trop$ of a fixed combinatorial type. It is straightforward to verify that $\Psi_{i\rel j}^\trop(\Gamma)$ lies in the cone of $\Sigma'$ corresponding to the following set $\sigma' \subseteq S \setminus \{i, j\}$: the set of $\ell$ for which the nearest point to the convex hull of $i$ and $j$ is the base of the leg $i$. (Equivalently, the path from $\ell$ to $j$ includes the base of the leg $i$.) In particular, $\sigma'$ depends only on the combinatorial type of $\Gamma$, hence only on $\sigma$. Thus $G_{i,j}^\trop(\sigma) \subseteq \sigma'$ as desired.
\end{proof}

\begin{remark}
    The Kapranov map factors through the reduction map from $\Mbar_{0,n}$ to Losev-Manin space, and the map $G_{i, j}$ accordingly factors through Losev-Manin space. 
    As a map of sets, $\Psi_{i\rel j}^\trop$ may be identified with the tropical reduction map from $\M_{0,S}^\trop$ to tropical Losev-Manin space, which is a refinement of the fan of $\P(\k^S/\k^{\{i,j\}})$. With this interpretation, Proposition \ref{prop:ToricFactorization} essentially follows from e.g. \cite[Sec. 3]{CavalieriHampeMarkwigRanganathan2016}.
\end{remark}

\subsection{Kapranov maps and forgetful maps} \label{subsec:KapranovAndForgetful}
We will consider compositions of $\kap_{i\rel j}$ and $\kap_{i\rel j}^\trop$ with forgetful maps. For $S'\subseteq S$ and $i,j\in S'$, we define:
\begin{align}\label{eq:ForgetfulKapranov}
    \kap_{S',i\rel j}:=\kap_{i\rel j}\circ\pi_{S'},
\end{align}
and similarly for the tropical maps. The following is easy to check from the definitions:
\begin{prop}
    The diagrams \eqref{eq:ToricMorphisms} and \eqref{eq:TropicalMorphisms} are compatible with forgetful maps. Precisely, if $S'\subseteq S$ is such that $\abs{S'}\ge3$ and $i,j\in S'$, then we have commutative diagrams:
    \begin{align}\label{eq:ToricMorphismsForget}
    \begin{tikzcd}[ampersand replacement=\&]
         \Mbar_{0,S}\arrow[r,"\Plucker"]\arrow[d,"\pi_{S'}"]\&\TV_{\Sigma(S)}\arrow[r,"G_{i,j}"]\arrow[d,"\pi_{S'}"]\&\P(\k^{S \setminus \{i, j\}})\arrow[d,dashed]\\
         \Mbar_{0,S'}\arrow[r,"\Plucker"]\&\TV_{\Sigma(S')}\arrow[r,"G_{i,j}"]\&\P(\k^{S' \setminus \{i, j\}}),
    \end{tikzcd}
\end{align}
where the vertical arrows (except the leftmost) are toric, and
\begin{align}\label{eq:TropicalMorphismsForget}
    \begin{tikzcd}[ampersand replacement=\&]
         \M_{0,S}^\trop\arrow[r,"\Plucker^\trop"]\arrow[d,"\pi_{S'}^{\trop}"]\&\R^{\binom{S}{2}}/\R^{S}\arrow[r,"G_{i,j}^\trop"]\arrow[d,"\pi_{S'}^{\trop}"]\&\R^{S \setminus \{i, j\}}/\R\arrow[d]\\
         \M_{0,S'}^\trop\arrow[r,"\Plucker^\trop"]\&\R^{\binom{S'}{2}}/\R^{S'}\arrow[r,"G_{i,j}^\trop"]\&\R^{S' \setminus \{i, j\}}/\R,
    \end{tikzcd}
\end{align}
where the vertical maps of vector spaces are linear.
\end{prop}

\subsection{Tropical \texorpdfstring{$\psi$}{psi}-hypersurfaces}\label{sec:Hypersurfaces}

Let $n\ge3$, and let $\bfT=\bbT^{\binom{n}{2}}/\bbT^n$ be the torus into which $\M_{0,n}$ embeds, and let $N$ be its co-character lattice. Let $\bfT^K$ be the base change to $K$. In this section, define and describe certain hypersurfaces in $\M_{0,n}^K$, chosen so that their tropicalizations represent pullbacks of tropical $\psi$-classes along forgetful maps. (See Corollary \ref{cor:HbarRepresentsPsi}.)

\begin{definition}\label{def:BasicHyperplane}
    Fix $S \subseteq [n]$ and $i, j \in S$. For each $\ell \in S \setminus \{i,j\}$, let $f_\ell \in \K$ be a nonzero Laurent series. Assume that the valuations $a_\ell := \val(f_\ell)$ are distinct. We consider the hyperplane
    \begin{align}\label{eq:VeryBasicHyperplane}
    \underline{H} :=\mathbb V\bigg(\sum_{\ell\in S\setminus\{i,j\}} f_\ell x_\ell\bigg)\subseteq \bbT_K^{S \setminus \{i,j\}}/\bbT_K.
    \end{align}
    Note that $\underline{H}$ is nonempty. We write $$\widetilde{H}=(G_{i,j}\circ\pi_S)^{-1}(\underline H)\subseteq\bfT^K\quad\quad\text{and}\quad\quad H=(G_{i,j}\circ\Plucker\circ\pi_S)^{-1}(\underline H)\subseteq\M_{0,n}^K$$ for the scheme-theoretic preimages, see \eqref{eq:ToricMorphismsForget}. We then have 
    \begin{equation}
    \Trop(H) = \Trop(\widetilde{H} \cap \M_{0,n}^K) \subseteq \M_{0,n}^{\trop}\subset\R^{\binom{n}{2}}/\R^n.
    \end{equation}
We say that $\Trop(H)$ is a \textit{tropical $\psi$-hypersurface that represents $\pi_S^*(\psi_i)$.} This terminology is justified by the following Proposition and Corollary.
The key well-behavedness features of these hyperplanes are as follows.

\begin{prop}\label{prop:HypersurfaceHasNoComponentsInBoundary}
    Let $\overline{H} \subseteq \Mbar_{0, n}^K$ and $\overline{\underline{H}} \subseteq \P(\K^{S \setminus \{i, j\}})$ denote closures.
    Then $\overline{H}=(\kap_{S,i\rel j})^{-1}(\overline{\underline{H}})$. In particular, $(\kap_{S,i\rel j})^{-1}(\overline{\underline{H}}) \subseteq\Mbar_{0,n}^K$ has no components contained in the boundary of $\Mbar_{0,n}^K$.
\end{prop}
\begin{proof} 
    The subvariety $\overline{\underline{H}} \subseteq \P(\K^{S \setminus \{i, j\}})$
    does not contain any of the centers of blowups of the birational map $\Psi_{i\rel j}:\Mbar_{0,S}^K\to \P(\K^{S \setminus \{i, j\}})$ by direct computation, since the $f_\ell$ are all nonzero and distinct, cf.~\cite[Lemma 1.2]{GillespieGriffinLevinson2022}. Thus, $\Psi_{i\rel j}^*(\overline{\underline{H}}) = \overline{\Psi_{i\rel j}^*(\underline{H})}$. Since $\kap_{S, i \rel j} = \kap_{i \rel j} \circ \pi_S$, the statement now follows from flatness of $\pi_S$ and the fact that closure commutes with flat pullback.
\end{proof}
\begin{cor}\label{cor:HbarRepresentsPsi}
    $\overline{H}$ is an effective representative for the class $\pi_{S}^*(\psi_i)\in A^1(\Mbar_{0,n}^K).$
\end{cor}

\end{definition}
    Letting $(\underline x_\ell)_{\ell\in S\setminus\{i,j\}}$ denote an element of $\R^{S \setminus \{i,j\}}/\R$, Equation \eqref{eq:VeryBasicHyperplane} tropicalizes to
    \begin{align} \label{eq:DistanceConditionHUnderlineTrop}
    \Trop(\underline{H}) = \{(\underline x_\ell)
    : \min_\ell(\underline x_\ell+ a_\ell)\text{ is achieved twice}\}\subseteq\R^{S \setminus \{i, j\}}/\R.
\end{align}

\begin{lem}\label{lem:SurjectiveMapOfLattices}
    The monomial map $G_{i,j}\circ\pi_S$ of tori corresponds to a surjective map of lattices.
\end{lem}
\begin{proof}
    Given $(y_\ell)_{\ell\in S\setminus \{i,j\}} \in \bbT_K^{S \setminus \{i, j\}}/\bbT_K$, define $x_{j,\ell} \coloneqq y_\ell$ for all $\ell\neq j$, and $x_{a,b}\coloneqq 1$ for all other pairs. Then $G_{i,j}((x_{a,b})_{a<b})_\ell = y_\ell/1 = y_\ell$, so $G_{i,j}$ induces a surjective map of lattices. Furthermore, $\pi_S$ clearly also induces a surjective map of lattices, so the composition $G_{i,j}\circ \pi_S$ does as well.
\end{proof}
Combining Lemma \ref{lem:SurjectiveMapOfLattices}, Lemma \ref{lem:HypersurfacePullback}, and the fact that $\underline{H}$ is a hyperplane (so all facets of $\Trop(\underline{H})$ have weight 1) now yields:
\begin{cor}\label{cor:FacetsOfHTildeHaveWeight1}
    $\Trop(\widetilde H)=(G_{i,j}^{\trop})^{-1}(\Trop(\underline{H}))$, and all facets have weight 1.
\end{cor}

We now describe $(\kap_{S,i\rel j}^{\trop})^{-1}(\Trop(\underline H))=\Trop(\widetilde H)\cap\M_{0,n}^{\trop}$. Suppose $\Gamma\in(\kap_{S,i\rel j}^{\trop})^{-1}(\Trop(\underline H))$. Then the minimum 
\begin{align}\label{eq:MinimumBeforeIntersecting}
    \min_{\ell\in S\setminus\{i,j\}}(\underline d_\ell+a_\ell)
\end{align}
    is achieved at least twice, say by $\ell_1, \ell_2 \in S \setminus \{i,j\}$ with $a_{\ell_1} > a_{\ell_2}$. Then
\[\underline d_{\ell_2} - \underline d_{\ell_1} = a_{\ell_1} - a_{\ell_2},\]
so $a_{\ell_1} - a_{\ell_2}$ is the distance between the attachment points of the legs $\ell_1$ and $\ell_2$ along the path between the legs $i$ and $j$.  In particular, the convex hull of $i,j,\ell_1,\ell_2$ in $\Gamma$ is of the form
\begin{align}\label{eq:ilkjPath0}
    \raisebox{-20pt}{\begin{tikzpicture}[scale=.7]
                \draw[line width=0.5pt,decorate, decoration={snake,amplitude=.35mm,segment length=2.25mm,post length=1mm}] (0,0)--(2,0);
                \draw (0,0) node {$\bullet$};
                \draw (2,0) node {$\bullet$};
                \draw (0,0)--++(180:.6) node[left] {$i$};
                \draw (0,0)--++(90:.6) node[above] {$\ell_1$};
                \draw (2,0)--++(90:.6) node[above] {$\ell_2$};
                \draw (2,0)--++(0:.6) node[right] {$j$,};
                \draw (1,-0.25) node[below] {$P$};
            \end{tikzpicture}
            }
        \end{align}
        where $P$ is a path in $\Gamma$ and $\len(P) = a_{\ell_1} - a_{\ell_2}$. Note that $\ell_1$ is closer to $i$.

\begin{lem}\label{lem:IndependentPerturbations}
    For every cone $\sigma\in\M_{0,n}^{\trop}$, the locus where the minimum \eqref{eq:MinimumBeforeIntersecting} is achieved $r$ times has codimension $r-1$ in $\sigma$.
\end{lem}
\begin{proof}
    If $a_{\ell_1}>\ldots>a_{\ell_r}$ achieve the minimum, then since the $a_{\ell_i}$s are all distinct, the $\underline{d}_{\ell_i}$s are all distinct. That is, the convex hull of $i,j,\ell_1,\ldots,\ell_r$ in $\Gamma$ is of the form:
\begin{align}\label{eq:lengthrPath0}
    \raisebox{-20pt}{\begin{tikzpicture}[scale=.7]
                \draw[line width=0.5pt,decorate, decoration={snake,amplitude=.35mm,segment length=2.25mm,post length=1mm}] (0,0)--(3,0);
                \draw (0,0) node {$\bullet$};
                \draw (1,0) node {$\bullet$};
                \draw (3,0) node {$\bullet$};
                \draw (0,0)--++(180:.6) node[left] {$i$};
                \draw (0,0)--++(90:.6) node[above] {$\ell_1$};
                \draw (1,0)--++(90:.6) node[above] {$\ell_2$};
                \draw (2,0.6) node[above] {$\cdots$};
                \draw (3,0)--++(90:.6) node[above] {$\ell_r$};
                \draw (3,0)--++(0:.6) node[right] {$j$,};
            \end{tikzpicture}
            }
        \end{align}

    Observe that if the position of $\ell_i$ is perturbed to the right, then $\ell_i$ is no longer minimizing. We may perturb $\ell_2, \ldots, \ell_r$ to the right independently modulo translation, which gives an $r-1$-dimensional normal space to the locus where the minimum is achieved $r$ times.
\end{proof}

\begin{lem}\label{lem:PullbackMinimum}
    We have $\Trop(H) = \Trop(\widetilde{H}) \cap\M_{0,n}^\trop$ as weighted polyhedral complexes.
\end{lem}

\begin{proof}
    By definition $H = (\kap_{S,i\rel j})^{-1}(\underline{H})$ and $\Trop(H) = \Trop(\widetilde{H} \cap \M_{0,n})\subseteq \R^{\binom{n}{2}}/\R^n$. Thus, the inclusion 
    \[
    \Trop(H) \subseteq \Trop(\widetilde{H}) \cap \M_{0,n}^\trop
    \]
    is immediate from the definition of tropicalization. For the reverse inclusion, let $\Gamma\in \Trop(\widetilde{H}) \cap \M_{0,n}^\trop$.
    By Lemma \ref{lem:IndependentPerturbations}, $\Trop(\widetilde H) \cap\M_{0,n}^\trop$ has codimension 1 in $\M_{0,n}^\trop$ in a neighborhood of $\Gamma$. Thus by \cite[Thm. 5.1.1]{OssermanPayne2013}, $\Gamma\in \Trop(H).$ This proves that $\Trop(H) = \Trop(\widetilde{H}) \cap \M_{0,n}^\trop$ as sets. Finally, for $\Gamma$ on a facet $\facet$ of $\Trop(H)$, since $\widetilde{H}$ and $\M_{0,n}$ are Cohen-Macaulay, \cite[Cor. 5.1.3]{OssermanPayne2013} implies that the weights of $\Trop(H)$ and $\Trop(\widetilde{H}) \cap \M_{0,n}^\trop$ agree, completing the proof.
\end{proof}
Corollary \ref{cor:FacetsOfHTildeHaveWeight1} and Lemma \ref{lem:PullbackMinimum} imply:
\begin{cor}\label{cor:DistanceCondition}
    \begin{align} \label{eq:DistanceConditionHTrop}
    \Trop(H) &= (\Psi_{S,i \rel j}^\trop)^{-1}(\Trop(\underline{H})) \\ &= \{\Gamma \in \M_{0, n}^\trop : \min_\ell(\underline d_\ell+ a_\ell)\text{ is achieved twice, where } (\underline d_\ell) = \Psi_{i \rel j}^\trop(\Gamma)\}.\nonumber
\end{align}
\end{cor}

We can use Corollary \ref{cor:DistanceCondition} to give local equations for $\Trop(H)$ in coordinates, on cones of $\M_{0,n}^{\trop}.$ Let $\sigma\subset\M_{0,n}^{\trop}$ be a cone, corresponding to a tree $\tau$, and let $\Gamma\in\Trop(H)\cap\relint(\sigma)$ be a point where the minimum \eqref{eq:DistanceConditionHTrop} is achieved exactly twice, say by $\ell_1,\ell_2\in S\setminus\{i,j\}$ with $a_{\ell_1}>a_{\ell_2}.$ Let $v_1,\ldots,v_{r}$ be primitive vectors along the rays of $\sigma$; these are in canonical bijection with the edges $e_1,\ldots,e_r$ of $\tau$, and they generate the lattice $L_{\Z}(\sigma).$ The cone coordinates on $\zeta$ with respect to the basis $v_1,\ldots v_{r}$ are the lengths $(\len(e_1),\ldots,  \len(e_{r}))$. 

Near $\Gamma$ on $\sigma$, $\Trop(H)$ is defined by the equation \begin{align}\label{eq:TropHEquationPathLength}
    \len(P)=a_{\ell_1}-a_{\ell_2},    
\end{align}
where $P$ is the path in \eqref{eq:ilkjPath0}. We may also write this equation as the matrix equation  
\begin{align}\label{eq:TropHMatrixEquation}
    \begin{bmatrix}
    c_1&\cdots&c_{r}
\end{bmatrix}
\cdot\begin{bmatrix}
    \len(e_1)&
    \cdots&
    \len(e_{r})
\end{bmatrix}^T
=\begin{bmatrix}
    a_{\ell_1}-a_{\ell_2}
\end{bmatrix},
\end{align}
where $$c_i=\begin{cases}
    1&e_i\text{ is on the path $P$, and}\\
    0&\text{otherwise.}
\end{cases}$$

\begin{prop}\label{prop:facetsofHtrophaveweightone}
Every facet of $\Trop(H)$ has weight $1$.
    
\end{prop}
\begin{proof}
Let $\kappa$ be a facet of $\Trop(H)$. Combining Lemmas \ref{lem:PullbackMinimum} and \ref{lem:IndependentPerturbations}, a general point $\Gamma\in\kappa$ is in the interior of a maximal cone $\zeta$ of $\M_{0,n}^{\trop}$, and also achieves the minimum in \eqref{eq:DistanceConditionHTrop} exactly twice, say by $\ell_1$ and $\ell_2$ with $a_{\ell_1} > a_{\ell_2}$. 

As above, let $v_1, \ldots, v_{n-3}$ be primitive vectors along the rays of $\zeta$. These span $L_{\Z}(\zeta)$ (see Section \ref{sec:Embeddings}), so we may extend to a basis of $N$. From the equation \eqref{eq:TropHMatrixEquation} for $\Trop(H)$ on $\zeta$ near $\Gamma$, and the fact that locally $\Trop(H)=\Trop(\widetilde H)\cap\zeta,$ we see that the equation for $\Trop(\widetilde H)$ (locally near $\Gamma$ in $N_{\R}$) has the form
\[
\begin{bmatrix}
    c_1&\cdots&c_{n-3}&*&\cdots&*
\end{bmatrix}
\cdot\begin{bmatrix}
    \len(e_1)&
    \cdots&
    \len(e_{r})&
    *&
    \cdots&
    *
\end{bmatrix}^T
=
\begin{bmatrix}
    a_{\ell_1}-a_{\ell_2}
\end{bmatrix},
\]
where $\vec{c} = (c_1, \ldots, c_{n-3})$, as before, is a nonzero vector of $0$s and $1$s.

We then add in the equations for $L_\Z(\zeta) \subset N$ --- these correspond to setting the coefficient of each basis vector other than the $v_i$s to zero. Combining, we obtain the block matrix equation for $\Trop(H)$ in $N_{\R}$ (locally near $\Gamma$):
\[
\begin{bmatrix}
   \vec{c} & \vec{*}\\
       \vec{0}  & I
\end{bmatrix}
\begin{bmatrix}
   \vec{e} \\ \vec{*}
\end{bmatrix}
= \begin{bmatrix}a_{\ell_1} - a_{\ell_2}\\ \vec{0}\end{bmatrix},
\]
   where $I$ is an identity matrix. The rows of the leftmost matrix span a saturated sublattice of $N^\vee$. Since $\zeta$ has weight $1$ \cite[Cor. 4.3.12]{MaclaganSturmfels2015}, and facets of $\Trop(\widetilde H)$ have weight $1$ by Lemma \ref{cor:FacetsOfHTildeHaveWeight1}, the tropical intersection multiplicity of $\Trop(\widetilde H)$ and $\zeta$ along $\kappa$ is 1, see \eqref{eq:TropIntMultDef2}. We conclude that $\Trop(H)$ has weight 1 along $\kappa$.
\end{proof}

\subsection{An example of the tropical Kapranov map and a tropical \texorpdfstring{$\psi$}{psi}-hypersurface}
\begin{ex}\label{ex:TropicalHypersurface}
    Let $S=[5],$ $i=1,$ $j=2,$ and let $f_\ell = t^{3\ell}$ for each $\ell$. By Equations \eqref{eq:VeryBasicHyperplane} and \eqref{eq:DistanceConditionHTrop}, we have
    \begin{align*}
    \underline H &= V(t^9x_3+t^{12}x_4+t^{15}x_5)\subseteq\bbT_K^3/\bbT_K, \\
    \Trop(\underline H) &= \{(\underline{x}_3,\underline{x}_4,\underline{x}_5):\min\{\underline x_3+9,\underline x_4+12,\underline x_5+15\}\text{ achieved twice}\}\subseteq\R^3/\R, \\
    \Trop(H) &= \{\Gamma : \kap_{1 \rel 2}(\Gamma) \in \Trop(\underline{H})\} \subseteq \M_{0, 5}^\trop.
    \end{align*}
    The second and third loci are drawn in red in Figures \ref{fig:LM5Plane} and \ref{fig:LM5-3D}, respectively. With coordinates $x = \underline{x}_4 - \underline{x}_3$ and $y = \underline{x}_5 - \underline{x}_3$ on $\mathbb{R}^2$, the tropical Kapranov map becomes
    \begin{alignat*}{3}
        \M_{0,5}^\trop &\quad\to\quad& \R^3/\R &\quad\to\quad& \R^2\\
        \Gamma&\quad\mapsto\quad& (d_3,d_4,d_5)/\sim&\quad\mapsto\quad& (d_4-d_3,d_5-d_3),
    \end{alignat*}
    In $(x,y)$ coordinates, the tropical hyperplane $\Trop(\underline{H})$ (pictured in red in Figure \ref{fig:LM5Plane}) is the locus where $\min\{0,x+3,y+6\}$ is achieved twice.
    
    For example, suppose $\Gamma$ has the combinatorial type
    \[
    \begin{tikzpicture}
        \draw (0,0)--(2,0);
        \draw (0,0) node {$\bullet$};
        \draw (1,0) node {$\bullet$};
        \draw (2,0) node {$\bullet$};
        \draw (0,0)--++(180:.5) node[left] {1};
        \draw (0,0)--++(90:.5) node[above] {5};
        \draw (1,0)--++(90:.5) node[above] {3};
        \draw (2,0)--++(90:.5) node[above] {4};
        \draw (2,0)--++(0:.5) node[right] {2};
    \end{tikzpicture}.
    \]
    Then the condition $\Gamma \in \Trop(H)$ is precisely $0 = y + 6$, since we have $x + 3 = d_4-d_3+3 > 0$ from the combinatorial type. This corresponds to the ray through the blue region of Figure \ref{fig:LM5Plane}.
    
    The locus $\Trop(H) \subseteq M_{0, n}^\trop$ has $3$ additional linear segments, as shown in Figure \ref{fig:LM5-3D}.
\end{ex}

\begin{figure}
    \centering
    \begin{tikzpicture}[scale=0.8]
    \filldraw[red,opacity=.1] (0,0)--(5,0)--(5,5);
    \draw (3.5,1.5) node {
        \begin{tikzpicture}[scale=.5]
            \draw (0,0)--(2,0);
            \draw (0,0) node {$\bullet$};
            \draw (1,0) node {$\bullet$};
            \draw (2,0) node {$\bullet$};
            \draw (0,0)--++(180:.5) node[left] {1};
            \draw (0,0)--++(90:.5) node[above] {3};
            \draw (1,0)--++(90:.5) node[above] {5};
            \draw (2,0)--++(90:.5) node[above] {4};
            \draw (2,0)--++(0:.5) node[right] {2};
        \end{tikzpicture}
        };
        \filldraw[blue,opacity=.1] (0,0)--(0,5)--(5,5);
    \draw (1.5,4) node {
        \begin{tikzpicture}[scale=.5]
            \draw (0,0)--(2,0);
            \draw (0,0) node {$\bullet$};
            \draw (1,0) node {$\bullet$};
            \draw (2,0) node {$\bullet$};
            \draw (0,0)--++(180:.5) node[left] {1};
            \draw (0,0)--++(90:.5) node[above] {3};
            \draw (1,0)--++(90:.5) node[above] {4};
            \draw (2,0)--++(90:.5) node[above] {5};
            \draw (2,0)--++(0:.5) node[right] {2};
        \end{tikzpicture}
        };
        \filldraw[green,opacity=.1] (0,0)--(0,5)--(-5,5)--(-5,0);
    \draw (-3,3) node {
        \begin{tikzpicture}[scale=.5]
            \draw (0,0)--(2,0);
            \draw (0,0) node {$\bullet$};
            \draw (1,0) node {$\bullet$};
            \draw (2,0) node {$\bullet$};
            \draw (0,0)--++(180:.5) node[left] {1};
            \draw (0,0)--++(90:.5) node[above] {4};
            \draw (1,0)--++(90:.5) node[above] {3};
            \draw (2,0)--++(90:.5) node[above] {5};
            \draw (2,0)--++(0:.5) node[right] {2};
        \end{tikzpicture}
        };
        \filldraw[orange,opacity=.1] (0,0)--(-5,0)--(-5,-5);
    \draw (-3.5,-1.5) node {
        \begin{tikzpicture}[scale=.5]
            \draw (0,0)--(2,0);
            \draw (0,0) node {$\bullet$};
            \draw (1,0) node {$\bullet$};
            \draw (2,0) node {$\bullet$};
            \draw (0,0)--++(180:.5) node[left] {1};
            \draw (0,0)--++(90:.5) node[above] {4};
            \draw (1,0)--++(90:.5) node[above] {5};
            \draw (2,0)--++(90:.5) node[above] {3};
            \draw (2,0)--++(0:.5) node[right] {2};
        \end{tikzpicture}
        };
        \filldraw[yellow,opacity=.1] (0,0)--(0,-5)--(-5,-5);
    \draw (-2,-4.5) node {
        \begin{tikzpicture}[scale=.5]
            \draw (0,0)--(2,0);
            \draw (0,0) node {$\bullet$};
            \draw (1,0) node {$\bullet$};
            \draw (2,0) node {$\bullet$};
            \draw (0,0)--++(180:.5) node[left] {1};
            \draw (0,0)--++(90:.5) node[above] {5};
            \draw (1,0)--++(90:.5) node[above] {4};
            \draw (2,0)--++(90:.5) node[above] {3};
            \draw (2,0)--++(0:.5) node[right] {2};
        \end{tikzpicture}
        };
        \filldraw[cyan,opacity=.1] (0,0)--(0,-5)--(5,-5)--(5,0);
    \draw (3,-3) node {
        \begin{tikzpicture}[scale=.5]
            \draw (0,0)--(2,0);
            \draw (0,0) node {$\bullet$};
            \draw (1,0) node {$\bullet$};
            \draw (2,0) node {$\bullet$};
            \draw (0,0)--++(180:.5) node[left] {1};
            \draw (0,0)--++(90:.5) node[above] {5};
            \draw (1,0)--++(90:.5) node[above] {3};
            \draw (2,0)--++(90:.5) node[above] {4};
            \draw (2,0)--++(0:.5) node[right] {2};
        \end{tikzpicture}
        };
        \draw (6.5,0) node {
        \begin{tikzpicture}[scale=.5]
            \draw (0,0)--(1,0);
            \draw (0,0) node {$\bullet$};
            \draw (1,0) node {$\bullet$};
            \draw (0,0)--++(180:.5) node[left] {1};
            \draw (0,0)--++(60:.5) node[above] {5};
            \draw (0,0)--++(120:.5) node[above] {3};
            \draw (1,0)--++(90:.5) node[above] {4};
            \draw (1,0)--++(0:.5) node[right] {2};
        \end{tikzpicture}
        };
        \draw (0,6) node {
        \begin{tikzpicture}[scale=.5]
            \draw (0,0)--(1,0);
            \draw (0,0) node {$\bullet$};
            \draw (1,0) node {$\bullet$};
            \draw (0,0)--++(180:.5) node[left] {1};
            \draw (0,0)--++(60:.5) node[above] {4};
            \draw (0,0)--++(120:.5) node[above] {3};
            \draw (1,0)--++(90:.5) node[above] {5};
            \draw (1,0)--++(0:.5) node[right] {2};
        \end{tikzpicture}
        };
        \draw (-6.5,0) node {
        \begin{tikzpicture}[scale=.5]
            \draw (0,0)--(1,0);
            \draw (0,0) node {$\bullet$};
            \draw (1,0) node {$\bullet$};
            \draw (0,0)--++(180:.5) node[left] {1};
            \draw (0,0)--++(90:.5) node[above] {4};
            \draw (1,0)--++(120:.5) node[above] {3};
            \draw (1,0)--++(60:.5) node[above] {5};
            \draw (1,0)--++(0:.5) node[right] {2};
        \end{tikzpicture}
        };
        \draw (0,-6) node {
        \begin{tikzpicture}[scale=.5]
            \draw (0,0)--(1,0);
            \draw (0,0) node {$\bullet$};
            \draw (1,0) node {$\bullet$};
            \draw (0,0)--++(180:.5) node[left] {1};
            \draw (0,0)--++(90:.5) node[above] {5};
            \draw (1,0)--++(120:.5) node[above] {3};
            \draw (1,0)--++(60:.5) node[above] {4};
            \draw (1,0)--++(0:.5) node[right] {2};
        \end{tikzpicture}
        };
        \draw (-6,-6) node {
        \begin{tikzpicture}[scale=.5]
            \draw (0,0)--(1,0);
            \draw (0,0) node {$\bullet$};
            \draw (1,0) node {$\bullet$};
            \draw (0,0)--++(180:.5) node[left] {1};
            \draw (0,0)--++(60:.5) node[above] {5};
            \draw (0,0)--++(120:.5) node[above] {4};
            \draw (1,0)--++(90:.5) node[above] {3};
            \draw (1,0)--++(0:.5) node[right] {2};
        \end{tikzpicture}
        };
        \draw (6,6) node {
        \begin{tikzpicture}[scale=.5]
            \draw (0,0)--(1,0);
            \draw (0,0) node {$\bullet$};
            \draw (1,0) node {$\bullet$};
            \draw (0,0)--++(180:.5) node[left] {1};
            \draw (0,0)--++(90:.5) node[above] {3};
            \draw (1,0)--++(120:.5) node[above] {4};
            \draw (1,0)--++(60:.5) node[above] {5};
            \draw (1,0)--++(0:.5) node[right] {2};
        \end{tikzpicture}
        };
        \draw (.75,-.75) node {
        \begin{tikzpicture}[scale=.5]
            \draw (0,0) node {$\bullet$};
            \draw (0,0)--++(180:.5) node[left] {1};
            \draw (0,0)--++(90:.5) node[above] {4};
            \draw (0,0)--++(135:.5) node[above left] {3};
            \draw (0,0)--++(45:.5) node[above right] {5};
            \draw (0,0)--++(0:.5) node[right] {2};
        \end{tikzpicture}
        };
        \draw[very thick] (-5.4,-5.4)--(5.4,5.4);
        \draw[very thick] (-5.4,0)--(5.4,0);
        \draw[very thick] (0,-5.4)--(0,5.4);
        \draw[very thick,red] (-1,-2)--(-1,5);
        \draw[very thick,red] (-1,-2)--(5,-2);
        \draw[very thick,red] (-1,-2)--(-4,-5);
        \draw[red] (-1,-2) node {$\bullet$};
        \draw[red] (-.7,-2.2) node[below] {\tiny $(-3,-6)$};
    \end{tikzpicture}
    \caption{This figure depicts the image of the tropical Kapranov map $\kap_{1 \rel 2} : \M_{0,5}^\trop \to \R^2$. The coordinates on $\R^2=\R^3/\R$ are $(x, y) = (d_4-d_3, d_5-d_3)$, where $d_\ell$ is as in Equation \eqref{eqn:PsiCoords}. We have labeled each region/ray by the the combinatorial type of tree obtained by contracting a tropical curve $\Gamma\in\M_{0,n}^{\trop}$ onto the convex hull of its legs $1$ and $2$. The tropical hyperplane $\Trop(\underline{H})$ of Example \ref{ex:TropicalHypersurface} appears in red. Note that $\M_{0, 5}^\trop$ has fifteen facets: 6 of them map bijectively to the regions above; $6$ others are contracted onto the 6 rays, and $3$ are contracted to the origin. See Figure \ref{fig:LM5-3D} for a three-dimensional view.}
    \label{fig:LM5Plane}
\end{figure}
\begin{figure}
    \begin{center}
    \begin{tabular}{>{\centering\arraybackslash}m{0.45\textwidth} >{\centering\arraybackslash}m{0.45\textwidth}}
  \includegraphics[valign=c, width=0.4\textwidth]{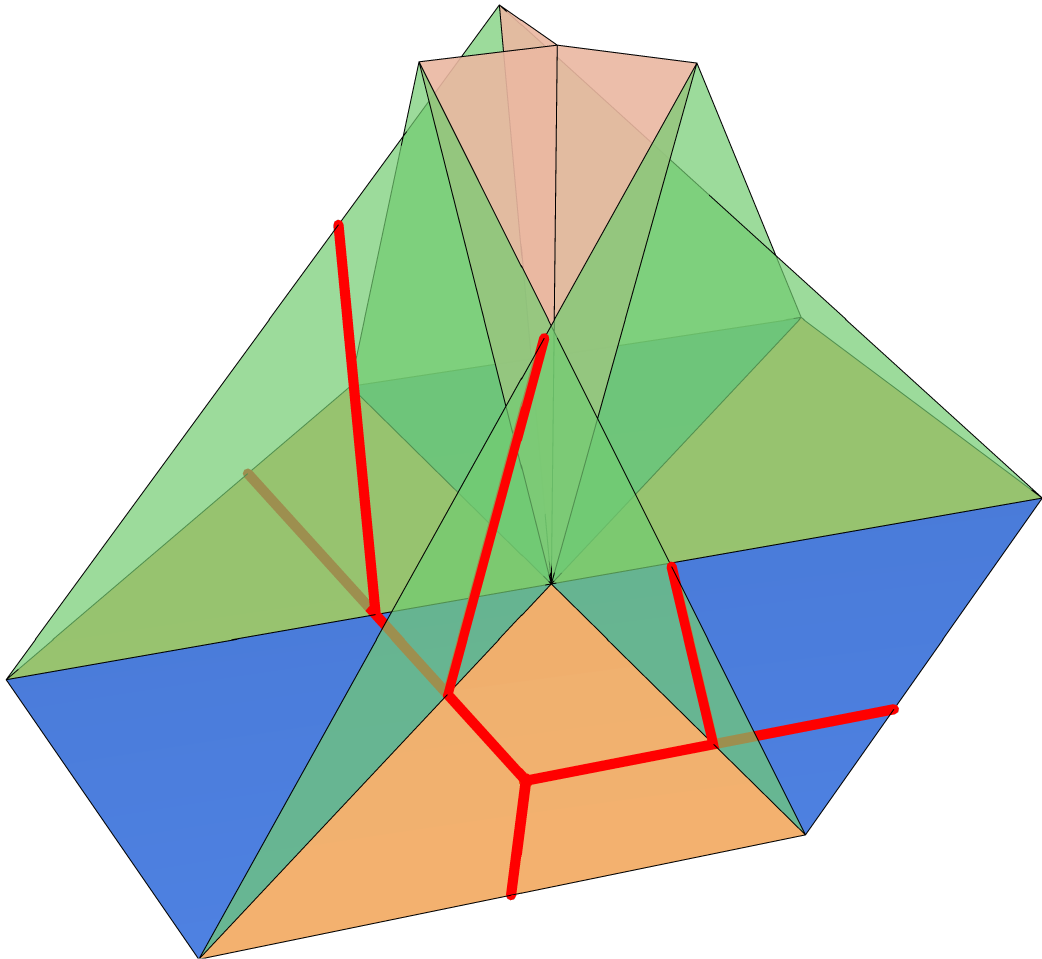} & \includegraphics[valign=c, width=0.2\textwidth]{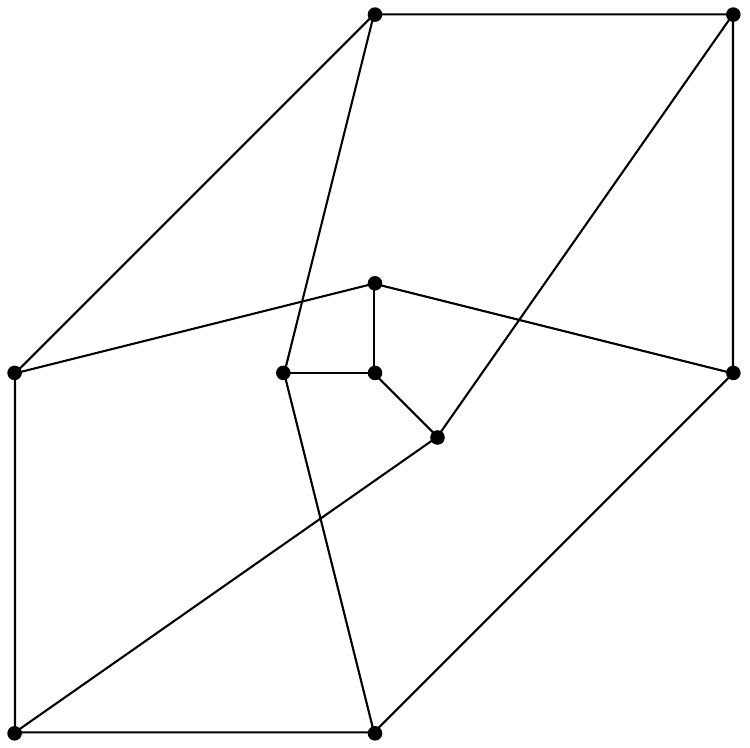} \\
  \includegraphics[valign=c, width=0.4\textwidth]{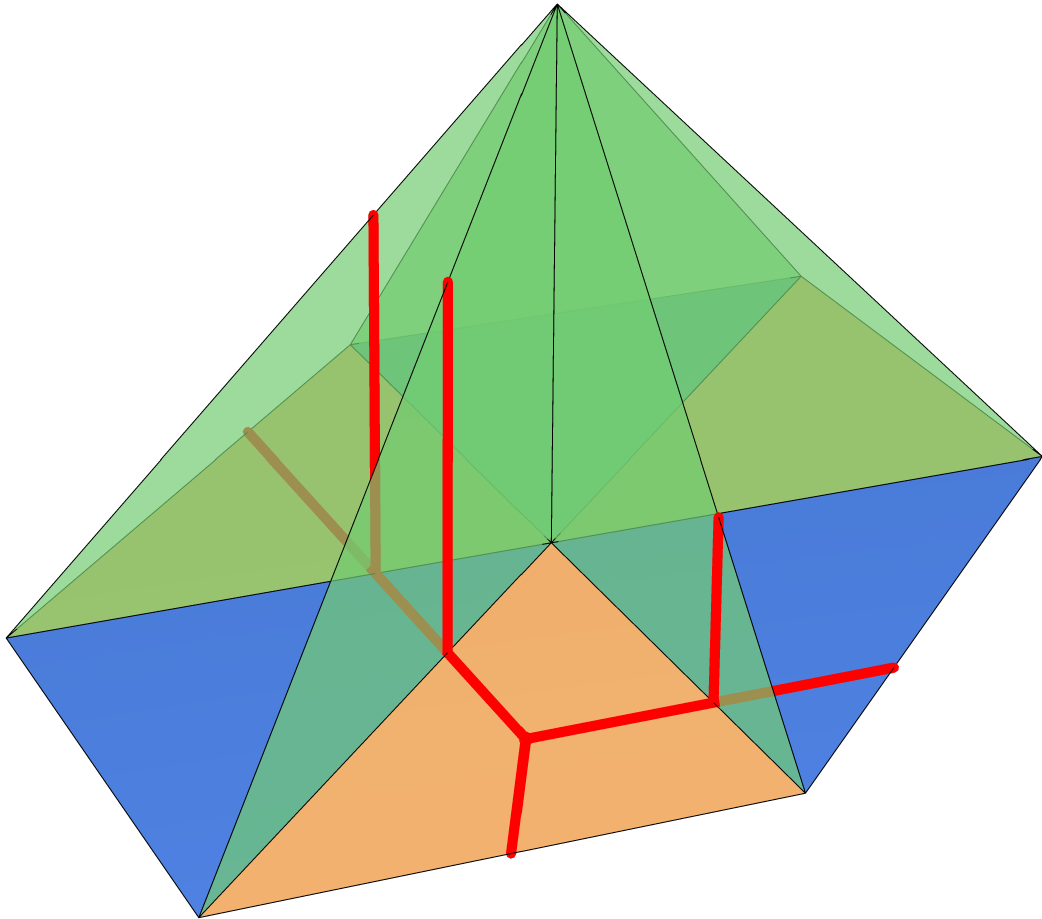} & \includegraphics[valign=c, width=0.2\textwidth]{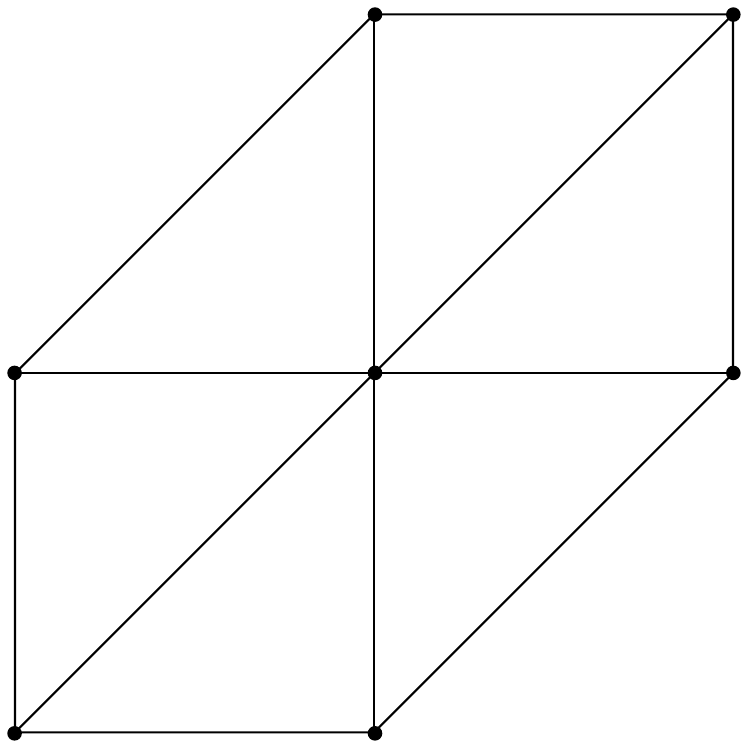} \\
  \includegraphics[valign=c, width=0.4\textwidth]{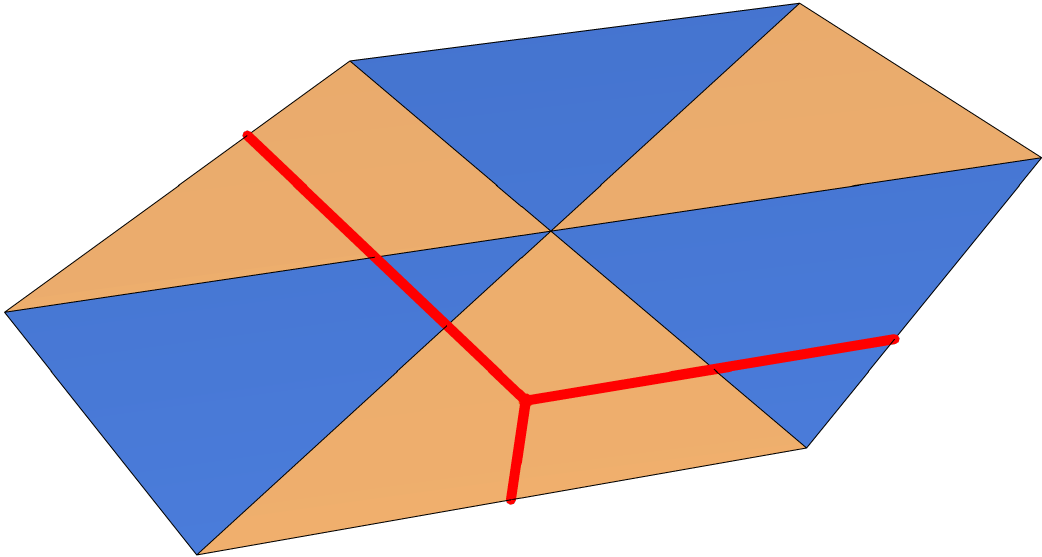} & \includegraphics[valign=c, width=0.2\textwidth]{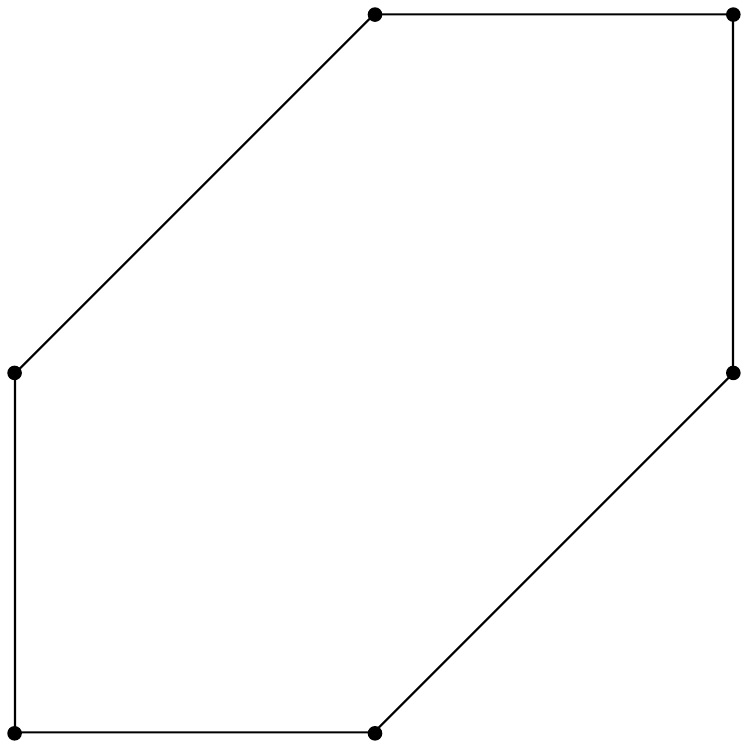} \\
\end{tabular}
    \end{center}
    \caption{
    A three-dimensional view of the tropical Kapranov map $\kap_{1 \rel 2} : \M_{0,5}^\trop \to \R^2$ of Figure \ref{fig:LM5Plane}. Each figure on the left is the cone over the graph to its right. 
    $\M_{0,5}^\trop$ (upper left) is the cone over the Petersen graph (upper right) and is shown immersed in $\R^3$, with $\R^2$ as the $xy$-plane of Figure \ref{fig:LM5Plane}. The tropical hypersurface $\Trop(H)$ (resp.~tropical hyperplane $\Trop(\underline{H})$) of Example \ref{ex:TropicalHypersurface} are shown in red in the upper left (resp.~lower left). The map $\kap_{1 \rel 2}$ can be thought of as contracting the three central cones into the positive $z$-axis (middle left); then projecting the result to the $xy$-plane (lower left) to obtain Figure \ref{fig:LM5Plane}.
    }
    \label{fig:LM5-3D}
\end{figure}

\section{Intersecting tropical \texorpdfstring{$\psi$}{psi}-hypersurfaces}\label{sec:TropIntisnice}

\subsection{The intersection problem}\label{sec:NotationForIntersection}

For the remainder of the paper, we fix the following notation.

\begin{notation}\label{not:Sqiqjq}
 Fix $n\ge3$ and $r\in\{0,\ldots,n-3\}$, along with subsets $S_1,\ldots,S_r\subseteq[n]$ with $\abs{S_q}\ge3$ for each $q=1,\ldots,r$, and distinct elements $i_q,j_q\in S_q$ for each $q$. Fix also an integer $B\gg0$. (It suffices to take $B \geq 2n+1$.) 
 
 For each $q$, we define $H_q$ and $\Trop(H_q)$ as in Definition \ref{def:BasicHyperplane}, using the Kapranov map $\kap_{S_q, i_q \rel j_q}$ and tropical Kapranov map $\kap_{S_q, i_q \rel j_q}^\trop$, taking the coefficient of $x_\ell$ in Equation \eqref{eq:VeryBasicHyperplane} to be the monomial
 \[
 f_\ell := (t^\ell)^{B^{n-3-q}}, \text{ with valuation } \val(f_\ell) = \ell B^{n-3-q}.
 \]
\end{notation}
The purpose of this choice is so that the coefficients of the hypersurface $H_1$ are powers of $t$ that are much larger than those of $H_2$, and so on. If $B=10$ and $r=4$ and $n=7$, for example, then the exponents of $t$ are multiples of $1000$ in $H_1$, multiples of $100$ in $H_2$, of $10$ in $H_3$ and of $1$ in $H_4$.

\subsection{Summary of this section} In this Section, we prove Theorem \ref{prop:RelativeInteriors2} and Corollary \ref{cor:TropIntFinite}, which say in particular that
\[\bigcap_{q=1}^r \Trop(H_q)\subseteq\M_{0,n}^{\trop}\]
intersects the $r$-skeleton $\Sigma_r\subseteq\M_{0,n}^{\trop}$ in finitely many points, contained in the relative interior of $\Sigma_r$, and with tropical intersection multiplicity $1$ at each point (as in Definition \ref{def:MultIndependentOfChoice}). This will allow us, in Section \ref{sec:RelateToAlgebraic}, to apply Theorem \ref{thm:mainthm} from Part 1 to conclude that  
\[\bigcap_{q=1}^r \overline{H}_q\subseteq\Mbar_{0,n}\]
has no irreducible components in the boundary, and that its limit cycle is supported on the codimension-$r$ boundary of $\Mbar_{0,n}$. 

We also prove some technical lemmas that are used in Section \ref{sec:RecursionandFirework} to state a recursive algorithm to explicitly compute $\Sigma_r\cap\bigcap_{q=1}^r \Trop(H_q)$, and therefore also the limit cycle of $\bigcap_{q=1}^r \overline{H}_q$. 

\subsection{Combinatorial descriptions and metric estimates for points of \texorpdfstring{$\TropInt$}{\TropIntPDFString}}

Recall that for $\Gamma\in\M_{0,n}$, 
$$\kap_{S_q,i_q\rel j_q}^\trop(\Gamma)=(\underline d^q_\alpha)_{\alpha\in[n]\setminus\{i_q,j_q\}},$$
 where $\alpha\in S_q\setminus\{i_q,j_q\},$ $\underline{d}^q_\alpha$ is the real coordinate of the leg $\alpha$ on the convex hull of $i_q$ and $j_q$, which is identified with $\R$. (As before, the tuple $(\underline{d}^q_\alpha)_{\alpha\in S_q\setminus\{i_q,j_q\}}$ is defined up to real translation.) 
By Corollary \ref{cor:DistanceCondition}, $\Trop(H_q)$ is the locus where the minimum
\begin{align}\label{eq:Min}
    \min_{\alpha\in S_q\setminus\{i_q,j_q\}}(\underline{d}^q_\alpha+\alpha B^{n-3-q})
\end{align} is achieved at least twice.

 Let $k_q<\ell_q\in S_q\setminus\{i_q,j_q\}$ be any two indices that minimize \eqref{eq:Min}. In particular, $$\underline{d}^q_{k_q}-\underline{d}^q_{\ell_q}=(\ell_q-k_q)B^{n-3-q}.$$ As in \eqref{eq:ilkjPath0}, the convex hull of $i_q, j_q, k_q$ and $\ell_q$ in $\Gamma$ is of the form \begin{align}\label{eq:ilkjPath}
    \raisebox{-20pt}{\begin{tikzpicture}[scale=.7]
        \draw (0,0)--(2,0);
                \draw (0,0) node {$\bullet$};
                \draw (2,0) node {$\bullet$};
                \draw (0,0)--++(180:.6) node[left] {$i_q$};
                \draw (0,0)--++(90:.6) node[above] {$\ell_q$};
                \draw (2,0)--++(90:.6) node[above] {$k_q$};
                \draw (2,0)--++(0:.6) node[right] {$j_q$,};
                \draw (1,0) node[below] {$P_q$};
            \end{tikzpicture}
            }
        \end{align}
where the depicted path $P_q$ has length exactly $(\ell_q-k_q)B^{n-3-q}$.

We now show that if $\Gamma$ is in the $r$-skeleton $\Sigma_r\subseteq\M_{0,n}^{\trop}$, \emph{and} is in the first $r$ tropical hypersurfaces $H_1^{\trop},\ldots,H_1^{\trop}$, then $\Gamma$ must have a tightly constrained form, in which, for each $q$, nearly all the length of the path $P_q$ comes from a single (uniquely determined) edge.

\begin{lem}\label{lem:GammaInTropIntCombinatorics}
    Let $\Gamma\in\TropInt$. Then:
    \begin{enumerate}
        \item \label{it:REdges} $\Gamma$ has exactly $r$ edges.
        \item \label{it:EdgeLengthBounds} For each $q \in [r]$ there is a unique edge $e_q$ with 
        \[
        \tfrac{1}{2} B^{n-3-q} \leq \len(e_q) \leq n B^{n-3-q}.
        \]
        In particular, $\len(e_1) > \len(e_2) > \cdots > \len(e_r).$

        \item \label{it:PqContainsEq} Let $P_q$ be the path in \eqref{eq:ilkjPath}. Then $P_q$ contains $e_q$ and does not contain $e_{q'}$ for any $q'<q$.
        \item For every $q\in[r],$ the minimum of \eqref{eq:Min} is achieved \emph{exactly} twice.\label{it:ExactlyTwice}
    \end{enumerate}
\end{lem}
\begin{proof}
    For each $q = 1, \ldots, r$, let $k_q < \ell_q$ be indices minimizing \eqref{eq:Min}. As above, $\Gamma$ is of the form $\eqref{eq:ilkjPath}$ for each $q$, with $\len(P_q)=(\ell_q-k_q)B^{n-3-q}$. It is convenient to note that, for any $q$,
    \begin{equation} \label{eqn:geom-series}
        \sum_{a=q+1}^r \len(P_a) = \sum_{a = q+1}^r (\ell_a - k_a) B^{n-3-a} \leq \tfrac{n}{B-1} B^{n-3-q} \leq \tfrac{1}{2} B^{n-3-q},
    \end{equation}
    using geometric series, the bounds $\ell_a - k_a \leq n$ all $a$ and the fact that $B \geq 2n+1$.
    
    Using \eqref{eqn:geom-series}, we next calculate
    \begin{equation*}
        \length(P_q)-\sum_{a=q+1}^r\length(P_a)
        \geq (\ell_q - k_q)B^{n-3-q} - \tfrac{1}{2} B^{n-3-q} \geq \tfrac{1}{2} B^{n-3-q},
    \end{equation*}
    since $\ell_q - k_q \geq 1$. In particular, $P_q$ contains an edge $e_q$ not in any of the paths $P_{q+1},\ldots,P_r$. Equivalently, $P_q$ contains $e_q$ but does not contain $e_{q'}$ for any $q' < q$. This produces $r$ edges $e_1,\ldots,e_r$, distinct by definition. Since $\Gamma\in\Sigma_r$, $\Gamma$ has at most $r$ edges, hence exactly $r$ edges, proving \eqref{it:REdges}.
 
    Furthermore, we have
    \[\length(e_q)\le\length(P_q)\le nB^{n-3-q}\]
    and, since $P_q$ consists of $e_q$ and a subset of $e_{q+1}, \ldots, e_r$,
    \[\length(e_q)\ge\length(P_q)-\sum_{a=q+1}^r\length(e_a)\ge\length(P_q)-\sum_{a=q+1}^r\length(P_a)\geq \tfrac{1}{2} B^{n-3-q}.\]
    The ordering of the $\len(e_q)$ now follows by the fact that $n B^{n-3-q} < \tfrac{1}{2} B^{n-3-(q-1)}$ when $B \geq 2n+1$. This proves \eqref{it:EdgeLengthBounds} and \eqref{it:PqContainsEq}.

    To see \eqref{it:ExactlyTwice}, let $\alpha\in S_q\setminus\{i_q,j_q,k_q,\ell_q\}$ and consider the two paths along the convex hull of $i_q$ and $j_q$, connecting the leg $\alpha$ to either $k_q$ or $\ell_q$. If $\alpha$ again minimized \eqref{eq:Min}, then the same reasoning as for \eqref{it:PqContainsEq}, with $k_q$ or $\ell_q$ replaced by $\alpha$, would show that $e_q$ is contained in both these paths. (Note that $e_q$ is uniquely determined by its length, by \eqref{it:EdgeLengthBounds}.) This is impossible, since $e_q$ disconnects $k_q$ and $\ell_q$. This proves \eqref{it:ExactlyTwice}.
\end{proof}
Motivated by Lemma \ref{lem:GammaInTropIntCombinatorics}, we define the following property for a combinatorial (i.e. non-metric) tree together with some combinatorial data.
\begin{definition}\label{def:ConditionStar}
Let $\tree$ be an $[n]$-marked stable tree with exactly $r$ edges, let $\vec e=(e_1>\cdots >e_r)$ be an ordering on the edges of $\tree,$ and let $\vec k=(k_1,\ldots,k_r)$ and $\vec \ell=(\ell_1,\ldots,\ell_r)$ with $k_q<\ell_q\in S_q\setminus\{i_q,j_q\}.$ We say $(\tree,\vec e,\vec k,\vec l)$ \textbf{satisfies condition $(*)$} if for all $q\in[r]$, we have \begin{enumerate}
        \item The convex hull of $i_q, j_q, k_q$ and $\ell_q$ in $\tree$ is of the form \eqref{eq:ilkjPath}, and 
        \item The depicted path $P_q$ contains $e_q$ and does not contain $e_{q'}$ for any $q'<q$.\label{eq:PqContainseq}
    \end{enumerate}
\end{definition}

\begin{cor}
    \label{cor:TropIntGivesConditionStar}
    Let $\Gamma\in\TropInt$. Let $\tau$ denote the underlying tree of $\Gamma$, and for $q\in[r],$ let $e_q,k_q,\ell_q$ as in Lemma \ref{lem:GammaInTropIntCombinatorics}. Then $(\tau,\vec e,\vec k,\vec \ell)$ satisfies condition $(*)$.
\end{cor}

 \subsection{The path matrix}\label{sec:pathmatrix}

\begin{definition}[Path matrices, path-length vectors and edge-length vectors]\label{def:PathMatrix}
        Let $(\tree,\vec e,\vec k,\vec \ell)$ be a tuple satisfying $(*)$, where $\tau$ has $r$ edges. The \emph{path matrix} of $\tree$ is the $r\times r$ matrix $A_\tree=(a_{q,q'})_{q,q'\in[r]}$, where $$a_{q,q'}=\begin{cases}
        1&\text{$e_{q'}$ is on the path $P_q$ in \eqref{eq:ilkjPath}}\\
        0&\text{otherwise}.
    \end{cases}$$
    We define the \emph{path-length vector} $\vec L_\tree$ associated to $\tree$ by $$\vec L_\tree=((\ell_q-k_q)B^{n-3-q})_{q\in[r]}.$$ By Definition \ref{def:ConditionStar}\eqref{eq:PqContainseq}, $A_\tree$ is upper-triangular, with 1s along the diagonal, and 0s and 1s above the diagonal. There is therefore a unique solution $\vec y_\tree=(y_{\tree,q})_{q\in[r]}$ to the matrix equation $$A_\tree \vec y=\vec L_\tree,$$
    with integer entries, which we call the \emph{edge-length vector} associated to $\tau$.
    \end{definition}

In fact, the same calculation as in Lemma \ref{lem:GammaInTropIntCombinatorics}\eqref{it:EdgeLengthBounds} shows:

\begin{lem}\label{lem:ConditionStarImpliesBoundOnEdgeLengths}
        Let $(\tree,\vec e,\vec k,\vec \ell)$ satisfy $(*)$. Then the entries of the edge-length vector $\vec y_\tree$ satisfy
        \[\tfrac{1}{2} B^{n-3-q} \leq y_{\tau, q} \leq n B^{n-3-q}.\]
        In particular, they are positive and $y_{\tree, 1} > \cdots > y_{\tree, r}$.
    \end{lem}

    \begin{proof}
        We proceed by reverse induction on $q$. The calculation is then identical to the one in \ref{lem:GammaInTropIntCombinatorics} (replacing the proofs of Lemma \ref{lem:GammaInTropIntCombinatorics}\eqref{it:REdges} and \eqref{it:PqContainsEq} by the two parts of condition $(*)$).
    \end{proof}

The positivity $y_{\tree, q} > 0$ allows us to define a metric tree $\Gamma$ on $\tree$ by setting $\len(e_q) = y_{\tree, q}$. The following properties hold for $\Gamma$:
\begin{prop}\label{prop:MetricSatisfiesPathEquations}
    Let $(\tree,\vec e,\vec k,\vec \ell)$ satisfy $(*)$ and let $\Gamma$ be the metric tree on $\tree$ by with $\len(e_q)=y_{\tau,q}$. Then:
    \begin{enumerate}
        \item For each $q\in[r]$, the path $P_q$ from \eqref{eq:ilkjPath} has length $(\ell_q-k_q)B^{n-3-q}$. In particular, the $k_q$ and $\ell_q$ terms of Equation \eqref{eq:Min} for $\Gamma$ are equal. \label{it:PathsHaveCorrectLength}
        \item For each $q \in [r]$, we have \[\tfrac{1}{2} B^{n-3-q} \leq \len(e_q) \leq n B^{n-3-q}.\]
        In particular, $\len(e_1) > \len(e_2) > \cdots > \len(e_r)$. \label{it:ConditionStarGivesEdgeOrderingByLength}
        \item Let $q \in [r]$ and $\alpha \in S_q \setminus \{i_1, j_q, k_q, \ell_q\}$ and suppose $\underline d_\alpha^q + \alpha B^{n-3-q}$ is greater than or equal to the common $k_q$ and $\ell_q$ value of \eqref{eq:Min}. Then it in fact exceeds that value by at least $\tfrac{1}{2}B^{n-3-q}$.\label{it:ConditionStarGivesMininequality}
    \end{enumerate}
\end{prop}
\begin{remark}\label{rmk:ConditionStarNotEnough}
    The metric tree $\Gamma$ is not necessarily in $\TropInt$, because the common value of the $k_q$ and $\ell_q$ parts of Equation \eqref{eq:Min} may not be the minimum value. In fact, we will use Proposition \ref{prop:MetricSatisfiesPathEquations} only in Section \ref{sec:RecursionandFirework}, but its proof belongs here with the discussion of the path matrix.
    
    In the mean time, the triangularity and unimodularity of the path matrix turn out to suffice for the analysis of the local geometry of $\TropInt$, in particular for establishing tropical transversality and applying Theorem \ref{thm:mainthm}.
\end{remark}

\begin{proof}
    (1) holds by definition of the path matrix and (2) holds by Lemma \ref{lem:ConditionStarImpliesBoundOnEdgeLengths}. For \eqref{it:ConditionStarGivesMininequality}, we have by hypothesis
    \begin{equation}\label{eq:minequality-1}
    \underline d_{\alpha}^q + \alpha B^{n-3-q} \geq \underline d_{k_q}^q + k_qB^{n-3-q}
    =\underline d_{\ell_q}^q + k_qB^{n-3-q}.
    \end{equation}
    Consider the two paths along the convex hull of $i_q$ and $j_q$, connecting the leg $\alpha$ to either $k_q$ or $\ell_q$. Since $e_q$ disconnects $k_q$ and $\ell_q$, there is a unique choice $m \in \{k_q, \ell_q\}$ for which the corresponding path does not contain $e_q$; let $P$ denote this path. Rearranging \eqref{eq:minequality-1}, we have
    \begin{equation}\label{eq:minequality-2}
    (\underline d_\alpha^q + \alpha B^{n-3-q}) - (\underline d_m^q + m B^{n-3-q}) =
    \pm \len(P) + (\alpha - m) B^{n-3-q} \ \big({\geq 0}\big).
    \end{equation}
    If $P$ contains an edge $e_{q'}$ with $q' < q$, then $\len(P)$ must occur with a positive sign on the right-hand side (or else the expression is negative). More precisely, the right-hand side is at least
    \begin{equation}\label{eq:minequality-3}
    \tfrac{1}{2} B^{n-3-q'} - n B^{n-3-q} \geq (\tfrac{1}{2}B-n)B^{n-3-q} \geq \tfrac{1}{2} B^{n-3-q}
    \end{equation}
    as claimed (using $B \geq 2n+1$). If $P$ does not contain such an edge, we must have $\alpha > m$ (or else the whole expression is negative), so $\alpha - m \geq 1$. Then, more precisely, the right-hand side is again at least
    \begin{equation}\label{eq:minequality-4}
    B^{n-3-q} - \sum_{a=q+1}^r \len(e_a) 
    \geq (1 - \tfrac{1}{2}) B^{n-3-q} = \tfrac{1}{2} B^{n-3-q}. \qedhere
    \end{equation}
\end{proof}

\subsection{Transversality of the tropical intersection}\label{sec:TropicalTransversality}
    Let $\Gamma\in\TropInt.$ Let $\tau$ be the underlying tree of $\Gamma$, and for $q\in[r],$ let $e_q,k_q,\ell_q$ as in Lemma \ref{lem:GammaInTropIntCombinatorics}. By Lemma \ref{lem:GammaInTropIntCombinatorics}\eqref{it:REdges}, $\tau$ has exactly $r$ edges, and by Corollary \ref{cor:TropIntGivesConditionStar}, $(\tau,\vec e,\vec k,\vec\ell)$ satisfies Condition $(*).$ Let $\sigma\in\Sigma_r$ be the cone containing $\Gamma$ in its relative interior. Then $\vec y=(\len(e_1),\ldots,\len(e_r))$ give coordinates on $\sigma.$ 
    \begin{prop}\label{prop:TropHqCutOutByPathMatrixEquation}
        On $\sigma$, in a neighborhood of $\Gamma$, $\Trop(H_q)$ is cut out by the $q$-th row of the matrix equation $$A_\tree\vec y=\vec L_\tree,$$ where $A_\tree$ is the path matrix and $\vec L_\tree$ is the path-length vector, as in Definition \ref{def:PathMatrix}. In particular, $\vec y_\tree$ is the vector of edge-lengths of $\Gamma$.
    \end{prop}
    \begin{proof}
        By Lemma \ref{lem:GammaInTropIntCombinatorics}\eqref{it:ExactlyTwice}, the minimum \eqref{eq:Min} is achieved by exactly two indices, namely $k_q$ and $\ell_q$. The statement now follows from \eqref{eq:TropHMatrixEquation} and Definition \ref{def:PathMatrix}.
    \end{proof}

\begin{cor}\label{cor:TropIntFinite}
    $\TropInt$ is finite.
\end{cor}
\begin{proof}
There are only finitely many ways for a tree to satisfy condition $(*)$, and each way of fixing this data already determines a unique metric tree $\Gamma$ from the path matrix equation in Proposition~\ref{prop:TropHqCutOutByPathMatrixEquation}.
\end{proof}   

The following are precisely the tropical-geometric properties we will use in Section \ref{sec:RelateToAlgebraic} to prove that $\Trop(\bigcap_{q=1}^rH_q)=\bigcap_{q=1}^r\Trop(H_q)$ locally near $\Sigma_r$, and to compute $\mult_\Gamma(\Trop(\bigcap_{q=1}^rH_q),\sigma;\Sigma)$ for $\Gamma\in\Trop(\bigcap_{q=1}^rH_q)$.
\begin{prop}\label{prop:RelativeInteriors2}
    Let $\Gamma\in \TropInt$. Then:
    \begin{enumerate}
        \item $\Gamma$ lies in the relative interior of an $r$-dimensional cone $\sigma\in\Sigma_r$,\label{it:SigmaRInterior2}
        \item $\Gamma$ is an isolated point of $\TropInt$,
        \label{it:GammaIsolated}

        \end{enumerate}
        Furthermore, let $\zeta$ be any maximal cone of $\M_{0,n}^{\trop}$ that contains $\Gamma.$ Then:
        \begin{enumerate}\setcounter{enumi}{2}
        \item For $q=1,\ldots,r$, locally near $\Gamma$ on $\zeta$, $\Trop(H_q)$ is an affine-linear codimension-$1$ subspace $W_q$,\label{it:WqDef}
        \item For $q=1,\ldots,r$, the intersection $V_q:=W_1\cap\cdots\cap W_q$ is transverse and intersects $\zeta^\circ$,\label{it:VqDef}
        \item For $q=1,\ldots, r-1$, the tropical intersection multiplicity $i(V_{q+1},V_q\cdot W_{q+1};\zeta)$ is $1$, and\label{it:VkWMult1}
        \item If $v_\zeta\in\zeta$ is generic, then the stable intersection $(v_\zeta+\sigma)\cdot V_r$ (computed inside $\zeta$) is a single point with weight $1$.\label{it:VrsigmaMult1}
    \end{enumerate}
\end{prop}

\begin{proof}
   \eqref{it:SigmaRInterior2} is immediate from Lemma \ref{lem:GammaInTropIntCombinatorics}\eqref{it:REdges}. \eqref{it:GammaIsolated} follows from Proposition \ref{prop:TropHqCutOutByPathMatrixEquation}, and the fact that $A_\tau$ is full rank (since it is upper triangular with 1s on the diagonal).

  Let $\tau$ be the underlying tree of $\Gamma$, and let $\sigma \in \Sigma_r$ be the corresponding $r$-dimensional cone, which contains $\Gamma$ in its relative interior. Let $\zeta$ be any maximal cone of $\M_{0,n}^{\trop}$ that contains $\Gamma$. Then $\zeta$ has $\sigma$ as a face, and corresponds to a tree $\tau'$, with edges $e_1,\ldots,e_r,e_{r+1},\ldots,e_{n-3}$, from which $\tau$ is obtained by contracting the edges $e_{r+1},\ldots,e_{n-3}$. The coordinates on $\sigma$ extend to coordinates $(\len(e_1),\ldots,\len(e_{n-3}))$ on $\zeta$. We have the following stronger version of Proposition \ref{prop:TropHqCutOutByPathMatrixEquation}:
    \begin{claim}\label{clm:ZetaPathMatrixEquation}
        Locally near $\Gamma$ in $\zeta$, $\Trop(H_q)$ is cut out by the $q$th row of a matrix equation of the form 
        \begin{align}\label{eq:ZetaPathMatrixEquation}
            \begin{bmatrix}
        A_\tau&*    
        \end{bmatrix}\cdot\begin{bmatrix}
    \len(e_1)&
    \cdots&
    \len(e_{r})&
    \len(e_{r+1})&
    \cdots&
    \len(e_{n-3})
\end{bmatrix}^T\,=\vec L_\tau,
        \end{align}
         where $*$ is an $r\times(n-3-r)$ matrix of 0s and 1s.
    \end{claim}
\begin{proof}[Proof of Claim]
     Denote by $P'_q$ the path  \eqref{eq:ilkjPath} in $\tau'$ and by $P_q$ the analogous path in $\tau$. As in the proof of Proposition \ref{prop:TropHqCutOutByPathMatrixEquation}, locally near $\Gamma$ in $\Trop(H)$, the minimum is achieved (only) by $k_q$ and $l_q$. This means that, on $\zeta$, sufficiently near $\Gamma$, $H_q$ is defined by the condition $\len(P'_q)=(\ell_q-k_q)B^{n-3-q}.$ Now, observe that for the first $r$ edges $e\in \{e_1,\ldots,e_r\}$, $e$ is on $P_q'$ if and only if it is on $P_q$. This part is in the $A_\tau$ block of the matrix, and the $*$ block describes which other edges are on $P'_q$. The claim follows. 
\end{proof}

Let $W_q$ be the codimension-1 affine-linear space cut out by the $q$th row of \eqref{eq:ZetaPathMatrixEquation}, which is nonzero since $A_\tau$ is upper triangular with 1s on the diagonal. \eqref{it:WqDef} is immediate from Claim \ref{clm:ZetaPathMatrixEquation}.

Since $V_q=W_1\cap\cdots\cap W_q$ is cut out by the first $q$ rows of \eqref{eq:ZetaPathMatrixEquation}. The fact that $A_\tau$ is upper triangular with 1s on the diagonal implies that the intersection is transverse. A dimension count, using the fact that $\Gamma$ is an isolated point of $\TropInt$, implies that $V_q$ is not contained in a proper face of $\zeta$, hence intersects $\zeta^{\circ}$. This proves \eqref{it:VqDef}.

Since $A_\tau$ is upper triangular with 1s on the diagonal, the first $q+1$ rows of the matrix $\begin{bmatrix}
    A_\tau&*
\end{bmatrix}$ span a saturated sublattice of $L_\Z(\zeta)^\vee$. An easy calculation using \eqref{eq:TropIntMultDef2} shows $i(V_{q+1},V_q\cdot W_{q+1};\zeta)=1$, proving \eqref{it:VkWMult1}.

On $\zeta$, the face $\sigma$ is cut out by setting $\len(e_i) = 0$ for $i > r$, that is, by the equation $$\begin{bmatrix}
    \mathbf{0}&I
\end{bmatrix}\cdot\begin{bmatrix}
    \len(e_1)&
    \cdots&
    \len(e_{r})&
    \len(e_{r+1})&
    \cdots&
    \len(e_{n-3})
\end{bmatrix}^T=\vec 0,$$ where $\mathbf{0}$ is an $(n-3-r)\times r$ matrix of zeroes, and $I$ is an $(n-3-r)\times(n-3-r)$ identity matrix. Fix $v_\zeta\in\zeta^{\circ}$. Let $\bar{v}_{\zeta}=\begin{bmatrix}
    \mathbf{0}&I
\end{bmatrix}\cdot v_\zeta$. Then $v_\zeta+\sigma\subset\zeta$ is cut out by the equation $$\begin{bmatrix}
    \mathbf{0}&I
\end{bmatrix}\cdot\begin{bmatrix}
    \len(e_1)&
    \cdots&
    \len(e_{r})&
    \len(e_{r+1})&
    \cdots&
    \len(e_{n-3})
\end{bmatrix}^T= \bar{v}_\zeta.$$ 
The intersection of $V_r$ with $\sigma+v_\zeta$ is cut out by the matrix equation  
\begin{align}\label{eq:VrSigmaIntersect}
\begin{bmatrix}
   A_{\tau} & *\\
       \mathbf{0}  & I
\end{bmatrix}
\cdot\begin{bmatrix}
    \len(e_1)&
    \cdots&
    \len(e_{r})&
    \len(e_{r+1})&
    \cdots&
    \len(e_{n-3})
    \end{bmatrix}^T
    =
\begin{bmatrix}
    \vec{L}_{\tau}\\
    \bar{v}_{\zeta}
\end{bmatrix}. 
\end{align}
Since the leftmost matrix is nonsingular, \eqref{eq:VrSigmaIntersect} has at most 1 solution in $\zeta$. In the degenerate case $v_\zeta=0$, $\Gamma$ is a solution in $\zeta$, and it follows that for any sufficiently small $v_\zeta\in\zeta^{\circ},$ there is a unique solution in $\zeta^\circ.$ As in the proof of \eqref{it:VrsigmaMult1}, the fact that the leftmost matrix in \eqref{eq:VrSigmaIntersect} has determinant 1 implies that the stable intersection $(v_\zeta+\sigma)\cdot V_r$ has weight 1 at the solution.
\end{proof}

\subsection{Extended tropicalizations}\label{sec:Extended Tropicalization}

In this section we introduce the extended tropical moduli space $\Mbar_{0,n}^{\trop}$, and sketch a proof of Proposition \ref{prop:ExtendedFireworkAlgorithm}, which is an analog of Corollary \ref{cor:TropIntFinite} in this setting. The extended tropical moduli space detects points and subvarieties of $\Mbar_{0,n}$ that are contained in the boundary. We use Proposition \ref{prop:ExtendedFireworkAlgorithm} in Section \ref{sec:RelateToAlgebraic} to conclude that the intersection of the closures of the $H_q$s does not have an irreducible component supported on the boundary of $\Mbar_{0,n}$.

\newcommand{\Rbar}{\overline{\R}}
\newcommand{\fin}{\mathrm{fin}}
Let $\Rbar = \R \cup \{\infty\}$ denote the extended real numbers. The space $\Mbar_{0,S}^\trop$ can be viewed as the space of stable $S$-marked extended metric trees, i.e. stable $S$-marked trees with edge lengths from $\Rbar_{>0}$. It is stratified as follows. Let $\tau_\infty$ be a stable tree, thought of as an extended metric tree with all edges of infinite length. Then there is a stratum $\M_{\tau_\infty}^\trop$ consisting of all stable extended metric trees $\Gamma$ for which contracting the finite edges gives $\tau_\infty$. Each $\M_{\tau_\infty}^\trop$ is a cone complex, with cones indexed by stable trees $\tau_\fin$ obtained by inserting edges into $\tau_\infty$ (called \emph{finite edges}). We have

\makeatletter
\newcommand{\vast}{\bBigg@{4}}
\newcommand{\Vast}{\bBigg@{5}}
\makeatother

\begin{align}\label{eq:ExtendedProductDecomposition}
    \M_{\tau_\infty}^\trop \cong \prod_{v \in V(\tau_\infty)} \M_{0, \deg(v)}^\trop.
\end{align}
\begin{ex}
    The isomorphism \eqref{eq:ExtendedProductDecomposition} corresponds to cutting the tree along its infinite edges:
    \begin{center}
    \begin{tikzpicture}
    \begin{scope}[shift={(0, 0)}]
        \draw (-1, 2)       node {$\bullet$}
            -- node[left, midway] {$a$} (-1, 1) node {$\bullet$}
            -- node[above, midway] {$\infty$} (0, 0.5) node {$\bullet$}
            -- node[above, midway] {$b$} (1.5, 0.5) node {$\bullet$}
            -- node[above left, midway] {$c$} (2.5, 1.5) node {$\bullet$};
        \draw (1.5, 0.5)
            -- node[above, midway] {$\infty$} (3, 0)  node {$\bullet$}
            -- node[above, midway] {$d$} (4, 0.25) node {$\bullet$};
        \draw (-1.4, 2.5) -- (-1, 2) -- (-0.6, 2.5);
        \draw (-1.5, 0.5) -- (-1, 1);
        \draw (0, 0.5) -- (0, 0);
        \draw (3, 1.5) -- (2.5, 1.5) -- (2.9, 1.9);
        \draw (2.5, 2) -- (2.5, 1.5);
        \draw (2.6, -0.5) -- (3, 0) -- (3.4, -0.5);
        \draw (4, 0.75) -- (4, 0.25) -- (4.5, 0.25);
    \end{scope}

    \node at (5, 1) {$\mapsto \Vast($};

    \begin{scope}[shift={(7.1, -0.5)}]
        \draw (-1, 2)       node {$\bullet$}
            -- node[left, midway] {$a$} (-1, 1) node {$\bullet$};
        \draw (-1.4, 2.5) -- (-1, 2) -- (-0.6, 2.5);
        \draw (-1.5, 0.5) -- (-1, 1) -- (-0.5, 0.5);
    \end{scope}

    \node at (6.75, 0) {,};

    \begin{scope}[shift={(7.5, 0)}]
        \draw (-0.5, 1)
            -- (0, 0.5) node {$\bullet$}
            -- node[above, midway] {$b$} (1.5, 0.5) node {$\bullet$}
            -- node[above left, midway] {$c$} (2.5, 1.5) node {$\bullet$};
        \draw (1.5, 0.5)
            -- (2, 0.15);
        \draw (0, 0.5) -- (0, 0);
        \draw (3, 1.5) -- (2.5, 1.5) -- (2.9, 1.9);
        \draw (2.5, 2) -- (2.5, 1.5);
    \end{scope}

    \node at (10.7, 0) {,};

    \begin{scope}[shift={(8.5, 0.5)}]
        \draw (2.5, 0.35)
            -- (3, 0)  node {$\bullet$}
            -- node[above, midway] {$d$} (4, 0.25) node {$\bullet$};
        \draw (2.6, -0.5) -- (3, 0) -- (3.4, -0.5);
        \draw (4, 0.75) -- (4, 0.25) -- (4.5, 0.25);
    \end{scope}

    \node at (13.25, 1) {$\Vast)$};
    \end{tikzpicture}
    \end{center}
\end{ex}
Forgetful maps likewise extend to maps $\pi_{S}^{\trop} : \Mbar_{0, n}^\trop \to \Mbar_{0, S}^\trop$ when $S' \subseteq [n]$. To extend the tropical Kapranov map to $\Mbar_{0,S}^{\trop}$, we recall tropical projective space 
\[
\P(K^{S \setminus \{i, j\}})^\trop = (\Rbar^{S \setminus \{i, j\} }\setminus \{\vv \infty\})/ \R,
\]
where $\vv{\infty}$ denotes the all-infinite vector and the quotient is by finite translations by the all-ones vector. We write $(\underline{x}_\ell)_{\ell \in S \setminus \{i, j\}}$ to denote an element of tropical projective space, with each $\underline{x}_\ell \in \Rbar$.

The tropical Kapranov map extends to a map 
\[
\Psi_{i \rel j}^\trop : \Mbar_{0, S}^\trop \to \P(K^{S \setminus \{i, j\}})^\trop.
\]
As in Section \ref{subsec:KapranovAndForgetful}, we consider compositions $$\Psi_{S,i \rel j}^\trop := \Psi_{i \rel j}^\trop\circ\pi_{S}^{\trop}:\Mbar_{0,n}^{\trop}\to\P(K^{S \setminus \{i, j\}})^\trop,$$ described similarly to Equation \eqref{eqn:PsiCoords}: for each $\ell \in S \setminus \{i, j\}$, let $p_\ell$ be the nearest point to the $\ell$-th leg along the convex hull of the legs $i$ and $j$. Let $p_i$ be the attachment point of the leg $i$ within the convex hull of $S$. Let $\underline d_\ell \in \Rbar$ be the distance from $p_i$ to $p_\ell$. Then
\begin{equation}\label{eqn:ExtendedPsiCoords}
    \Psi_{S,i \rel j}^\trop(\Gamma) = (\underline d_\ell : \ell \in S \setminus \{i, j\}).
\end{equation}
The stability condition ensures that the image is not $\vv{\infty}$. The main difference in working with extended metric trees is that, if $\pi_S^\trop(\Gamma)$ has an infinite edge $e$ between the legs $i$ and $j$, then $\underline d_\ell = \infty$ for all $\ell$ with $p_\ell$ on the $j$-side of $e$. In particular, the map does not see the information of the relative positions of these $p_\ell$.

\begin{lem}
    With $a_\ell$ for $\ell \in S \setminus \{i, j\}$ and $H$ as in Definition \ref{def:BasicHyperplane}, the set $\Trop(\overline{H})$ is given by
    \[
    \Trop(\overline{H})=\big\{\Gamma\in\Mbar_{0,n}^{\trop} : \min_{\ell\in S\setminus\{i,j\}}(\underline d_\ell+ a_\ell)\text{ is achieved twice}\big\}.
    \]
\end{lem}

Note that the stability condition again implies that this minimum is finite.

\begin{proof}
    By \cite[Thm. 6.2.18]{MaclaganSturmfels2015}, $\Trop(\overline{H})$ is the closure of $\Trop(H)$ in $\Mbar_{0, n}^\trop$. It is straightforward to check that the closure is as described above, see \eqref{eq:DistanceConditionHTrop}.
\end{proof}

Let $\tau_\infty$ be an extended stable tree with all edges infinite. Let $\Sigma_{\tau_{\infty},r}$ denote the $r$-skeleton of $\Trop(\M_{\tau_{\infty}}),$ which consists of all extended metric trees $\Gamma$ with at most $r$ finite edges, and such that contracting all finite edges gives $\tau_{\infty}$. Corollary \ref{cor:TropIntGivesConditionStar} can be modified to show that elements of 
\[\Sigma_{\tau_{\infty},r}\cap\bigcap_{q=1}^r\Trop(\bar{H_q})\]
satisfy a modified version of Condition $(*)$ from Definition \ref{def:ConditionStar} as follows: the underlying tree $\tau$ has exactly $r$ \emph{finite} edges, such that the contraction of those edges gives $\tau_\infty$, and such that the path $P_q$ in Equation \eqref{eq:ilkjPath} does not contain any infinite edges. This is sufficient to define a path matrix as in Section \ref{sec:pathmatrix}, and to use arguments similar to those in Section \ref{sec:TropicalTransversality} to prove: 

\begin{prop}\label{prop:ExtendedFireworkAlgorithm} 
    The intersection $\Sigma_{\tau_{\infty},r}\cap\bigcap_{q=1}^r\Trop(\bar{H_q})$ consists of a finite set of points of $\relint(\Sigma_{\tau_{\infty},r})$.
\end{prop}

\section{From tropical to algebraic intersections}\label{sec:RelateToAlgebraic}

As in Section \ref{sec:Hypersurfaces}, let $\bfT$ be the torus into which $\M_{0,n}$ embeds, and let $\bfT^K$ be its base change to $K$. Recalling the notation established in Subsection~\ref{sec:NotationForIntersection}, define, for $r=1,\ldots, m$,
\begin{align*}
    X_r &\coloneqq \bigcap_{q=1}^r H_q  \subseteq \M_{0,n}^K.
\end{align*}
Let $\overline{X}_r$ denote the closure of $X_r$ in $\Mbar_{0, n}$. 

In this section, we apply Theorem~\ref{thm:mainthm} to characterize the limit cycle $[(\overline{X}_r)_0]$ in terms of the tropical intersection points $\TropInt$. We will show:
\begin{thm}\label{thm:XbarRepresentsPsiProduct}
We have an equality of cycle classes
\begin{equation}\label{eq:ProdHqBarEqualsXCycle}
    \prod_{q=1}^r[\bar H_q]=[\bar{X}_r]\in A^r(\Mbar_{0,n}^K).
    \end{equation}
\end{thm}
\begin{thm}\label{thm:LimitCycleFW}
    We have an equality of \emph{cycles} on $\Mbar_{0,n}$
    \begin{equation}\label{eq:XLimitCycleExpansion}
    [(\bar{X}_r)_0]=\sum_{\substack{\Gamma\in\TropInt\\ \tree \text{ tree of } \Gamma}}[\Mbar_{\tree}].
    \end{equation}
\end{thm}
\begin{cor}\label{cor:FWRepresentsPsiProduct}
    We have an equality of cycle classes
    \begin{equation}\label{eq:CycleClassExpansion}
    \pi_{S_1}^*(\psi_{i_1})\cdots \pi_{S_r}^*(\psi_{i_r})=\sum_{\substack{\Gamma\in\TropInt\\ \tree \text{ tree of } \Gamma}}[\Mbar_{\tree}] \in A^r(\Mbar_{0,n}).
    \end{equation}
\end{cor}
\begin{proof}[Proof of Corollary \ref{cor:FWRepresentsPsiProduct}, assuming Theorems \ref{thm:XbarRepresentsPsiProduct} and \ref{thm:LimitCycleFW}.]
    We have \begin{align*}
        \sum_{\substack{\Gamma\in\TropInt\\ \tree \text{ tree of } \Gamma}}[\Mbar_{\tree}]&=[(\bar{X}_r)_0]&&\text{by Theorem \ref{thm:LimitCycleFW}}\\
        &=[\bar{X}_r]&&\text{by \cite[Prop. 8.2(b)]{Fulton1998}}\\
        &=\prod_{q=1}^r[\bar H_q]&&\text{by Theorem \ref{thm:XbarRepresentsPsiProduct}}\\
        &=\pi_{S_1}^*(\psi_{i_1})\cdots \pi_{S_r}^*(\psi_{i_r})&&\text{by Corollary \ref{cor:HbarRepresentsPsi}}.\qedhere
    \end{align*}
\end{proof}

\begin{remark}
Although Theorem \ref{thm:LimitCycleFW} is satisfactory as a purely tropical statement, it is possible to give a concrete description of $\TropInt$ and hence of the cycle $[(\overline{X}_r)_0]$. We do so in Section \ref{sec:RecursionandFirework}.
\end{remark}

\subsection{Proof of Theorem \ref{thm:XbarRepresentsPsiProduct}}

\begin{prop}\label{prop:XHasCodimensionR}
    $X_r$ has pure codimension $r$ in $\M_{0,n}^K$.
\end{prop}
\begin{proof}
    Since $X_r$ is the intersection of $r$ hypersurfaces, every irreducible component has codimension $\le r.$ We have $$\Sigma_r\cap\Trop(X_r)=\Sigma_r\cap\Trop(\bigcap_{q=1}^r H_q)\subseteq\Sigma_r\cap\bigcap_{q=1}^r\Trop(H_q).$$  In particular, by Proposition \ref{prop:RelativeInteriors2} and Corollary \ref{cor:TropIntFinite}, $\Sigma_r\cap\Trop(X_r)$ is finite and contained in the relative interior $\relint(\Sigma_r).$ (This is Condition \ref{cond:GlobalIntersectionCondition}.) By Corollary \ref{cor:NoTooBigComponents}, $X_r$ has pure codimension $r$.
\end{proof}

\begin{cor}\label{cor:XbarHasCodimensionR}
    $\bar X_r$ has pure codimension $r$ in $\Mbar_{0,n}^K.$
\end{cor}

\begin{prop}\label{prop:NoIrreducibleComponentsInBoundary}
    $\bigcap_{q=1}^r\bar H_q=\bar{X}.$ That is, $\bigcap_{q=1}^r\bar H_q$ has no irreducible components contained in the boundary of $\Mbar_{0,n}^K.$
\end{prop}
\begin{proof}
Suppose $\bigcap_{q=1}^r\bar H_q$ has an irreducible component $Z$ contained in the boundary of $\Mbar_{0,n}^K.$ Then $Z$ has codimension $\le r$ in $\Mbar_{0,n}^K$, hence has codimension $\le r-1$ in the boundary of $\Mbar_{0,n}^K.$ Then there exists a unique locally closed boundary stratum $\M_{\tau}^K$ of $\Mbar_{0,n}^K$, corresponding to some tree $\tau$, such that $\M_\tau^K$ contains a dense open subset of $Z^\circ$ of $Z$.

Since $Z^{\circ}\subseteq \bigcap_{q=1}^r\bar{H_q}\cap\M_\tau$, we have $\Trop(Z^{\circ}) \subseteq
\bigcap_{q=1}^r\Trop(\bar{H_q})\cap\Trop(\M_\tau).$ In particular, \[
\Sigma_{\tau, r} \cap \Trop(Z^{\circ}) \subseteq
\Sigma_{\tau, r} \cap
\bigcap_{q=1}^r\Trop(\bar{H_q})\cap\Trop(\M_\tau)\] By Proposition \ref{prop:ExtendedFireworkAlgorithm}, the latter is finite and contained in the relative interior of $\Sigma_{\tau,r}$. Thus the same is true for $\Sigma_{\tau,r} \cap \Trop(Z^{\circ})$. By Corollary \ref{cor:NoTooBigComponents}, $Z^\circ$ has pure codimension $r$ in $\M_\tau^K$, a contradiction.
\end{proof}

\begin{proof}[Proof of Theorem \ref{thm:XbarRepresentsPsiProduct}]
    By Corollary \ref{cor:XbarHasCodimensionR} and Proposition \ref{prop:NoIrreducibleComponentsInBoundary}, $\bar{X}_r$ is the scheme-theoretic intersection of the hypersurfaces $\bar H_q$, and has the expected dimension. Thus $[\bar{X}_r]=\prod_{q=1}^r[\bar H_q].$
\end{proof}

\subsection{Proof of Theorem \ref{thm:LimitCycleFW}}

We use the purely tropical statements of Proposition \ref{prop:RelativeInteriors2} to apply \cite[Theorems 1.2 and 5.1.3]{OssermanPayne2013}, treating $M_{0, n}$ as the ambient space. We show that locally near $\Gamma \in \TropInt$, intersection in $M_{0, n}$ commutes with tropicalization for the first $r$ hypersurfaces.

\begin{prop}\label{prop:intersectionandtropicalization}
    Let $\Gamma\in \TropInt$, and let $\sigma$ denote the $r$-dimensional cone of $\M_{0,n}^{\trop}$ whose relative interior contains $\Gamma$. Then for $q=1,\ldots, r$, we have:
    \begin{enumerate}
        \item As sets, $\Trop(X_q)=\Trop(H_1)\cap\cdots\cap \Trop(H_q)$ locally near $\Gamma$, and \label{it:intersectioncommutes}
        \item If $\zeta$ is a maximal cone of $\M_{0,n}^{\trop}$ that contains $\sigma$ as a face, then near $\Gamma$ on $\zeta,$ $\Trop(X_q)$ is an affine-linear codimension-$q$ space with weight $1$ that intersects $\zeta^\circ$. \label{it:Xkmultone}
    \end{enumerate}
\end{prop}
\begin{proof}
    We induct on $q$. For $q=1$, \eqref{it:intersectioncommutes} is trivially true, and \eqref{it:Xkmultone} reduces to Proposition \ref{prop:RelativeInteriors2}\eqref{it:WqDef} together with Proposition \ref{prop:facetsofHtrophaveweightone}. Now, assume that the Proposition holds for some $q$. 
    
    Let $\zeta$ be a maximal cone of $\M_{0,n}^{\trop}$ that has $\sigma$ as a face. We refer to the notation from Proposition \ref{prop:RelativeInteriors2}:
    \begin{itemize}
    \item $W_{q+1}\subset\zeta$ is the codimension-$1$ affine-linear space that locally agrees with $H_{q+1}$ near $\Gamma$, 
    \item $V_q\subset\zeta$ is the affine-linear space that locally agrees with $\Trop(H_1)\cap\cdots\cap\Trop(H_q)$, and
    \item $V_{q+1} \subset \zeta$ is the affine-linear space that locally agrees with $\Trop(H_1)\cap\cdots\cap\Trop(H_{q+1})$. 
    \end{itemize}
    By Proposition \ref{prop:RelativeInteriors2}\eqref{it:VqDef}, $V_{q+1}$ has codimension $q+1$ in $\zeta$ and intersects $\zeta^\circ.$ By definition, $X_{q+1}= X_q\cap H_{q+1}$.
    By the inductive hypothesis, near $\Gamma$ on $\zeta$,
    \begin{equation*}
        \Trop(X_q) = \Trop(H_1)\cap\cdots\cap\Trop(H_q)=V_q \text{ with weight $1$.}
    \end{equation*}
    We thus have:
    \begin{equation*}
        \Trop(X_q) \cap \Trop(H_{q+1})=V_q\cap W_{q+1}=V_{q+1}.
    \end{equation*}
    By \cite[Thm. 1.2]{OssermanPayne2013}, near $\Gamma$ on $\zeta^\circ$, $\Trop(X_{q+1})=V_{q+1}$ as sets. (We have used that $\M_{0,n}^{\trop}$ has weight $1$ along $\zeta$ \cite[Cor. 4.3.12]{MaclaganSturmfels2015}, so every point of $\zeta^{\circ}$ is a simple point of $\M_{0,n}^{\trop}$ in the sense of Section \ref{sec:BackgroundTropicalIntersectionTheory}.) Since $\Trop(X_{q+1})\subseteq\Trop(X_q)\cap\Trop(H_{q+1}),$ and the latter locally agrees with $V_{q+1}$ on all of $\zeta$, we conclude by taking closures that $\Trop(X_{q+1})=V_{q+1}$ near $\Gamma$ on $\zeta.$

    Iterating over all choices of $\zeta$, we conclude that $\Trop(X_{q+1})=\Trop(H_1)\cap\cdots\cap\Trop(H_{q+1})$ near $\Gamma$ on $\M_{0,n}^{\trop},$ i.e. \eqref{it:intersectioncommutes} holds for $\Trop(X_{q+1})$.

    We now compute the weight of $\Trop(X_{q+1})$ along $V_{q+1}$. By Proposition \ref{prop:facetsofHtrophaveweightone}, $\Trop(H_{q+1})$ has weight 1 along $W_{q+1}.$ By the inductive hypothesis, $\Trop(X_q)$ has weight 1 along $V_q.$ By Proposition \ref{prop:RelativeInteriors2}\eqref{it:VkWMult1}, the tropical intersection multiplicity of $V_q$ and $W_{q+1}$ along $V_{q+1}$ is 1. By \cite[Thm. 5.1.3]{OssermanPayne2013}, and the fact that $H_{q+1}$ and $X_q$ are Cohen-Macaulay, we conclude that $\Trop(X_{q+1})$ has weight 1 along $V_{q+1}$. (Again, we have used the fact that $\zeta$ is a weight-1 facet of $\M_{0,n}^{\trop}.$) This proves \eqref{it:Xkmultone}.
\end{proof}

\begin{cor}\label{prop:FWTropX}
$\TropInt=\Sigma_r\cap{\Trop(X_r)}$.
\end{cor}

\begin{cor}\label{cor:TropMultIs1}
    The tropical intersection multiplicity $\mult_{\Gamma}(\Trop(X_r),\sigma;\Sigma)$ of $\Trop(X_r)$ and $\Sigma_r$ at $\Gamma\in\TropInt$ is 1.
\end{cor}
\begin{proof}
    Let $\zeta$ be a maximal cone of $\M_{0,n}^{\trop}$ containing $\Gamma.$ By Proposition \ref{prop:intersectionandtropicalization}\eqref{it:intersectioncommutes}, locally near $\Gamma$ in $\zeta$, $\Trop(X_r)$ is (set-theoretically) the codimension-$r$ affine-linear space $V_r.$ By Proposition \ref{prop:intersectionandtropicalization}\eqref{it:Xkmultone}, $\Trop(X_r)$ has weight 1 along $V_r$. The statement now follows from Proposition \ref{prop:RelativeInteriors2}\eqref{it:VrsigmaMult1}. 
\end{proof}

Finally, we prove Theorem \ref{thm:LimitCycleFW}.
\begin{proof}[Proof of Theorem \ref{thm:LimitCycleFW}]

By Proposition \ref{prop:FWTropX}, Theorem \ref{thm:mainthm}, and the fact that $\Gamma$ is in the relative interiors of $\Trop(X)$ and $\Sigma_r$, we have
\begin{align*}
    [(\bar{X})_0]&=\sum_{\substack{\Gamma\in\TropInt\\ \tree \text{ tree of } \Gamma}}\mult_{\Gamma}(\Trop(X),\sigma;\Sigma)\cdot[\Mbar_{\tree}]\\
    &=\sum_{\substack{\Gamma\in\TropInt\\ \tree \text{ tree of } \Gamma}}[\Mbar_{\tree}],
\end{align*}
where the second equality is Corollary \ref{cor:TropMultIs1}.
\end{proof}

\begin{remark}
    Analogues of Theorems \ref{thm:XbarRepresentsPsiProduct} and \ref{thm:LimitCycleFW} and Corollary \ref{cor:FWRepresentsPsiProduct} likewise hold for limit cycles of intersections of moving $\psi$-class hypersurfaces \emph{with a boundary stratum},
    \[\Mbar_{\tau} \cap \bigcap_{q=1}^r \overline{H}_q.\]
    These follow straightforwardly using the extended tropicalizations of the $\overline{H}_q$, as outlined in Section \ref{sec:Extended Tropicalization}.
\end{remark}

\begin{remark}
    One can set up this argument differently, using Theorem \ref{thm:XTilde2-OP} (with $Y=\bfT$ and $\widetilde X=\bigcap_{q=1}^r \widetilde H_q$) to compute $\mult_\Gamma(\Trop(X),\sigma;\Sigma)$ instead of computing it directly in $M_{0, n}^\trop$ as we have done. In doing so, one circumvents the full power of Theorem \ref{thm:mainthm}, instead using only the special case of Theorem \ref{thm:mainthm} that is \cite[Thm. 10.1]{Katz2009}, see Remark \ref{rem:discussion-of-thm-E}\eqref{it:UseKatz}. As discussed there, the conclusion of Theorem \ref{thm:mainthm} is then an expression for $[(\bar{\widetilde X})_0]$ as a sum of toric boundary strata in $Y_\Sigma$, from which an additional intersection-theoretic calculation yields the desired formula for $[\bar{X}_0]$.
\end{remark}

\section{Recursion and the Firework algorithm}\label{sec:RecursionandFirework}

As discussed in Section \ref{sec:pathmatrix}, a tuple $(\tree, \vec e, \vec k, \vec \ell)$ satisfying condition $(*)$ allows us to construct a metric tree $\Gamma$ with underlying tree $\tau$, with some but not all of the properties necessary to have $\Gamma \in \TropInt$ (see Proposition \ref{prop:MetricSatisfiesPathEquations}). We end our paper by refining this condition to give a recursive procedure, which we call the \emph{firework algorithm}, to precisely generate $\TropInt$. The algorithm comes from examining the recursive properties of condition $(*)$.
  
\subsection{Recursive nature of condition \texorpdfstring{$(*)$}{(*)}: edge contraction}

We first consider the effect of contracting the last edge in condition $(*)$. Let $(\tau', \vec{e}\,', \vec{k}\,', \vec{\ell}\,')$ satisfy $(*)$, where $\tau'$ is a tree with $r+1$ edges. Let $\tau$ be the tree obtained from $\tau'$ by contracting $e'_{r+1}$, and let $\vec e,\vec k,\vec\ell$ be obtained from $\vec e\,',\vec k\,',\vec \ell\,'$ by deleting the last entry. The following is immediate from Definition \ref{def:ConditionStar}:    
\begin{prop}\label{prop:ContractStar}
    The tuple $(\tau,\vec e,\vec k,\vec \ell)$ satisfies $(*)$.
\end{prop}
In particular, we may define path matrices, edge-length vectors and metric trees $\Gamma'$ and $\Gamma$ for $\tree'$ and $\tree$, as described in Lemma \ref{lem:ConditionStarImpliesBoundOnEdgeLengths} and Proposition \ref{prop:MetricSatisfiesPathEquations}. We will say {\bf $\Gamma$ is obtained from $\Gamma'$ by contracting the shortest edge and readjusting edge lengths}. (Recall that, by Proposition \ref{prop:MetricSatisfiesPathEquations}\eqref{it:ConditionStarGivesEdgeOrderingByLength}, $e'_{r+1}$ is the shortest edge of $\Gamma'$.)

We have the following comparison result:
\begin{lem}\label{lem:ContractRelationshipBetweenMetrics}
    The metrics on $\Gamma'$ and $\Gamma$ are related by, for each $q \leq r$,
    \[
    \abs{\len(e'_q) - \len(e_q)} < n \cdot 2^n \cdot B^{n-3-(r+1)}.
    \]
    \end{lem}
    \begin{remark} \label{rmk:rounding-lengths}
        Edge lengths in $\Gamma$ are all integer multiples of $B^{n-3-r}$. For $B \gg 0$ ($B > n\cdot 2^{n+1}$ suffices), the bound in Lemma \ref{lem:ContractRelationshipBetweenMetrics} implies
        \begin{align}
            \abs{\len(e_q')-\len(e_q)}<\tfrac{1}{2}B^{n-3-r},
        \end{align}
        so we may think of $\Gamma$ as being obtained from $\Gamma'$ by contracting the shortest edge (whose length was of order $B^{n-3-(r+1)}$) and rounding each edge length to the nearest integer multiple of $B^{n-3-r}$. If $B > n^2 2^{n+1}$, we even see that every path length in $\Gamma$ is obtained by rounding the corresponding path length in $\Gamma'$, since there are at most $n$ edges.
    \end{remark}
    \begin{proof}
        The path matrices of $\tree'$ and $\tree$, and their inverses, are related as follows \begin{align*}
            A_{\tree'}&=\left(\begin{array}{@{}c|c@{}}
  \begin{matrix}
   && && \\
   && A_{\tree} &&\\
   &&&&
  \end{matrix}
  & E \\
\hline
  \mathbf{0} &
  1
\end{array}\right)
&&\text{and}&
A_{\tree'}^{-1}&=\left(\begin{array}{@{}c|c@{}}
  \begin{matrix}
   && && \\
   && A_{\tree}^{-1} &&\\
   &&&&
  \end{matrix}
  & -A_{\tree}^{-1}E \\
\hline
  \mathbf{0} &
  1
\end{array}\right),
        \end{align*}
        where $E$ is a length-$r$ vector of 0s and 1s.

Let $\vec L\,'$ and $\vec L$ denote the path-length vectors associated to $(\tau',\vec e\,',\vec k\,',\vec \ell\,')$ and $(\tau,\vec e,\vec k,\vec \ell)$, respectively and $\vec y_{\tree'}$ and $\vec y_\tree$ the corresponding edge-length vectors, so $\len(e_q) = y_{\tree, q}$ and $\len(e_q') = y_{\tree', q}$ for each $q$. By definition, we have $\vec y_\tree=A_\tree^{-1}\vec L$ and $\vec y_{\tree'}=A_{\tree'}^{-1}\vec L\,'$. For $q\in [r]$, then by the above block decomposition of $A_{\tree'}^{-1}$ we have \begin{align}\label{eq:DifferenceInMetrics}
    y_{\tree',q}-y_{\tree,q}=c_q(\ell_{r+1}-k_{r+1})B^{n-3-(r+1)},
\end{align} where $c_q$ is the $q$-th entry of the vector $-A_{\tree}^{-1}E$. 

We claim that $\abs{c_q}\le2^{r-q}$. Denote by $d_{q,q'}$ the entries of $A_{\tree}^{-1}$. Since $E$ is a vector of 0s and 1s, it is sufficient to show $\sum_{q'\ge q}\abs{d_{q,q'}}\le 2^{r-q}$ for $q\in[r]$. Let $q\in[r],$ and let $q'>q.$ By definition, $\sum_{\beta=q}^ra_{q,\beta}d_{\beta,q'}=0$. Since $a_{q,q}=1,$ we have  
\[\abs{d_{q,q'}}=\abs{\sum_{\beta=q+1}^{q'}a_{q,\beta}d_{\beta,q'}}\le\sum_{\beta=q+1}^{q'}\abs{a_{q,\beta}}\abs{d_{\beta,q'}}\le\sum_{\beta=q+1}^{q'}\abs{d_{\beta,q'}}.\]
Induction, using $d_{q',q'}=1$, yields $\abs{d_{q,q'}}\le2^{q'-q-1}.$ (This particular bound was already known, see e.g. \cite{SpeyerMOPost2011}.) Thus 
\[\abs{c_q}\le\sum_{q'=q}^r\abs{d_{q,q'}}=1+\sum_{q'=q+1}^r\abs{d_{q,q'}}\le1+\sum_{q'=q+1}^r2^{q'-q-1}=2^{r-q}\]
as desired.

Finally, using \eqref{eq:DifferenceInMetrics} and the fact that $\ell_{r+1}, k_{r+1}, r, q \in [n]$, we obtain
\begin{align*}
    \abs{y_{\tree',q}-y_{\tree,q}}&\leq2^{r-q}(\ell_{r+1}-k_{r+1})B^{n-3-(r+1)} \\
    &\leq 2^n n B^{n-3-(r+1)}. \qedhere
\end{align*}
    \end{proof}

\subsection{Recursive nature of \texorpdfstring{$\TropInt$}{\TropIntPDFString}, part 1: edge contraction}

    We now examine how contracting edges and readjusting edge lengths affects containment in the tropical $\psi$-hypersurfaces $\Trop(H_q)$.
    
    Let $(\tau', \vec{e}\,', \vec{k}\,', \vec{\ell}\,')$ and $(\tau,\vec e,\vec k,\vec \ell)$ be as in the previous subsection, where $\tau'$ has $r+1$ edges and $\tau$ has $r$ edges. Let $\Gamma'$ and $\Gamma$ be the corresponding metric trees, with $\Gamma$ obtained from $\Gamma'$ by contracting the shortest edge and readjusting edge lengths. We show:
    
\begin{prop}\label{prop:EdgeContractionFromTropInt2}
    Let $q \in [r]$. Then $\Gamma' \in \Trop(H_q)$ if and only if $\Gamma \in \Trop(H_q)$.
\end{prop}
\begin{proof}
    Let $(\underline d_\alpha')$ and $(\underline d_\alpha)$ denote the images of $\Gamma'$ and $\Gamma$ under the tropical Kapranov map $\Psi_{S_q,i_q \rel j_q}^\trop$. By Proposition \ref{prop:MetricSatisfiesPathEquations} and Remark \ref{rmk:ConditionStarNotEnough}, the $k_q$ and $\ell_q$ parts of Equation \eqref{eq:Min} for $\Gamma'$ are equal; likewise for \eqref{eq:Min} for $\Gamma'$. We show that $k_q$ and $\ell_q$ achieve the minimum for $\Gamma'$ if and only if they achieve the minimum for $\Gamma$.

    Let $\alpha \in S \setminus \{i_q, j_q, k_q, \ell_q\}$ be arbitrary. Proceeding similarly to the proof of Proposition \ref{prop:MetricSatisfiesPathEquations}\eqref{it:ConditionStarGivesMininequality}, let $m \in \{k_q, \ell_q\}$ be the unique choice such that the path along the convex hull of $i_q$ and $j_q$ from the leg $\alpha$ to the leg $m$ does not contain $e_q$. Let $P'$ denote this path in $\Gamma'$ and $P$ the corresponding path in $\Gamma$. We note that $P'$ differs from $P$ only in that the contracted edge $e_{r+1}'$ may be contained in $P'$ (and absent in $P$), and that the other edge lengths are readjusted. We consider the quantities
    \begin{align}\label{eq:comparison-lower-bound-1}
    (\underline d_\alpha' + \alpha B^{n-3-q}) -
    (\underline d_m' + m B^{n-3-q})
    &=
    \underbrace{(\underline d_\alpha' - \underline d_m')}_{\pm \len(P')} + (\alpha - m)B^{n-3-q}, \\
    \label{eq:comparison-lower-bound-2}
    (\underline d_\alpha + \alpha B^{n-3-q}) -
    (\underline d_m + m B^{n-3-q})
    &=
    \underbrace{(\underline d_\alpha - \underline d_m)}_{\pm \len(P)} + (\alpha - m)B^{n-3-q}.
    \end{align}
    We wish to show that \eqref{eq:comparison-lower-bound-1} is nonnegative if and only if \eqref{eq:comparison-lower-bound-2} is nonnegative. By Proposition \ref{prop:MetricSatisfiesPathEquations}\eqref{it:ConditionStarGivesMininequality}, the one that is nonnegative is actually positive and bounded below by $\tfrac{1}{2}B^{n-3-q}$. We note also that when $P$ is nonempty, $\len(P')$ has the same sign in \eqref{eq:comparison-lower-bound-1} as $\len(P)$ in \eqref{eq:comparison-lower-bound-2}.

    If $B > n^2\cdot 2^{n+1}$, we may deduce the positivity of both equations from the fact that the metrics on $\Gamma'$ and $\Gamma$ are nearly identical (see Remark \ref{rmk:rounding-lengths}). In particular, equations \eqref{eq:comparison-lower-bound-1} and \eqref{eq:comparison-lower-bound-2} then differ by at most
    \[|\len(P) - \len(P')| < \tfrac{1}{2}B^{n-3-r} \leq \tfrac{1}{2}B^{n-3-q}.\]
    
    Alternatively, with our weaker running assumption $B \geq 2n+1$, one can instead argue as in the proof of Proposition \ref{prop:MetricSatisfiesPathEquations}\eqref{it:ConditionStarGivesMininequality}, that the positivity of \eqref{eq:comparison-lower-bound-1} and \eqref{eq:comparison-lower-bound-2} is caused by a single, necessarily positive, dominant term. That term therefore occurs positively in both equations, since the same combinatorial edges (except $e_{r+1}'$, which never dominates) contribute to both and with approximately the same lengths.

    Concretely, suppose an edge $e_{q'}$ with $q' < q$ is contained in $P$ (and so $e'_{q'}$ is contained in $P'$). This edge has length $\geq \tfrac{1}{2} B^{n-3-{q'}}$ in both $\Gamma'$ and $\Gamma$ by Proposition \ref{prop:MetricSatisfiesPathEquations}\eqref{it:ConditionStarGivesEdgeOrderingByLength}, so it must occur with positive sign in whichever of \eqref{eq:comparison-lower-bound-1} and \eqref{eq:comparison-lower-bound-2} is positive; cf. Equation \eqref{eq:minequality-3}. Then both $\len(P)$ and $\len(P')$ occur with positive sign, so
    \[
    \underbrace{(\underline d_\alpha - \underline d_m)}_{\len(P)} + (\alpha - m)B^{n-3-q} \geq \tfrac{1}{2}B^{n-3-q+1} - n B^{n-3-q} = (\tfrac{1}{2}B-n) B^{n-3-q} \geq \tfrac{1}{2}B^{n-3-q},
    \]
    and similarly for $\underline d'_\alpha$. 
    
    If no such edge occurs, then the $(\alpha-m)B^{n-3-q}$ term dominates and is positive (in one of \eqref{eq:comparison-lower-bound-1} and \eqref{eq:comparison-lower-bound-2}), so $\alpha > m$; cf. Equation \eqref{eq:minequality-4}. Then, by Proposition \ref{prop:MetricSatisfiesPathEquations}\eqref{it:ConditionStarGivesEdgeOrderingByLength} again, 
    \[
    \underbrace{(\underline d_\alpha - \underline d_m)}_{\pm \len(P)} + (\alpha - m)B^{n-3-q} \geq B^{n-3-q} - \sum_{a=q+1}^r \len(e_a) \geq (1-\tfrac{1}{2}) B^{n-3-q} = \tfrac{1}{2} B^{n-3-q},
    \]
    and similarly for $\underline d'_\alpha$.
\end{proof}

\begin{cor}\label{cor:ContractingSendsTropIntToTropInt}
    If $\Gamma'\in\Sigma_{r+1}\cap\bigcap_{q=1}^{r+1}\Trop(H_q)$, and $\Gamma$ is obtained from $\Gamma'$ by contracting the shortest edge and readjusting edge lengths, then $\Gamma\in\TropInt$.
\end{cor}

The converse almost holds --- that is, 
\[\text{If } \Gamma \in \TropInt, \text{ then } \Gamma' \in \Sigma_{r+1} \cap \bigcap_{q=1}^r\Trop(H_q) \text{ \quad(with the intersection only up to $r$)}.\]
To formulate a recursive algorithm to produce all possible $\Gamma'$ (assuming we have found all possible $\Gamma$), it then remains to explain how to insert edges into $\Gamma$ while satisfying the final condition $\Gamma' \in \Trop(H_{r+1})$.

\subsection{Edge-insertion into stable trees}
 We introduce some terminology related to $S$-marked trees.

\begin{definition}\label{def:Branches}
Let $\tree$ be a stable $S$-marked tree, and let $v$ be a vertex of $\tree.$ Let $\mathcal B(v)$ denote the set of connected components of the topological space $\tree\setminus v$. We refer to the elements of $\mathcal B(v)$ as the \emph{branches} of $\tree$ at $v$. (Note $\abs{\mathcal B(v)}=\val(v)$.) Recording the marked points on each branch gives the \emph{partition of $S$ induced by $v$}, which has $\abs{\mathcal B(v)}$ nonempty parts.
\end{definition}
\begin{definition}\label{def:EdgeInsertion}
    Let $\tree$ be a stable $S$-marked tree, and let $v$ be a vertex of $\tree.$ We define an operation called \emph{edge-insertion at $v$}. Given $\mathcal B(v)=\mathcal B_1\sqcup\mathcal B_2$ a partition with $\abs{\mathcal{B}_1},\abs{\mathcal{B}_2}\ge2,$ we define a stable $S$-marked tree $\EI(\tree,v,\mathcal B_1\sqcup\mathcal B_2)$ with vertices $(V(\Gamma)\setminus\{v\})\cup\{v_1,v_2\}$, such that the two vertices $v_1,v_2$ are connected by an edge $e$, the branches in $\mathcal B_1$ are attached to $v_1$, and the branches in $\mathcal B_2$ are attached to $v_2$.
\end{definition}
\begin{definition}\label{def:S-Stable}
    Let $\tree$ be a stable $S$-marked tree, and let $v$ be a vertex of $\tree.$ For $S'\subseteq S,$ we call $v$ \emph{$S'$-stable} if at least three branches at $v$ contain an element of $S'$.
\end{definition}
Observe that the $S'$-stable vertices of $\tree$ are precisely the images of vertices under the map $\pi_{S'}^{\trop}(\tree)\into\tree.$ 

\begin{ex}
Consider the following tree $\tau$ and its decomposition into branches $\mathcal{B}(v)$:
\begin{center}
\begin{tikzpicture}[scale=.75]
% first tree
            \draw (0,0)--(1.5,0);
            \draw (0,0) node {$\bullet$};
            \draw (0,0) node {$\bullet$} node[below] {$w$};
            \draw (1.5,0) node {$\bullet$} node[below] {$v$};
            \draw (0,0)--++(180:.8) node[left] {1};
            \draw (0,0)--++(90:.7) node[above] {3};
            \draw (1.5,0)--++(120:.8) node[above] {4};
            \draw (1.5,0)--++(60:.8) node[above] {5};
            \draw (1.5,0)--++(0:.8) node[right] {2};
% branches
\begin{scope}[shift = {(7, 0)}]
            \draw (0,0)--(1.2,0);
            \draw (0,0) node {$\bullet$};
            \draw (0,0)--++(180:.8) node[left] {1};
            \draw (0,0)--++(90:.7) node[above] {3};
            \draw (0,0) node {$\bullet$} node[below] {$w$};
            \draw (1.5,0)++(120:.3)--++(120:.8) node[above] {4};
            \draw (1.5,0)++(60:.3)--++(60:.8) node[above] {5};
            \draw (1.5,0)++(0:.3)--++(0:.8) node[right] {2};
\end{scope}
        \end{tikzpicture}
\end{center}
Thus $v$ induces the set partition $[5]=\{1,3\}\sqcup\{2\}\sqcup\{4\}\sqcup\{5\}$ into branches $\mathcal{B}(v)$. Letting $\mathcal{B}_1 = \{1,3\} \sqcup \{2\}$ and $\mathcal{B}_2 = \{4\} \sqcup \{5\}$, we obtain the tree $\EI(\tree,v,\mathcal B_1\sqcup\mathcal B_2)$:
    \begin{center}
    $\EI(\tree,v,\mathcal B_1\sqcup\mathcal B_2)$\quad$=$\quad
    \raisebox{-0.5\height}{
        \begin{tikzpicture}[scale=.75]
            \draw (0,0)--(1.5,0);
            \draw (0,0) node {$\bullet$};
            \draw (0,0) node {$\bullet$} node[below] {$w$};
            \draw (1.5,0) node {$\bullet$} node[below] {$v_1$};
            \draw (3,0) node {$\bullet$} node[below] {$v_2$};
            \draw (0,0)--++(180:.8) node[left] {1};
        \draw (0,0)--++(90:.8) node[above] {3};
            \draw (1.5,0)--++(90:.8) node[above] {2};
            \draw (3,0)--++(90:.8) node[above] {5};
            \draw (3,0)--++(0:.8) node[right] {4};
            \draw[very thick] (1.5,0)--(3,0);
            \draw (2.25,0) node[below] {$e$};
        \end{tikzpicture}
        }
    \end{center}
In $\EI(\tree,v,\mathcal B_1\sqcup\mathcal B_2)$, the vertices $w$ and $v_1$ are $\{1,2,3,4\}$-stable, and the vertices $w$ and $v_2$ are $\{1,3,4,5\}$-stable.
\end{ex}

\subsection{Recursive nature of \texorpdfstring{$\TropInt$}{\TropIntPDFString} part 2: edge insertion}\label{sec:edgeinsertion}

Fix $\Gamma\in\TropInt$. Let $\tau$ denote the underlying tree of $\Gamma$, and for $q\in[r],$ let $e_q,k_q,\ell_q$ as in Lemma \ref{lem:GammaInTropIntCombinatorics}. We describe an algorithm to obtain ``nearby" points of $\Gamma'\in\Sigma_{r+1}\cap\bigcap_{q=1}^{r+1}\Trop(H_{q})$. We define the data:
        \begin{itemize}
            \item Let $\underline v$ denote the vertex of $\pi_{S_{r+1}}^{\trop}(\tree)$ marked by $i_{r+1}$, and let $v$ be the image of $\underline{v}$ under the natural inclusion $\pi_{S_{r+1}}^{\trop}(\tree)\into\tree$, an $S_{r+1}$-stable vertex.
            \item Let $I$ be the branch of $\Gamma$ at $v$ that contains $i_{r+1}$. Note that by construction, $I$ contains no other elements of $S_{r+1}$.
            \item Let $J$ be the branch of $\Gamma$ at $v$ that contains $j_{r+1}$. If $J=I$, then the output of the algorithm for this $\Gamma$ is empty. If $J\neq I$, then continue.
            \item Let $k_{r+1}$ be the minimal element of $S_{r+1}\setminus (I\cup J)$, which is nonempty since $v$ is $S_{r+1}$-stable. Let $K$ be the branch of $\Gamma$ at $v$ containing $k_{r+1}$.
        \end{itemize}
        \begin{definition}\label{def:edgeinsertionfromtropint}
         Suppose $\mathcal B(v)=\mathcal B_1\sqcup\mathcal B_2$ is a set partition satisfying:
         \begin{itemize}
             \item $I\in\mathcal B_1$, and $J,K\in\mathcal B_2$, and
             \item there exists a branch in $B_1\setminus \{I\}$ containing an element of $S_{r+1}$. 
         \end{itemize}
            Given such a set partition, we define $\tree'$ to be $\EI(\tree,v,\mathcal B_1\sqcup\mathcal B_2))$. The edges of $\tree'$ are canonically identified with $e_1,\ldots,e_{r}$, plus one new edge, which we denote $e_{r+1}$. Let $\ell_{r+1}$ be the minimal element of $S_{r+1}\setminus I$ such that the branch $L$ containing $\ell_{r+1}$ is in $\mathcal B_1$. We write
            \begin{align*}
                \vec e\,'&=(e_1>\cdots>e_{r+1}),&\vec k\,'&=(k_1,\ldots,k_{r+1}),&&\text{and}&\vec \ell\,'&=(\ell_1,\ldots,\ell_{r+1}).
            \end{align*}
    \end{definition}

\begin{lem}\label{lem:conditionstarafteredgeinsertion}

Any $(\tau', \vec e\,', \vec k\,', \vec\ell\,')$ obtained from $\Gamma$ according to Definition \ref{def:edgeinsertionfromtropint} satisfies condition $(*)$. 
    
\end{lem}
\begin{proof}
     For $q\le r,$ we have $\ell_{q}>k_{q}$ by Lemma \ref{lem:GammaInTropIntCombinatorics}. Since $\ell_{r+1}\not\in I,J,K$, we have $\ell_{r+1}>k_{r+1}$ by definition of $k_{r+1}$. Thus the first part of $(*)$ is satisfied. Since $\tree'$ is obtained from $\tree$ by edge insertion, the edge $e_q$ still separates $i_q$ and $\ell_q$ from $j_q$ and $k_q$. Thus the convex hull of $i_q,\ell_q,j_q,k_q$ in $\tree'$ is of the form \eqref{eq:ilkjPath}. By construction, the resulting path $P_q'\subseteq\tree'$ contains all edges corresponding to those of $P_q$, and possibly also contains $e_{r+1};$ both cases satisfy $(*)$. 
     
     It remains to check the case $q={r+1}$. By construction, $e_{r+1}$ separates $i_{r+1}$ and $\ell_{r+1}$ from $j_{r+1}$ and $k_{r+1}$; this gives the first part of condition $(*)$. Finally, since $i_{r+1}$, $j_{r+1}$, $k_{r+1}$ and $\ell_{r+1}$ were in different branches of $v$ in $\tree$, the path $P_{r+1}$ in $\tau'$ consists of the single edge $e_{r+1}$. 
     
     We conclude that $(\tree',\vec e\,',\vec k\,',\vec \ell\,')$ satisfies $(*)$.
\end{proof}

We may therefore define a metric tree $\Gamma'$ on $\tree'$ via Proposition \ref{prop:MetricSatisfiesPathEquations}, using the tuple $(\tau', \vec e\,', \vec k\,', \vec\ell\,')$. We say {\bf $\Gamma'$ is obtained from $\Gamma$ by inserting an allowable edge and readjusting edge lengths}.

\begin{prop}\label{prop:tropintafteredgeinsertion}
Fix $\Gamma \in \TropInt$.
\begin{enumerate}
    \item If $\Gamma'$ is obtained from $\Gamma$ by inserting an allowable edge and readjusting edge lengths, then
$\Gamma'\in\Sigma_{r+1}\cap\bigcap_{q=1}^{r+1}\Trop(H_{q})$.\label{it:MapToTropIntR+1}
\item  For any such $\Gamma'$, contracting the shortest edge and readjusting edge lengths recovers $\Gamma$.\label{it:MapToTropIntR+1LandsInsideFiber}

\item Let $\Gamma'' \in\Sigma_{r+1}\cap\bigcap_{q=1}^{r+1}\Trop(H_{q})$ be such that contracting the shortest edge and readjusting edge lengths gives $\Gamma$. Then $\Gamma''$ can be obtained from $\Gamma$ by inserting an allowable edge and readjusting edge lengths.\label{it:MapToTropIntR+1SurjectiveOnFibers}
\end{enumerate}
\end{prop}
\begin{proof}
Statement (2) holds because the newly inserted edge is the shortest edge, by Proposition \ref{prop:MetricSatisfiesPathEquations}\eqref{it:ConditionStarGivesEdgeOrderingByLength}.

For (1), if $q \in [r]$, then $\Gamma' \in \Trop(H_q)$ by Proposition \ref{prop:EdgeContractionFromTropInt2}. For $q=r+1$, the fact that $k_{r+1}$ and $\ell_{r+1}$ achieve the minimum in Equation \eqref{eq:Min} uses the minimality in the choices of $k_{r+1}$ and $\ell_{r+1}$ in Definition \ref{def:edgeinsertionfromtropint}, as follows. Let $\alpha \in S^{r+1} \setminus \{i_{r+1}, j_{r+1}, k_{r+1}, \ell_{r+1}\}$. Since $\alpha \notin S^{r+1}$, $\alpha$ is not on the $I$ branch, so $\underline d_\alpha \geq \underline d_{\ell_{r+1}}$.  If $\underline d_\alpha = \underline d_{\ell_{r+1}}$, then $\underline d_\alpha + \alpha B^{n-3-(r+1)}$ exceeds the minimum since $\alpha > \ell_{r+1}$ (by the choice of $\ell_{r+1}$ in Definition \ref{def:edgeinsertionfromtropint}). Similar logic holds if $\underline d_\alpha = \underline d_{k_{r+1}}$. Otherwise, $\underline d_\alpha$ is separated from $\underline d_{k_{r+1}}$ by an edge of length $\geq \tfrac{1}{2}B^{n-3-r}$, and so $\underline d_\alpha + \alpha B^{n-3-(r+1)}$ exceeds the minimum by at least $\tfrac{1}{2}B^{n-3-(r+1)}$.

For (3), we recover $k_{r+1}$ and $\ell_{r+1}$ as the unique minimizers of Equation \eqref{eq:Min} for $\Gamma''$ with $q=r+1$, by Lemma \ref{lem:GammaInTropIntCombinatorics}\eqref{it:ExactlyTwice}. The shortest edge of $\Gamma''$ then determines the set partition $\mathcal{B}_1 \sqcup \mathcal{B}_2$. Reasoning as above, we see that $k_{r+1}$ and $\ell_{r+1}$ are necessarily the minimal elements described in Definition \ref{def:edgeinsertionfromtropint} and that $\mathcal{B}_1 \sqcup \mathcal{B}_2$ has the required form.
\end{proof}
\begin{cor}\label{cor:TropIntRecursion}
     $\Sigma_{r+1}\cap\bigcap_{q=1}^{r+1}\Trop(H_q)$ is precisely the set of metric trees obtained from elements of $\TropInt$ by inserting allowable edges and readjusting edge lengths.
\end{cor}
\begin{proof}
    Proposition \ref{prop:tropintafteredgeinsertion}\eqref{it:MapToTropIntR+1} implies that the set of metric trees obtained from an element of $\TropInt$ by inserting an allowable edge and readjusting edge lengths is a subset of $\Sigma_{r+1}\cap\bigcap_{q=1}^{r+1}\Trop(H_q)$. For the other inclusion, let $\Gamma'\in\Sigma_{r+1}\cap\bigcap_{q=1}^{r+1}\Trop(H_q)$, and let $\Gamma$ be the metric tree obtained by contracting the shortest edge and readjusting edge lengths. By Corollary \ref{cor:ContractingSendsTropIntToTropInt}, $\Gamma\in\TropInt.$ By Proposition \ref{prop:tropintafteredgeinsertion}\eqref{it:MapToTropIntR+1SurjectiveOnFibers}, $\Gamma'$ can be obtained from $\Gamma$ by inserting an allowable edge and readjusting edge lengths. This proves the other inclusion.
\end{proof}

\begin{cor}\label{cor:InjectivityForMetricTrees}
    Each $\Gamma'\in\Sigma_{r+1}\cap\bigcap_{q=1}^{r+1}\Trop(H_q)$ is obtained from a \emph{unique} $\Gamma\in\TropInt$ by inserting an allowable edge and readjusting edge lengths.
\end{cor}
\begin{proof}
    Immediate from Proposition \ref{prop:tropintafteredgeinsertion}\eqref{it:MapToTropIntR+1LandsInsideFiber}.
\end{proof}

\subsection{The firework algorithm}
\label{subsec:firework}
We may phrase the above as an algorithm, as follows. Let $\FW_0$ denote the set containing the cone point of $\M_{0,n}^{\trop}.$ Trivially, we have $$\FW_0=\Sigma_{0}\cap\bigcap_{q=1}^{0}\Trop(H_q).$$ For each $r=1,\ldots,m$, let $\FW_r\subset\M_{0,n}^{\trop}$ be the finite set obtained by inserting an allowable edge all possible ways, then readjusting edge lengths, to all elements of $\FW_{r-1}.$ By Corollary \ref{cor:TropIntRecursion}, we conclude $\FW_r=\TropInt.$ Theorem \ref{thm:LimitCycleFW} yields:
\begin{cor}
    We have an equality of \emph{cycles} on $\Mbar_{0,n}$
\begin{equation}\label{eq:FWClassExpansion2}
    [(\bar{X})_0]=\sum_{\substack{\Gamma\in\FW_r \\ \tree \text{ tree of } \Gamma}}[\Mbar_{\tree}].
\end{equation}
\end{cor}
\begin{remark}
    By Corollary \ref{cor:InjectivityForMetricTrees}, the above algorithm is ``tree-like'', as depicted in Figure \ref{fig:FireworkPicture}.
\end{remark}

In fact, more is true: all elements of $\FW_r$ lie in \emph{distinct cones} of $\M_{0,n}^{\trop}$:
\begin{prop}\label{prop:Injective}
    The underlying trees of any two elements $\Gamma,\Gamma'\in \FW_r$ are distinct.
\end{prop}

\begin{proof}
    First, we observe the following. Let $\tau'$ be a combinatorial tree with $r+1$ edges. Suppose $\Gamma'\in\FW_{r+1}$ has underlying tree $\tau'$. By definition of $\FW_{r+1}$, $\Gamma'$ is obtained from some $\Gamma\in\FW_r$ by inserting an allowable edge and readjusting edge lengths. The new edge $e_{r+1}$ is, by construction, the first edge along the path from $i_{r+1}$ to $j_{r+1}$ in the convex hull of $S_{r+1}$. In particular, it is entirely determined by the combinatorial tree $\tau'$. In other words, we can tell which edge of $\Gamma'$ is shortest, just from the combinatorics of $\tau'.$

    We conclude that given a tree $\tau'$ with $r+1$ edges, then there is \emph{at most one} tree $\tau$ (obtained by contracting the above edge $e_{r+1}$, if it exists) such that any $\Gamma'\in\FW_{r+1}$ with underlying tree $\tau'$ is obtained by inserting an allowable edge and readjusting edge lengths from an element $\Gamma\in\FW_r$ with underlying tree $\tau.$

    We now proceed by induction on $r$. If $r=0$ the statement is obvious. Suppose the statement holds up to a given $r$. Let $\Gamma_1',\Gamma_2'\in\FW_{r+1}$ be distinct. Let $\Gamma_1,\Gamma_2\in\FW_r$ be obtained from $\Gamma_1'$ and $\Gamma_2'$ respectively, by contracting the shortest edge and readjusting edge lengths.

    \medskip

    \noindent\textbf{Case 1: $\Gamma_1=\Gamma_2$.} Let $\Gamma=\Gamma_1=\Gamma_2$. By Proposition \ref{prop:tropintafteredgeinsertion}\eqref{it:MapToTropIntR+1SurjectiveOnFibers}, $\Gamma_1'$ and $\Gamma_2'$ are obtained from $\Gamma$ by two different allowable edge insertions. It is clear from the definition of allowable edge insertion that $\Gamma_1'$ and $\Gamma_2'$ then have different underlying combinatorial trees.

    \medskip

    \noindent\textbf{Case 2: $\Gamma_1\ne\Gamma_2$.} In this case, by the inductive hypothesis, the underlying combinatorial trees of $\Gamma_1$ and $\Gamma_2$ are distinct. By our observation above, the underlying combinatorial trees of $\Gamma_1'$ and $\Gamma_2'$ must be distinct.
\end{proof}

\begin{cor}\label{cor:InjectivityCorollary}
    In the expression \eqref{eq:FWClassExpansion2} for $[(\bar{X})_0]$, each boundary stratum $\Mbar_{\tau}$ of $\Mbar_{0,n}$ appears with coefficient either 0 or 1.
\end{cor}

\subsection{Example of the firework algorithm}

We now carry out the firework algorithm in an example. Following Notation \ref{not:Sqiqjq}, we set $n=6$ and
\begin{alignat*}{3}
    S_1&=[6], & \quad i_1&=2, & \quad j_1&=4, \\
    S_2&=\{1,3,4,6\}, & \quad i_2&=3, & \quad j_2&=1,\\
    S_3&= \{1,2,4,5,6\}, & \quad i_3&=5, & \quad j_3&=4.
\end{alignat*}
We take $B=10$ for illustrative purposes (it suffices for this example, even though we usually required $B \geq 2n+1$). In the $0$-th step, we define $\FW_0$ to be the set containing the unique stable metric tree $\Gamma_0$ with one vertex and $6$ legs.

In performing the first step, we note that (in the notation of Section \ref{sec:edgeinsertion}) $k_1=1$, $I=\{2\}$, $J = \{4\}$, $K=\{1\}$. There are $7$ metric trees obtainable from $\Gamma_0$ by inserting an allowable edge and readjusting edge lengths, shown below. In each case, $\ell_1$ and $k_1$ are indicated, and the new edge has length $(\ell_1 - k_1)\cdot 100$:

\def\edgelength{0.85}
\begin{center}
    \begin{tikzpicture}[scale=.75]
    \draw (0,0) node {$\bullet$} -- (1.5,0) node {$\bullet$};
    \draw (0,0)--++(150:\edgelength) node[left] {2};
    \draw (0,0)--++(210:\edgelength) node[left] {$\ell_1=3$};
    \draw (1.5,0)--++(55:\edgelength) node[right] {$1 = k_1$};
    \draw (1.5,0)--++(20:\edgelength) node[right] {4};
    \draw (1.5,0)--++(-20:\edgelength) node[right] {5};
    \draw (1.5,0)--++(-55:\edgelength) node[right] {6};
    \draw (.75,0) node[above] {$200$};
    
    \begin{scope}[shift = {(7, 0)}]
    \draw (0,0) node {$\bullet$} -- (1.5,0) node {$\bullet$};
    \draw (0,0)--++(150:\edgelength) node[left] {2};
    \draw (0,0)--++(210:\edgelength) node[left] {$\ell_1=5$};
    \draw (1.5,0)--++(55:\edgelength) node[right] {$1 = k_1$};
    \draw (1.5,0)--++(20:\edgelength) node[right] {3};
    \draw (1.5,0)--++(-20:\edgelength) node[right] {4};
    \draw (1.5,0)--++(-55:\edgelength) node[right] {6};
    \draw (.75,0) node[above] {$400$};
    \end{scope}
    
    \begin{scope}[shift = {(14, 0)}]
    \draw (0,0) node {$\bullet$}-- (1.5,0) node {$\bullet$};
    \draw (0,0)--++(150:\edgelength) node[left] {2};
    \draw (0,0)--++(210:\edgelength) node[left] {$\ell_1=6$};
    \draw (1.5,0)--++(55:\edgelength) node[right] {$1 = k_1$};
    \draw (1.5,0)--++(20:\edgelength) node[right] {3};
    \draw (1.5,0)--++(-20:\edgelength) node[right] {4};
    \draw (1.5,0)--++(-55:\edgelength) node[right] {5};
    \draw (.75,0) node[above] {$500$};
    \end{scope}
    
    \begin{scope}[shift = {(0, -3)}]
    \draw (0,0) node {$\bullet$} -- (1.5,0) node {$\bullet$};
    \draw (0,0)--++(140:\edgelength) node[left] {2};
    \draw (0,0)--++(180:\edgelength) node[left] {$\ell_1=5$};
    \draw (0,0)--++(220:\edgelength) node[left] {6};
    \draw (1.5,0)--++(40:\edgelength) node[right] {$1 = k_1$};
    \draw (1.5,0)--++(0:\edgelength) node[right] {3};
    \draw (1.5,0)--++(-40:\edgelength) node[right] {4};
    \draw (.75,0) node[above] {$400$};
    \end{scope}
    \begin{scope}[shift = {(7, -3)}]
    \draw (0,0) node {$\bullet$} -- (1.5,0) node {$\bullet$};
    \draw (0,0)--++(140:\edgelength) node[left] {2};
    \draw (0,0)--++(180:\edgelength) node[left] {$\ell_1=3$};
    \draw (0,0)--++(220:\edgelength) node[left] {6};
    \draw (1.5,0)--++(40:\edgelength) node[right] {$1 = k_1$};
    \draw (1.5,0)--++(0:\edgelength) node[right] {4};
    \draw (1.5,0)--++(-40:\edgelength) node[right] {5};
    \draw (.75,0) node[above] {$200$};
    \end{scope}
    
    \begin{scope}[shift = {(14, -3)}]
    \draw (0,0) node {$\bullet$} -- (1.5,0) node {$\bullet$};
    \draw (0,0)--++(140:\edgelength) node[left] {2};
    \draw (0,0)--++(180:\edgelength) node[left] {$\ell_1=3$};
    \draw (0,0)--++(220:\edgelength) node[left] {5};
    \draw (1.5,0)--++(40:\edgelength) node[right] {$1 = k_1$};
    \draw (1.5,0)--++(0:\edgelength) node[right] {4};
    \draw (1.5,0)--++(-40:\edgelength) node[right] {6};
    \draw (.75,0) node[above] {$200$};
    \end{scope}

    \begin{scope}[shift = {(0, -6)}]
    \draw (0,0) node {$\bullet$} -- (1.5,0) node {$\bullet$};
    \draw (0,0)--++(125:\edgelength) node[left] {2};
    \draw (0,0)--++(160:\edgelength) node[left] {$\ell_1=3$};
    \draw (0,0)--++(200:\edgelength) node[left] {5};
    \draw (0,0)--++(235:\edgelength) node[left] {6};
    \draw (1.5,0)--++(30:\edgelength) node[right] {$1 = k_1$};
    \draw (1.5,0)--++(-30:\edgelength) node[right] {4};
    \draw (.75,0) node[above] {$200$};
    \end{scope}
    \end{tikzpicture}
\end{center}

We perform the second step on the first metric tree listed above. (More metric trees are produced from the others, which we do not show.) We have $I = \{3\},J=\{1\}$, $k_2=4$, and the algorithm produces the following 2 metric trees $T_1$ and $T_2$. The new edge has length $(\ell_2 - k_2) \cdot 10$ and the other edge length is readjusted:

\begin{center}
    \begin{tikzpicture}[scale=.75]
    \draw[red] (0,0)--(2.25,0);
    \draw (0,0)--++(150:\edgelength) node[left] {2};
    \draw[red] (0,0)--++(210:\edgelength) node[left] {3};
    \draw[red] (2.25,0)--++(55:\edgelength) node[right] {1};
    \draw[red] (2.25,0)--++(20:\edgelength) node[right] {4};
    \draw (2.25,0)--++(-20:\edgelength) node[right] {5};
    \draw[red] (2.25,0)--++(-55:\edgelength) node[right] {6};
    \draw (2.25,0) node {$\bullet$} node[below] {$\underline{v}$};
    \draw (1.125,0) node[above] {$200$};
    \draw (0,0) node {$\bullet$};
    \draw (4,0) node {$\leadsto$};
    
    \begin{scope}[shift = {(7.25, 0)}]
    \draw (-2,0) node {$T_1=$};
    \draw (0,0) node {$\bullet$} -- (1.5, 0) node {$\bullet$} -- (2.5,0) node {$\bullet$};
    \draw (0,0)--++(150:\edgelength) node[left] {2};
    \draw (0,0)--++(210:\edgelength) node[left] {3};
    \draw (2.5,0)--++(40:\edgelength) node[right] {1};
    \draw (2.5,0)--++(0:\edgelength) node[right] {$4 = k_2$};
    \draw (2.5,0)--++(-40:\edgelength) node[right] {5};
    \draw (1.5,0)--++(-90:\edgelength*0.85) node[below,xshift=-.4cm] {$\ell_2=6$};
    \draw (.75,0) node[above] {$180$};
    \draw (2,0) node[above] {$20$};
    \end{scope}
    
    \begin{scope}[shift = {(15.5, 0)}]
    \draw (-2,0) node {$T_2=$};
    \draw (0,0) node {$\bullet$} -- (1.5, 0) node {$\bullet$} -- (2.5,0) node {$\bullet$};
    \draw (0,0)--++(150:\edgelength) node[left] {2};
    \draw (0,0)--++(210:\edgelength) node[left] {3};
    \draw (2.5,0)--++(30:\edgelength) node[right] {1};
    \draw (2.5,0)--++(-30:\edgelength) node[right] {$4 = k_2$};
    \draw (1.5,0)--++(-60:\edgelength) node[below,xshift=.4cm] {$6=\ell_2$};
    \draw (1.5,0)--++(-120:\edgelength) node[below] {5};
    \draw (.75,0) node[above] {$180$};
    \draw (2,0) node[above] {$20$};
    \end{scope}
    \end{tikzpicture}
\end{center}

We perform the $3$rd step with $S_3 = \{1,2,4,5,6\}$, $i_3 = 5$, $j_3 = 4$. For the metric tree $T_1$, we have $I=\{5\}$, $J=\{4\}$, and $k_3=1$:
\begin{center}
    \begin{tikzpicture}[scale=.75] 
    \draw[red] (0,0)--(2.75,0);
    \draw[red] (0,0)--++(150:\edgelength) node[left] {2};
    \draw (0,0)--++(210:\edgelength) node[left] {3};
    \draw[red] (2.75,0)--++(40:\edgelength) node[right] {1};
    \draw[red] (2.75,0)--++(0:\edgelength) node[right] {4};
    \draw[red] (2.75,0)--++(-40:\edgelength) node[right] {5};
    \draw[red] (1.6,0) node {$\bullet$} --++(-90:\edgelength*0.85) node[below] {6};
    \draw (.8,0) node[above] {$180$};
    \draw (2.175,0) node[above] {$20$};
    \draw (2.75,0) node {$\bullet$} node[below] {$\underline{v}$};
    \draw (0,0) node {$\bullet$};
    \draw (1.6,0) node {$\bullet$};
    \draw (4.75,0) node {$\leadsto$};
    
    \begin{scope}[shift = {(8, 0)}]
    \draw (0,0) node {$\bullet$} -- (1.5, 0) node {$\bullet$} -- (2.5, 0) node {$\bullet$} -- (3,0) node {$\bullet$};
    \draw (0,0)--++(150:\edgelength) node[left] {$\ell_3=2$};
    \draw (0,0)--++(210:\edgelength) node[left] {3};
    \draw (3,0)--++(30:\edgelength) node[right] {1};
    \draw (3,0)--++(-30:\edgelength) node[right] {4};
    \draw (2.5,0)--++(-90:\edgelength*0.85) node[below] {5};
    \draw (1.5,0)--++(-90:\edgelength*0.85) node[below] {6};
    \draw (.75,0) node[above] {$180$};
    \draw (2,0) node[above] {$19$};
    \draw (2.75,0) node[above] {$1$};
    \end{scope}
    \end{tikzpicture}
\end{center}

In performing the $3$rd step on $T_2$ above, then $I = \{5\}$, $J=\{1,4\}$ $k_3=2$ and the algorithm produces:

\begin{center}
    \begin{tikzpicture}[scale=.75] 
    \draw[red] (0,0)--(2.75,0);
    \draw[red] (0,0)--++(150:\edgelength) node[left] {2};
    \draw (0,0)--++(210:\edgelength) node[left] {3};
    \draw[red] (2.75,0)--++(30:\edgelength) node[right] {1};
    \draw[red] (2.75,0)--++(-30:\edgelength) node[right] {4};
    \draw[red] (1.5,0)--++(-60:\edgelength) node[below] {6};
    \draw[red] (1.5,0)--++(-120:\edgelength) node[below] {5};
    \draw (.75,0) node[above] {$180$};
    \draw (2.125,0) node[above] {$20$};
    \draw (1.5,0) node {$\bullet$} node[above] {$\underline{v}$};
    \draw (0,0) node {$\bullet$};
    \draw (2.75,0) node {$\bullet$};
    \draw (4.75,0) node {$\leadsto$};
    
    \begin{scope}[shift = {(7, 0)}]
    \draw (0,0) node {$\bullet$} -- (1.5,0) node {$\bullet$} -- (2.75,0) node {$\bullet$};
    \draw (0,0)--++(150:\edgelength) node[left] {2};
    \draw (0,0)--++(210:\edgelength) node[left] {3};
    \draw (2.75,0)--++(30:\edgelength) node[right] {1};
    \draw (2.75,0)--++(-30:\edgelength) node[right] {4};
    \draw (1.5,0)--(1.5,-0.6) node {$\bullet$};
    \draw (1.5,-.6)--++(-60:\edgelength) node[below,xshift=.4cm] {$6=\ell_3$};
    \draw (1.5,-.6)--++(-120:\edgelength) node[below] {5};
    \draw (1.5,-.27) node[right,xshift=-0.05cm] {$4$};
    \draw (.75,0) node[above] {$180$};
    \draw (2.125,0) node[above] {$20$};
    \end{scope}
    \end{tikzpicture}
\end{center}
For the first tree, the vector in Equation \eqref{eq:Min} for each of the three tropical Kapranov maps has the form
\begin{alignat*}{2}
    \kap_{[6], 2 \rel 4}(T) &+ 100 \vec\alpha &&= (d_1^{q=1}, d_3^{q=1}, d_5^{q=1}, d_6^{q=1}) + 100\cdot(1,3,5,6) \\
    &&&= (200, 0, 199, 180) + (100, 300, 500, 600) = (300, 300, 699, 780), \\
    \kap_{\{1, 3, 4, 6\}, 3 \rel 1}(T) &+ 10\vec\alpha &&= (d_4^{q=2}, d_6^{q=2}) + 10\cdot(4,6) \\
    &&&= (20, 0) + (40, 60) = (60, 60), \\
    \kap_{\{1, 2, 4, 5, 6\}, 5 \rel 4}(T) &+ 1 \cdot\vec\alpha &&= (d_1^{q=3}, d_2^{q=3}, d_6^{q=3}) + (1,2,6) \\
    &&&= (1, 0, 0) + (1, 2, 6) = (2, 2, 6).
\end{alignat*}
We see that the minimum is achieved exactly twice in each case. Similarly, for the second tree,
\begin{alignat*}{3}
    \kap_{[6], 2 \rel 4}(T) &+ 100 \vec\alpha
    &&= (d_1^{q=1}, d_3^{q=1}, d_5^{q=1}, d_6^{q=1}) + 100\cdot(1,3,5,6) \\
    &&&= (200, 0, 180, 180) + (100, 300, 500, 600) = (300, 300, 680, 780), \\
    \kap_{\{1, 3, 4, 6\}, 3 \rel 1}(T) &+ 10\vec\alpha
    &&= (d_4^{q=2}, d_6^{q=2}) + 10\cdot(4,6)\\
    &&&= (20, 0) + (40, 60) = (60, 60), \\
    \kap_{\{1, 2, 4, 5, 6\}, 5 \rel 4}(T) &+ 1 \vec\alpha
    &&= (d_1^{q=3}, d_2^{q=3}, d_6^{q=3}) + (1,2,6) \\
    &&&= (24, 4, 0) + (1, 2, 6) = (25, 6, 6).
\end{alignat*}

\bibliography{Tropical_Psi_Pullbacks.bbl}
\bibliographystyle{amsalpha}
\end{document}